\documentclass[final]{amsart}

\usepackage[utf8]{inputenc}
\usepackage[english]{babel}
\usepackage{lmodern}
\usepackage[a4paper, margin=3cm]{geometry}
\usepackage{mathtools}
\usepackage{amsfonts}
\usepackage{amsthm}
\usepackage{amssymb}
\usepackage{listings}
\usepackage{graphicx}
\usepackage{tikz}
\usetikzlibrary{cd}
\usetikzlibrary{babel}
\usepackage{bbold}
\usepackage[colorinlistoftodos,textwidth=2cm,textsize=tiny]{todonotes}
\usepackage[font=small,labelfont=bf]{caption}
\usepackage[subrefformat=parens]{subcaption}
\usepackage{overpic}
\usepackage{ifdraft}

\usepackage{enumitem}
\setlist{leftmargin=6.5mm}
\setenumerate[1]{label={\arabic*.}}

\newtheorem{theorem}{Theorem}[section]
\newtheorem{corollary}[theorem]{Corollary}
\newtheorem{lemma}[theorem]{Lemma}
\newtheorem{proposition}[theorem]{Proposition}

\newtheorem{question}{Question}
\newtheorem*{addendum}{Addendum}

\theoremstyle{definition}
\newtheorem{definition}[theorem]{Definition}
\newtheorem{example}[theorem]{Example}

\theoremstyle{remark}
\newtheorem{remark}[theorem]{Remark}

\graphicspath{ {images/} }
\usepackage{xcolor}
\usepackage[colorlinks,
    linkcolor={red!50!black},
    citecolor={blue!50!black},
    urlcolor={blue!80!black}]{hyperref}

\newcommand{\Z}{\mathbb{Z}}

\newcommand{\R}{\mathbb{R}}
\newcommand{\C}{\mathbb{C}}
\DeclareMathOperator{\lk}{lk}
\DeclareMathOperator{\Ext}{Ext}
\DeclareMathOperator{\Tor}{Tor}
\DeclareMathOperator{\tr}{tr}
\DeclareMathOperator{\ev}{ev}
\newcommand{\sign}{\sigma}
\newcommand{\nul}{\eta}
\DeclareMathOperator{\rank}{rank}
\newcommand{\Tm}{\mathbb{T}^{\mu}_{\ast}}
\DeclareMathOperator{\dsign}{dsign}
\newcommand{\sgn}{\mathrm{sgn}}

\ifoptionfinal{\newcommand{\new}{\textcolor{black}}}{\newcommand{\new}{\textcolor{blue}}}

\title{Torres-type formulas for link signatures}

\author{David Cimasoni}
\address{David Cimasoni -- Section de math\'ematiques, Universit\'e de Gen\`eve, Suisse}
\email{david.cimasoni@unige.ch}

\author{Maciej Markiewicz}
\address{Maciej Markiewicz -- Institute of Mathematics, University of Warsaw, Poland.}
\email{ma.markiewicz3@uw.edu.pl}

\author{Wojciech Politarczyk}
\address{Wojciech Politarczyk -- Institute of Mathematics, University of Warsaw, Poland.}
\email{wpolitarczyk@mimuw.edu.pl}

\date{}

\begin{document}

\makeatletter
   \providecommand\@dotsep{5}
 \makeatother


 
\begin{abstract}
We investigate the limits of the multivariable signature function~$\sigma_L$ of a~$\mu$-component link~$L$ as some variable tends to~$1$ via two different approaches: a three-dimensional and a four-dimensional one. The first uses the definition of~$\sigma_L$ by generalized Seifert surfaces and forms. The second relies on a new extension of~$\sigma_L$ from its usual domain~$(S^1\setminus\{1\})^\mu$ to the full torus~$\mathbb{T}^\mu$ together with
a Torres-type formula for~$\sigma_L$, results which are of independent interest. Among several consequences, we obtain
new estimates on the value of the Levine-Tristram signature of a link close to~$1$.
\end{abstract}

\maketitle

\section{Introduction}
\label{sec:intro}

\subsection{Background on the Levine-Tristram signature}
\label{sub:backgroundLT}

Let~$L$ be an~$m$-component oriented link in the three-sphere~$S^3$,
and let~$A$ be an arbitrary Seifert matrix for~$L$.
The {\em Levine-Tristram signature\/} of~$L$ is the function
\[
\sigma_L\colon S^1\setminus\{1\}\longrightarrow\Z,\quad \omega\longmapsto\sigma(H(\omega))
\]
where
\[
H(\omega)=(1-\omega)A+(1-\overline{\omega})A^T
\]
and~$\sign(H)$ denotes the
signature of the Hermitian matrix~$H$. As one easily checks (see e.g.~\cite{Lic}), this function
does not depend on the choice of the Seifert matrix, and is therefore an invariant of the link~$L$. Similarly, the {\em Levine-Tristram nullity\/}
of~$L$ is the function~$\eta_L\colon S^1\setminus\{1\}\to\Z$ given by~$\eta_L(\omega)=\eta(H(\omega))$, where~$\nul(H)$ stands for the nullity of~$H$.

Since its discovery by Trotter~\cite{Tro} in the case~$\omega=-1$, its study by Murasugi~\cite{Mur}, and its extension by Tristram~\cite{Tri} and Levine~\cite{Lev} to~$S^1\setminus\{1\}$, the Levine-Tristram signature has been the subject of intense activity. Among its numerous remarkable properties, let us mention
the facts that~$\sigma_L$ vanishes if~$L$ is amphicheiral, that it is locally constant on the complement of the roots of the Alexander polynomial~$\Delta_L$, that it provides lower bounds on the unknotting number of~$L$~\cite{Liv}, on its splitting number~\cite{CCZ}, as well as on its Seifert genus, i.e. on the minimal genus of an orientable surface~$S\subset S^3$ with oriented boundary~$\partial S=L$.

More subtly, if~$\omega$ is not the root of any polynomial~$p(t)\in\Z[t,t^{-1}]$ with~$p(1)=\pm 1$, then~$\sigma_L(\omega)$ also provides a lower bound on the {\em topological four-genus\/} of~$L$,
i.e. on the mininal genus of a locally flat orientable surface~$F$ in the four-ball~$B^4$ with oriented boundary~$\partial F= F\cap\partial B^4=L$, see~\cite{NP}. This fact was already noticed by Murasugi using the classical definition of~$\sigma_L$ stated above. However, the current understanding of this phenomenon relies on an alternative interpretation of~$\sigma_L(\omega)$, as the signature of some associated four-dimensional manifold,
an approach pioneered by Rohlin~\cite{Roh} and Viro~\cite{Vir}, see also~\cite{K-T}.
There are several variations on this theme (see e.g.~\cite{Kau} and references therein), but the most practical and now most commonly used one is to consider the intersection form of the four-manifold obtained from~$B^4$ by removing 
a tubular neighborhood of~$F$, with so-called {\em twisted coefficients\/} determined by~$\omega$ (see Section~\ref{sub:twisted} below). Remarkably, this approach is due to Viro once again~\cite{Viro}, some 36 years after his first 
seminal contribution to the subject.
We refer the interested reader to the survey~\cite{Conway-survey} and references therein for more information on the Levine-Tristram signature.

\medskip

Despite all these results, some elementary properties of~$\sigma_L$ remain mysterious.
For example, it is rather frustrating not to have this function naturally extended to the full
circle~$S^1$, as the definition above
yields a trivial signature and ill-defined
nullity at~$\omega=1$.
A related question is the following: what can be said of its value at~$\omega$ close to~$1$?
If~$m=1$, i.e. if the link~$L$ is a knot, then one easily checks that~$\lim_{\omega\to 1}\sigma_L(\omega)$ vanishes,
but in the general case of an~$m$-component link, this elementary approach only yields the inequality~$\vert\lim_{\omega\to 1}\sigma_L(\omega)\vert\le m-1$ (see e.g.~\cite{G-L}).

\medskip

In the recent article~\cite{BZ}, Borodzik and Zarzycki used so-called {\em Hermitian variation structures\/} to show the following result. If~$L=K_1\cup\dots\cup K_m$ is an oriented link
whose Alexander polynomial~$\Delta_L$ does not vanish and is not divisible by~$(t-1)^m$, then 
\[
\lim_{\omega\to 1}\sigma_L(\omega)=\sign(\mathit{Lk}_L)\,,
\]
where~$\mathit{Lk}_L$ denotes the {\em linking matrix\/} of~$L$ defined by
\begin{equation}
    \label{eq:Lk}
(\mathit{Lk}_L)_{ij}=\begin{cases}
\lk(K_i,K_j)&\text{if~$i\neq j$};\\
-\sum_{k\neq i}\lk(K_i,K_k)&\text{if~$i=j$}\,.
\end{cases}
\end{equation}
The assumptions on~$\Delta_L$ are slightly mysterious and the tools rather unorthodox,
but this result puts forward the value~$\sigma_L(1)=\sign(\mathit{Lk}_L)$
as the natural extension of~$\sigma_L$ to
the full circle (a fact that can also
be traced back to the proof of Lemma~5.4 in~\cite{NP}).
Furthermore, this indicates that a naive extension of~$\sigma_L$ to~$S^1$ using the standard
four-dimensional interpretation does not yield the correct answer in general.

\subsection{Results on the Levine-Tristram signature}
\label{sub:resultsLT}

Our first original result on the Levine-Tristram signature is the following inequality (Theorem~\ref{thm:LT}).

\begin{theorem}
\label{thm:intro1}
For any oriented link~$L$, we have
\[
\left|\lim_{\omega\to 1}\sigma_L(\omega)-\sign(\mathit{Lk}_L)\right|\le
\nul(\mathit{Lk}_L)-1-\rank A(L)\,,
\]
where~$A(L)$ denotes the one-variable Alexander module of~$L$.
\end{theorem}

In particular, it implies that~$\lim_{\omega\to 1}\sigma_L(\omega)=\sign(\mathit{Lk}_L)$
for all links with~$\rank A(L)=\nul(\mathit{Lk}_L)-1$.
As shown in Remark~\ref{rem:B-Z}, the equality~$\nul(\mathit{Lk}_L)=1$ is equivalent to the Alexander polynomial~$\Delta_L$ not vanishing and not being divisible by~$(t-1)^m$. Therefore, this theorem
extends the aforementioned result of~\cite{BZ}.
It also implies several immediate and pleasant corollaries, such as the elementary but not so obvious inequalities
\[
\rank A(L)\le \nul(\mathit{Lk}_L)-1
\]
and
\[
\left|\lim_{\omega\to 1}\sigma_L(\omega)\right|\le m-1-\rank A(L)\,,
\]
valid for any oriented link~$L$.

\medskip

As will be explained in Section~\ref{sub:resultsCF},
we have also obtained similar results for more general signatures,
results that can then be applied back to the Levine-Tristram signature. 
To test the power of our methods, we have tried to determine the limit
of the Levine-Tristram signature of
an arbitrary~$2$-component link, showing the following statement (Corollary~\ref{cor:LT-2}).

\begin{corollary}
\label{cor:intro2}
If~$L$ is a 2-component oriented link with linking number~$\ell$
and two-variable Conway function~$\nabla_L$, then its Levine-Tristram signature satisfies
\[
\lim_{\omega\to 1}\sigma_L(\omega)=\begin{cases}
-\sgn(\ell)&\text{if~$\ell\neq 0$, or if~$\nabla_L=0$ (in which case~$\ell=0$);}\\
\sgn(f(1,1))&\text{if~$\ell=0$,~$\nabla_L\neq 0$ and~$f(1,1)\neq 0$};\\
\pm 1\new{\text{ or $0$}}&\text{if~$\ell=0$,~$\nabla_L\neq 0$ and~$f(1,1)= 0$},
\end{cases}
\]
where in the last two cases, we have~$\nabla_L(t_1,t_2)=(t_1-t^{-1}_1)(t_2-t^{-1}_2)f(t_1,t_2)\in\Z[t^{\pm 1}_1,t^{\pm 1}_2]$.
\end{corollary}

Note that the result of~\cite{BZ} covers
precisely the case of non-vanishing linking number, \new{while the last case is classical (see e.g.~\cite{G-L})}; the other cases are new.

Testing our results on~$3$-component links would be an entertaining
exercise that we have not attempted,
but we expect a similar outcome.

\subsection{Background and questions on the multivariable link signature}
\label{sub:backgroundCF}

As is well-known, the Alexander polynomial admits a multivariable extension for links. A slightly less familiar fact is that
the Levine-Tristram also admits such a generalization.
The most natural setting for it is that of colored links, that we now recall.

Let~$\mu$ be a positive integer. A~$\mu${\em-colored link} is an oriented link~$L$ each of whose components is endowed with a {\em color\/} in~$\{1,\dots,\mu\}$ so that all colors are used.
Such a colored link is commonly denoted by~$L=L_1\cup\dots\cup L_\mu$,
with~$L_i$ the sublink of~$L$ consisting of the components of color~$i$.
Two colored links are {\em isotopic\/} if they are related by an ambient isotopy which respects
the orientation and color of all components.
Obviously, a~$1$-colored link is nothing but an oriented link,
while a~$\mu$-component~$\mu$-colored link is an oriented ordered link.
Most of our results hold for arbitrary~$\mu$-colored links, but some of them (e.g. Theorem~\ref{thm:intro6}) are restricted to such ordered links, which we often simply call~{\em$\mu$-component
links}.

Given an arbitrary~$\mu$-colored link~$L$ in~$S^3$,
the {\em multivariable signature\/} of~$L$ is the function
\[
\sigma_L\colon (S^1\setminus\{1\})^\mu\longrightarrow\Z,\quad \omega=(\omega_1,\dots,\omega_\mu)\longmapsto\sigma(H(\omega))\,,
\]
where~$H(\omega)$ is a Hermitian matrix built from
{\em generalized Seifert matrices\/} associated with
generalized Seifert surfaces known as {\em C-complexes\/}, see Section~\ref{sub:C-complex}.
Similarly, the {\em multivariable nullity\/}
of~$L$ is the function~$\eta_L\colon (S^1\setminus\{1\})^\mu\to\Z$ given by~$\eta_L(\omega)=\eta(H(\omega))$.
These invariants were introduced by Cooper~\cite{Cooper} in the~$2$-component~$2$-colored case, and fully developed and studied
in~\cite{C-F}.

As one immediately sees from the definitions, the case~$\mu=1$
recovers the Levine-Tristram signature and nullity, justifying the slight abuse of notation.
However, there is another way
in which these multivariable functions can be applied back to their one-variable
counterparts. Indeed, given any~$\mu$-colored link~$L=L_1\cup\dots\cup L_\mu$,
we have
\begin{equation}
    \label{eq:multi-LT}
    \sigma_{L}(\omega,\dots,\omega)=\sigma_{L^{\text{or}}}(\omega)+\sum_{i<j}\lk(L_i,L_j)
\end{equation}
for all~$\omega\in S^1\setminus\{1\}$, where~$L^{\text{or}}$ denotes the~($1$-colored)
oriented link underlying~$L$ (see~\cite[Proposition~2.5]{C-F}).
As a consequence, this multivariable extension can be a valuable tool even if
one is only interested in the original Levine-Tristram signature.

In a nutshell, all the agreeable properties of the Levine-Tristram signature
mentioned in Section~\ref{sub:backgroundLT} extend to the multivariable setting.
In particular, the function~$\sigma_L$ is constant on the connected components
of the complement in~$(S^1\setminus\{1\})^\mu$ of the zeros of the multivariable
Alexander polynomial~$\Delta_L(t_1,\dots,t_\mu)$, see Theorem~4.1 and Corollary~4.2 of~\cite{C-F}. Also, if~$(\omega_1,\dots,\omega_\mu)$ is not the root of any Laurent polynomial~$p(t_1,\dots,t_\mu)$ with~$p(1,\dots,1)=\pm 1$, then~$\sigma_L(\omega_1,\dots,\omega_\mu)$ and~$\eta_L(\omega_1,\dots,\omega_\mu)$ are invariant under {\em topological concordance\/} of colored links. As in the~$1$-variable case, the understanding of
this fact came in incremental steps (see in particular~\cite[Section~7]{C-F}),
its definitive treatment (and extension to~$0.5$-solvability) being achieved in~\cite{CNT}. Once again, the modern proof relies on an alternative
definition of~$\sigma_L(\omega)$ as the twisted signature of the four-manifold
obtained from~$B^4$ by removing a tubular neighborhood of a union of surfaces~$F=F_1\cup\dots\cup F_\mu$ with~$\partial F_i=F_i\cap\partial B^4=L_i$
for all~$i$.

\medskip

Despite these results, several questions remain unanswered.

\begin{question}
\label{qu:ext}
Is there a natural extension of~$\sigma_L$ and~$\eta_L$
from~$(S^1\setminus\{1\})^\mu$ to the full torus~$\mathbb{T}^\mu$?
\end{question}

As in the~$1$-variable case, the definition via (generalized) Seifert matrices
yields a trivial signature and ill-defined
nullity as soon as some coordinate is equal to~$1$.
Moreover, the `naive' extension of the standard four-dimensional interpretation from~\cite{Viro,DFL,CNT} is in general not well-defined either (see e.g.~\cite[Section~4.4]{DFL}).

\medskip

The second question is relevant to the title of this work.
The celebrated {\em Torres formula\/}~\cite{Tor} relates the multivariable
Alexander polynomial~$\Delta_L$ of a~$\mu$-component ordered link~$L=L_1\cup\dots\cup L_\mu$ to the Alexander polynomial of the~$(\mu-1)$-component link~$L\setminus L_1$
via the equality
\begin{equation}
    \label{eq:Torres}
    \Delta_L(1,t_2,\dots,t_\mu)\;=\;(t_2^{\lk(L_1,L_2)}\cdots t_\mu^{\lk(L_1,L_\mu)}-1)\;\Delta_{L\setminus L_1}(t_2,\dots,t_\mu)
\end{equation}
in~$\Z[t_2,t_2^{-1},\dots,t_\mu,t_\mu^{-1}]$, up to multiplication by units of this ring.
Assuming that a satisfactory answer to Question~\ref{qu:ext} has been found,
is there an anolog of the Torres formula for the multivariable signature and
nullity? In other words:

\begin{question}
    \label{qu:Torres}
    Is there a simple formula relating~$\sigma_L(1,\omega_2,\dots,\omega_\mu)$
    and~$\sigma_{L\setminus L_1}(\omega_2,\dots,\omega_\mu)$, and one relating~$\eta_L(1,\omega_2,\dots,\omega_\mu)$
    and~$\eta_{L\setminus L_1}(\omega_2,\dots,\omega_\mu)$?
\end{question}

\medskip

The third question was already posed in the~$1$-variable context at the end of Section~\ref{sub:backgroundLT}.

\begin{question}
    \label{qu:lim}
    For a fixed~$(\omega_2,\dots,\omega_\mu)$, what can be said of the
    limits~$\lim_{\omega_1\to 1}\sigma_{L}(\omega_1,\omega_2,\dots,\omega_\mu)$~?
\end{question}

Here note the plural in ``limits'': unlike in the~$1$-variable case where the symmetry~$\sigma_L(\overline{\omega})=\sigma_L(\omega)$ ensures that~$\lim_{\omega\to 1}\sigma_L(\omega)$ is well-defined, in the multivariable case
the limit might depend on whether~$\omega_1\in S^1$ tends to~$1$ from above or from below (see e.g. Example~\ref{ex:sign-T}). We shall denote these two limits by~$\omega_1\to 1^+$ and~$\omega_1\to 1^-$.
Note that if one keeps~$\omega'=(\omega_2,\dots,\omega_\mu)\in(S^1\setminus\{1\})^{\mu-1}$ fixed, then these two limits do exist by the locally constant behaviour of signatures described in~\cite[Theorem~4.1]{C-F}. On the other hand, if one allows for any sequence of elements~$\omega\in(S^1\setminus\{1\})^\mu$ converging to~$(1,\omega')$,
then the corresponding limits of signature might not be well-defined (see e.g. Figure~\ref{fig:graph} with~$\omega'$ a third root of unity, and Example~\ref{ex:sign-T}). However,
the estimates that we obtain on what we denote by~$\lim_{\omega\to 1^+}\sigma_L(\omega)$ and~$\lim_{\omega\to 1^-}\sigma_L(\omega)$ hold for any such sequence.

\subsection{Results on the multivariable link signature}
\label{sub:resultsCF}

In short, our work provides rather satisfactory answers to the three questions raised above.

\medskip

First, we extend the signature and nullity to the full torus.
To give a sense that these extensions are ``the right ones'', before giving more ample evidence of this fact below,
we gather in one statement several of their pleasant features.

\begin{theorem}
\label{thm:intro2.5}
Given an arbitrary~$\mu$-colored link~$L$, there exist an extension of the signature~$\sigma_L$ and
of the nullity~$\eta_L$ from~$(S^1\setminus\{1\})^\mu$ to the full torus~$\mathbb{T}^\mu$, which satisfy the following
properties.
\begin{enumerate}
\item The extensions~$\sigma_L\colon\mathbb{T}^\mu\to\Z$ and~$\eta_L\colon\mathbb{T}^\mu\to\Z$ only depend on the isotopy class of the~$\mu$-colored link~$L$ (see Theorem~\ref{thm:extension}).
\item If~$L$ is a ($1$-colored) oriented link, then~$\sigma_L(1)=\sign(\mathit{Lk}_L)$ (see Theorem~\ref{thm:Torres}~(1)).
    \item If~$L=L_1\cup\dots\cup L_\mu$ is a~$\mu$-component link with~$\lk(L_1,L_j)$ not all vanishing, then
 for any~$\omega'\in(S^1\setminus\{1\})^{\mu-1}$ such that~$\Delta_L(1,\omega')\neq 0$, we have
 \[
 \sigma_L(1,\omega')=\frac{1}{2}\left(\lim_{\omega_1\to 1^+}\sigma(\omega_1,\omega')+\lim_{\omega_1\to 1^-}\sigma(\omega_1,\omega')\right)=\sigma_{L\setminus L_1}(\omega')\,.
 \]
 (See Corollary~\ref{cor:limit-of-signature-equality} and Theorem~\ref{thm:Torres}~(3).)
    \item For any~$\omega\in\mathbb{T}^\mu$ (with the possible exception of~$(1,\dots,1)$
if~$\mu\ge 2$), the integer~$\sigma_L(\omega)$ can be obtained
as the signature of a matrix evaluated at~$\omega$ (see Lemma~\ref{lemma:representing-int-forms}).
\end{enumerate}
\end{theorem}

\new{In a nutshell, this extension is defined as follows (see Section~\ref{sub:extension} for details).
We first build the \emph{generalized Seifert surgery} on~$L$,
a closed} three-dimensional manifold~$M_L$ which only depends on the colored link~$L$, and which (in the ordered case) coincides with the manifold defined by Toffoli in~\cite[Construction~4.17]{toffoli}.
\new{The point of this construction is that it admits a natural (though not unique) homomorphism~$\varphi\colon H_1(M_L)\to\Z^\mu$, making~$M_L$ a 
so-called~$\Z^\mu$\emph{-manifold}.
From the pair~$(M_L,\varphi)$, we then define an auxiliary link~$L^\#$ such that~$M_{L^\#}$ is a manifold that bounds over~$\Z^\mu$.
More precisely, we construct a four-dimensional~$\Z^\mu$-manifold~$W_F$ from a union of surfaces~$F\subset B^4$ bounded by~$L^\#$, and show
that~$\partial W_F=M_{L^\#}$ over~$\Z^\mu$.
Finally, the extended signature and nullity of~$L$ are defined} by considering the twisted signature 
and nullity of~$W_F$.
While we provide no further detail in the present introduction, the construction of this \new{extension and the proof of its invariance}
take up a significant portion of this article
and might be considered as its most
technical contribution (see in particular Appendix~\ref{ap:plumbed}).

Some explicit computations yield more evidence that these extensions are very natural indeed.

\begin{example}
Let~$\{L(k)\}_{k\in\Z}$ be the family of~$2$-component links illustrated in Figure~\ref{fig:link-Lk}.
 For~$k\neq 0$ (resp. for~$k=0$) the signature function~$\sigma_{L(k)}\colon(S^1\setminus\{1\})^2\to\Z$ is constant equal to~$1$ (resp.~$0$), while~$\eta_{L(k)}$ is constant equal to~$0$ (resp.~$1$), see Example~\ref{ex:twist}.
 As computed in Examples~\ref{ex:Whitehead} and~\ref{ex:null-L(k)}, the above extensions yield constant functions~$\sigma_L$
 and~$\eta_L$ on~$\mathbb{T}^2\setminus\{(1,1)\}$ in all cases.
\end{example}

\medskip

We now turn to the second question, i.e. to Torres-type formulas for these extended signature and nullity functions. To state these results, it is convenient to make use of the notion of {\em slope\/}, as defined and studied by Degtyarev, Florens and Lecuona in~\cite{DFL}. Without stating the formal definition (see Remark~\ref{rem:slope}), let us recall that given a~$\mu$-colored link~$L=L_1\cup\dots\cup L_\mu=:L_1\cup L'$ with~$L_1=:K$ a knot, the associated slope is a function
assigning a value~$(K/L')(\omega')\in\C\cup\{\infty\}$
to each~$\omega'=(\omega_2,\dots,\omega_\mu)\in(S^1\setminus\{1\})^{\mu-1}$ such that~$\omega_2^{\lk(K,L_2)}\cdots\omega_\mu^{\lk(K,L_\mu)}=1$.
Most importantly for our applications, Theorem~3.2 of~\cite{DFL} asserts that, in generic cases, it can be computed explicitly via the Conway function~$\nabla_L$ of~$L$, see Equation~\eqref{eq:slope}.

We can now state (a particular case of) our Torres formula for the signature
(see Theorem~\ref{thm:Torres}, Remark~\ref{rem:slope} and Remark~\ref{rem:Torres-general} for the full statement).

\begin{theorem}
\label{thm:intro3}
Let~$L=L_1\cup\dots\cup L_\mu=:L_1\cup L'$ be a~$\mu$-colored link with~$\mu\ge 2$ and~$L_1=:K$
a knot. For all~$\omega'\in(S^1\setminus\{1\})^{\mu-1}$, we have
\[
\sigma_L(1,\omega')=
\begin{cases}
\sigma_{L'}(\omega')+\sgn((K/L')(\omega'))&\text{if~$\lk(K,K')=0$ for all~$K'\subset L'$};\\
\sigma_{L'}(\omega')&\text{else},
\end{cases}
\]
where~$\sgn\colon\R\cup\{\infty\}\to\{-1,0,1\}$ denotes the sign function extended via~$\sgn(\infty)=0$.
\end{theorem}

We have also obtained a Torres formula relating~$\eta_L(1,\omega')$
with~$\eta_{L'}(\omega')$, which involves the slope once again.
However, its formulation being rather cumbersome and not particularly illuminating, we refer the reader to Theorem~\ref{thm:torres-formula-nullity} for its statement. 

\medskip

We now turn to the third and last question, namely the estimation of the~$\omega_1\to 1^\pm$ limits of multivariable signatures. Our answer to this question is among the motivations of
the results stated above. In particular, it shows that our extensions
of the signature and nullity functions are sensible ones.

First, and as already mentioned in Theorem~\ref{thm:intro2.5}, these extensions are such that for any given~$\mu$-colored link~$L$,
and for any~$\omega\in\mathbb{T}^\mu$ (with the possible exception of~$(1,\dots,1)$
if~$\mu\ge 2$), the integers~$\sigma_L(\omega)$ and~$\eta_L(\omega)$ can be obtained
as the signature and nullity of a matrix evaluated at~$\omega$ (Lemma~\ref{lemma:representing-int-forms}).
Then, we can use elementary estimates on the difference between the limit
of the signature of a matrix and the signature of a limit (Lemma~\ref{lemma:limit-signature}), together with the aforementioned Torres formulas for the signature and nullity, to obtain the following result (Theorem~\ref{thm:main}).

\begin{theorem}
\label{thm:intro4}
Let \(L=L_{1}\cup L_2\cup\ldots\cup L_{\mu}=:L_1\cup L'\) be a colored link with~$\mu\ge 2$ and~$L_1=:K$ a knot. Consider~\(\omega=(\omega_1,\omega')\in\mathbb{T}^{\mu}\) with~$\omega'\in(S^1\setminus\{1\})^{\mu-1}$.
\begin{enumerate}
    \item 
If there exists a component~$K'\subset L'$ with~$\lk(K,K')\neq 0$, then we have:
\[
\left|\lim_{\omega_1\to 1^\pm}\sigma_L(\omega)-\sigma_{L'}(\omega')\right|\le
\eta_{L'}(\omega')-1+\sum_{K'\subset L'}|\lk(K,K')|-\rank  A(L)\,,
\]
where~$A(L)$ denotes the~$\mu$-variable Alexander module of~$L$.
\item If ~$\lk(K,K')=0$ for all components~$K'\subset L'$, then we have
\[
\left|\lim_{\omega_1\to 1^\pm}\sigma_L(\omega)-\sigma_{L'}(\omega')-s(\omega')\right|\le
\eta_{L'}(\omega')+\varepsilon(\omega')-\rank A(L)\,,
\]
where
\[
s(\omega')=
\begin{cases}
+1&\text{if $(K/L')(\omega')\in(0,\infty)$}\\
-1&\text{if $(K/L')(\omega')\in(-\infty,0)$}\\
0&\text{if $(K/L')(\omega')\in\{0,\infty\}$}
\end{cases}
\quad\text{and}\quad
\varepsilon(\omega')=
\begin{cases}
+1&\text{if $(K/L')(\omega')=0$}\\
-1&\text{if $(K/L')(\omega')=\infty$}\\
0&\text{else.}
\end{cases}
\]
\end{enumerate}
\end{theorem}

As discussed in Sections~\ref{sub:lim-mult} and~\ref{sub:comparison}, this theorem
is quite powerful in the second, so-called {\em algebraically split\/} case.
Indeed, it implies in particular the following result (Corollary~\ref{cor:alg-split-equality}). 
 
\begin{corollary}
\label{cor:intro5}
        Let~\(L=K\cup L'\) be a~$\mu$-colored link as above, such that~$\lk(K,K')=0$ for all~$K'\subset L'$. Then, we have
\[
\lim_{\omega_1\to 1^+}\sigma_L(\omega_1,\omega')=\lim_{\omega_1\to 1^-}\sigma_L(\omega_1,\omega')=\sigma_{L'}(\omega')+\sgn\left(-\frac{\frac{\partial\nabla_L}{\partial t_1}(1,\sqrt{\omega'})}{\nabla_{L'}(\sqrt{\omega'})}\right)
\]
for all~$\omega'\in(S^1\setminus\{1\})^{\mu-1}$ such that~$\nabla_{L'}(\sqrt{\omega'})\neq 0$ and~$\frac{\partial\nabla_L}{\partial t_1}(1,\sqrt{\omega'})\neq 0$.
\end{corollary}

In the non-algebraically split case (case~1 in Theorem~\ref{thm:intro4}), the inequality implies a good upper bound on the difference of the two limits (see Corollary~\ref{cor:diff-lim} and Remark~\ref{cor:diff-lim}). However, 
since it does not distinguish these two (possibly different) limits,
it does not allow
for a good estimation of each of these limits, especially if the linking
numbers are large.

\medskip

To address this issue, we have also attacked this question via a totally different approach, namely coming back to the original definition of the signature and nullity
via C-complexes: this is the subject of Section~\ref{sec:3D}, and of the PhD Thesis of the second author~\cite{Mar}.

To put it briefly, the strategy is the same as the one of the classical proof that the limit $\lim_{\omega\to 1^\pm}\sigma_K(\omega)$ vanishes if~$K$ is a knot:
first conjugate the Hermitian matrix~$H(\omega)$ by a suitable diagonal matrix,
and then estimate the difference between the limit of its signature and the signature of its limit. The result can be phrased as follows,
see Theorem~\ref{thm:limit-of-signature-inequality} and its addendum
for the full statement.

\begin{theorem}
\label{thm:intro6}
        For any~\(\mu\)-component link~\(L=L_1\cup\dots\cup L_\mu=:L_1\cup L'\)
        and any~\(\omega'\in(S^1\setminus\{1\})^{\mu-1}\), we have
    \[
        \left|\lim_{\omega_{1}\rightarrow 1^{\pm}}\sigma_{L}(\omega_{1},\omega')-\sigma_{L'}(\omega')\mp\rho_{\ell}(\omega')\right|\le \eta_{L'}(\omega')+\tau_{\ell}(\omega')-\rank A(L)\,,
        \]
        where~$A(L)$ is the multivariable Alexander module of~$L$, while
         \[
             \tau_\ell(\omega')=\begin{cases}
    1 &\text{if~$\omega_2^{\lk(L_1,L_2)}\cdots\omega_\mu^{\lk(L_1,L_\mu)}=1$;}\cr
    0 &\text{else,}
    \end{cases}
    \]
    and~$\rho_\ell\colon(S^1\setminus\{1\})^{\mu-1}\to\Z$ is an explicit function which only depends on~$\{\lk(L_1,L_j)\}_{j\ge 2}$.
    \end{theorem}

This leads in particular to the following result (Corollary~\ref{cor:lim-equal}).

\begin{corollary}
\label{cor:intro7}
If \(L=L_1\cup\dots\cup L_\mu=:L_1\cup L'\) is a \(\mu\)-component link, then we have
        \[\lim_{\omega_{1}\rightarrow 1^{\pm}}\sigma_{L}(\omega_{1},
        \omega')=\sigma_{L'}(\omega')\pm \rho_{\ell}(\omega')\]
        for all \(\omega'\in(S^1\setminus\{1\})^{\mu-1}\) such that
        $\Delta_{L}(1,\omega')\neq 0$.
\end{corollary}

\medskip

A remarkable fact, discussed in Section~\ref{sub:comparison},
is that the two approaches described above are  complementary.
Indeed, in the algebraically split case, the four-dimensional approach is very powerful
and the three-dimensional one less so. On the other hand, the bigger the linking numbers, the more the 4D approach looses efficiency and the 3D approach gains power.
It is quite amusing to note that in case of total linking number~$\vert\lk(L_1,L_2)\vert+\dots+\vert\lk(L_1,L_\mu)\vert=1$, the two approaches give exactly the same estimate on the
limit of the signature.

\subsection*{Organisation of the article}
Section~\ref{sec:background} deals with the definition of the main objects of interest in this work; in particular, the three and four-dimensional definitions
of the signature and nullity are recalled in Sections~\ref{sub:C-complex} and~\ref{sub:twisted}, together with the Novikov-Wall theorem in Section~\ref{sub:NW}. The first original results appear in Section~\ref{sec:plumbed}, namely technical lemmas on plumbed three-manifolds, whose proofs are provided in Appendix~\ref{ap:plumbed}.

   Section~\ref{sec:3D} contains the results of the three-dimensional approach
   to Question~\ref{qu:lim}, and can be read independently from the rest of the article (apart from Section~\ref{sub:C-complex}). More precisely, Section~\ref{sub:statement} contains the statement of Theorem~\ref{thm:intro6} together with its consequences, including Corollary~\ref{cor:intro7}, while Section~\ref{sub:proof} deals with the proof 
   of this theorem.

In Section~\ref{sec:Torres}, we address Questions~\ref{qu:ext} and~\ref{qu:Torres} above. Indeed, we start in Section~\ref{sub:extension} by constructing the extension
of the signature and nullity functions to the full torus.
Then, in Sections~\ref{sub:Torres-sign} and~\ref{sub:Torres-null}, we prove our Torres-type formulas for these extended signatures and nullity,
in particular Theorem~\ref{thm:intro3}.

Finally, in Section~\ref{sec:4D}, we present the four-dimensional approach to Question~\ref{qu:lim}.
We start in Section~\ref{sub:prelim-lemma} by stating some preliminary lemmas, whose proofs are given in Appendix~\ref{sec:repr-inters-forms}.
The Levine-Tristram signature is studied in Section~\ref{sub:LT}, proving Theorem~\ref{thm:intro1},
and limits of multivariable signatures with all variables tending to~$1$ in Section~\ref{sub:lim-mult-1}. 
More general limits are considered in Section~\ref{sub:lim-mult},
including the proofs of Theorem~\ref{thm:intro6} and of Corollaries~\ref{cor:intro5} and~\ref{cor:intro2}.
Finally, Section~\ref{sub:comparison} contains a discussion of the comparison of the three and four-dimensional approaches.

\subsection*{Acknowledgments}
The authors thank Anthony Conway, \new{Livio Ferretti,} Min Hoon Kim and Gaetan Simian
for helpful discussions, and the anonymous referee for helpful suggestions.
DC acknowledges partial support from the
Swiss NSF grants 200020-20040 and 200021-212085.
MM acknowledges support from University of Warsaw's IDUB program IV.1.2. and partial support from the Swiss NSF grant 200020-20040.
Part of this work
was done while MM was visiting DC at the University of Geneva, whose
hospitality is thankfully acknowledged.

\section{Background and preliminaries}
\label{sec:background}

This section deals with the definition of the main objects of study
together with several preliminary lemmas.
More precisely, we start in Section~\ref{sub:linalg} by recalling
the definition of the signature and nullity of a Hermitian matrix,
and prove an elementary but crucial lemma.
In Section~\ref{sub:C-complex}, we review
the three-dimensional definition of the signature and nullity via C-complexes.
In Section~\ref{sub:twisted}, we then briefly explain the four-dimensional
viewpoint on these invariants,
and recall the Novikov-Wall theorem in Section~\ref{sub:NW}. Finally, Section~\ref{sec:plumbed} contains \new{a review of
plumbed three-manifolds over~$\Z^\mu$, the construction of a closed three-manifold~$M_L$ over~$\Z^\mu$ associated to a~$\mu$-colored link~$L$, as well as a technical lemma on plumbed~$\Z^\mu$-manifolds} whose proof is deferred to Appendix~\ref{ap:plumbed}.

\subsection{Limits of signatures for Hermitian matrices}
\label{sub:linalg}

Recall that a complex-valued square matrix~$H$ is said to be {\em Hermitian\/} if it coincides with its conjugate transpose~$H^*$.

By the spectral theorem, such a matrix can be diagonalized (by a unitary matrix), and the resulting diagonal matrix has real coefficients.
As a consequence, the eigenvalues of~$H$ are real, and one defines
the {\em signature\/} of~$H$ as the integer~$\sign(H)\in\mathbb{Z}$ given by the number of positive eigenvalues of~$H$ minus the number of negative eigenvalues. The {\em nullity\/} of~$H$ is defined as the
non-negative integer~$\nul(H)\in\mathbb{Z}_{\ge 0}$ equal to the number of
vanishing eigenvalues of~$H$.

\medskip

Many of our results are based on the following elementary but crucial lemma, whose easy proof we include for completeness.

	\begin{lemma}\label{lemma:limit-signature}
	    Let \((H(t))_{t\ge 0}\) be a continuous one-parameter family of Hermitian matrices.
	    Then
	    \[\left|\lim_{t \to 0^{+}} \sign(H(t)) - \sign(H(0)) \right| \leq \nul(H(0)) - \lim_{t \to 0^{+}} \nul(H(t))\,.\]
	\end{lemma}
	\begin{proof}
	    By continuity, there exists some~$\epsilon>0$ such that \(\rank(H(t))\) is constant for \(t \in (0,\epsilon)\). As a consequence, both \(\sign(H(t))\) and~$\nul(H(t))$ are constant for \(t \in (0,\epsilon)\). At~$t=0$, precisely~$\nul(H(0)) - \lim_{t \to 0^{+}}\nul(H(t))$ eigenvalues vanish, yielding the expected upper bound on the
	    difference of signatures.
	\end{proof}

\subsection{Signature and nullity via C-complexes}
\label{sub:C-complex}

The aim of this section is to briefly recall the original definition of the signature and nullity of a colored link, following~\cite{Cooper,C-F}.

\begin{figure}[tbp]
    \centering
    \includegraphics[width=4cm]{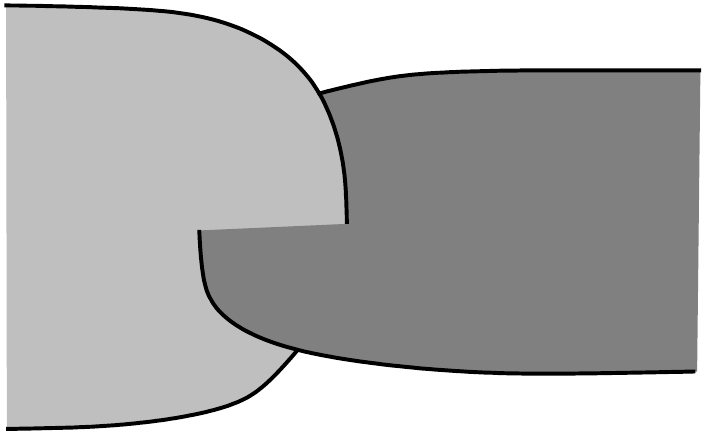}
    \caption{A clasp intersection.}
    \label{fig:clasp}
\end{figure}

\begin{definition}
\label{def:C-cplx}
A {\em C-complex\/} for a~$\mu$-colored link~$L=L_1\cup\dots\cup L_\mu$ is a union~$S=S_1\cup\dots\cup S_\mu$ of surfaces embedded in~$S^3$ such that~$S$ is connected and satisfies the following conditions:
\begin{enumerate}
    \item for all~$i$, the surface~$S_i$ is a connected Seifert surface for~$L_i$;
    \item for all~$i\neq j$, the surfaces~$S_i$ and~$S_j$ are either disjoint or intersect in a finite number of {\em clasps\/}, see Figure~\ref{fig:clasp};
    \item for all~$i,j,k$ pairwise distinct, the intersection~$S_i\cap S_j\cap S_k$ is empty.
\end{enumerate}
Such a C-complex is said to be {\em totally connected} if~$S_i\cap S_j$ is non-empty for all~$i\neq j$.
\end{definition}

The existence of a (totally connected) C-complex for any given colored link is fairly easy to establish, see~\cite{Cim}. On the other hand, the corresponding notion of S-equivalence is more difficult to establish, and the correct version appeared only recently~\cite{DMO}.

These C-complexes, which should be thought of as generalized Seifert surfaces, allow to define {\em generalized Seifert forms\/} as follows.
For any choice of signs~$\varepsilon=(\varepsilon_1,\dots,\varepsilon_\mu)\in\{\pm 1\}^\mu$, let
\[
\alpha^\varepsilon\colon H_1(S)\times H_1(S)\longrightarrow\mathbb{Z}
\]
be the bilinear form given by~$\alpha^\varepsilon(x,y)=\lk(x^\varepsilon,y)$,
where~$x^\varepsilon$ denotes a well-chosen representative of the homology class~$x\in H_1(S)$ pushed-off~$S_i$ in the~$\varepsilon_i$-normal direction (see~\cite{C-F} for a more precise definition). We denote by~$A^\varepsilon$ the corresponding {\em generalized Seifert matrices\/}, defined with respect to a fix basis of~$H_1(S)$.

Consider an element~$\omega=(\omega_1,\dots,\omega_\mu)$ of~$\mathbb{T}^\mu_*\coloneqq(S^1\setminus\{1\})^\mu$, and set
\[
H(\omega)\coloneqq\sum_\varepsilon\prod_{j=1}^\mu(1-\overline{\omega}_j^{\varepsilon_j})A^\varepsilon\,.
\]
Using the identity~$A^{-\varepsilon}=(A^\varepsilon)^T$, one easily checks that~$H(\omega)$ is a Hermitian matrix, and hence admits a well-defined signature~$\sign(H(\omega))\in\mathbb{Z}$ and nullity~$\nul(H(\omega))\in\mathbb{Z}_{\ge 0}$.

\begin{definition}
\label{def:sign-C-complex}
The {\em signature\/} and {\em nullity\/} of the~$\mu$-colored link~$L$ are the functions
\[
\sigma_L,\eta_L\colon \mathbb{T}^\mu_*\longrightarrow \mathbb{Z}
\]
defined by~$\sigma_L(\omega)\coloneqq\sign(H(\omega))$ and~$\eta_L(\omega)\coloneqq\nul(H(\omega))$, respectively.
\end{definition}
The fact that these functions are well-defined invariants, i.e. do not depend on the choice of the C-complex~$S$ for~$L$, relies on the aforementioned notion of S-equivalence~\cite{C-F,DMO}.

Note that for any given colored link, it is not difficult
to find a C-complex and to compute the associated generalized
Seifert matrices: an algorithm has even been recently implemented in~\cite{FKQ}.

\medskip

We now present two (infinite families of) examples that will serve as running examples throughout this article.
 
\begin{figure}[tbp]
 \centering
	    \includegraphics[width=10cm]{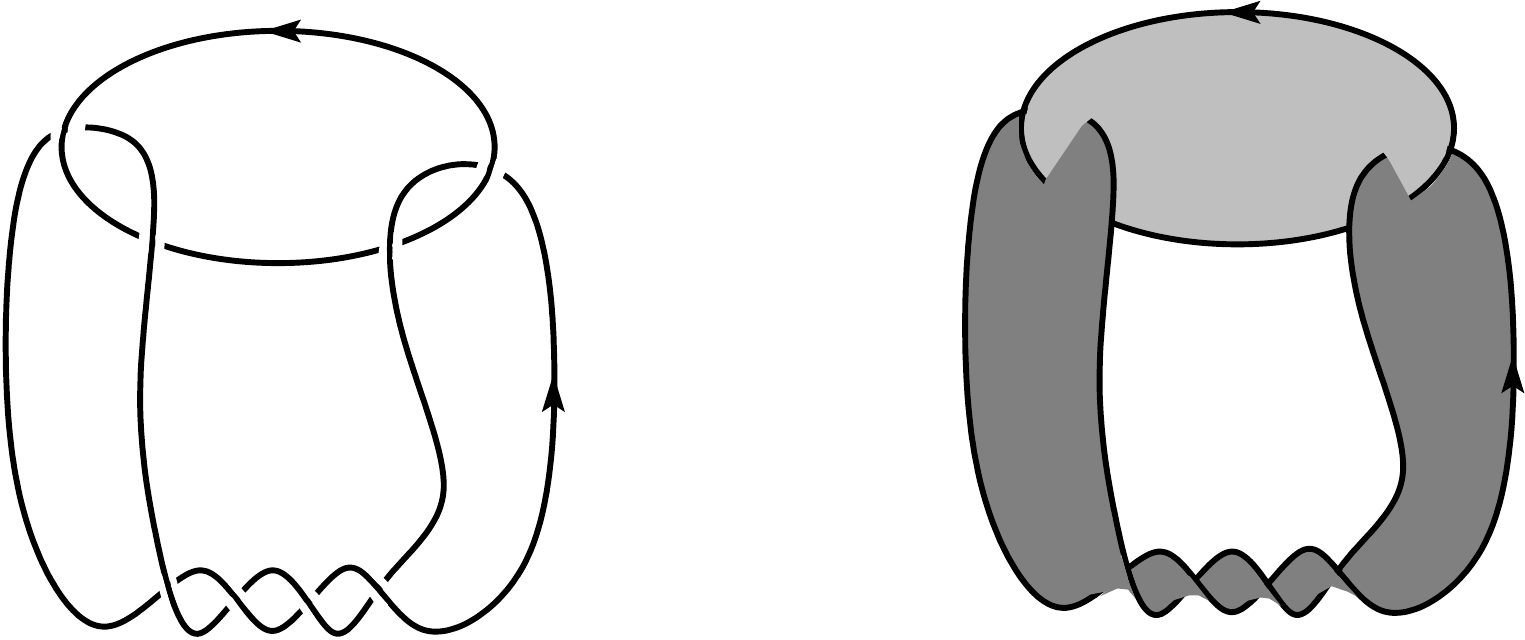}
	    \caption{The link~\(L(k)\), together with an associated C-complex, in the case~$k=2$.}
	    \label{fig:link-Lk}
	\end{figure}

	\begin{example}\label{ex:twist}
	    For any~$k\in\Z$, consider the {\em twist link\/}~\(L(k)\) depicted in the left of Figure~\ref{fig:link-Lk}, where the bottom part consists of~$|k|$ full twists of the same sign as~$k$. For example, the value~$k=0$ yields the trivial link, while~$k=\pm 1$ yields Whitehead links and~$L(2)$ is given in Figure~\ref{fig:link-Lk}.
     
     On the right of this figure, a C-complex
     is given, which has the homotopy type of a circle.
     As one easily checks, the corresponding generalized Seifert matrices are all equal to~$A^\varepsilon=(k)$, leading to the Hermitian matrix
     \[
H(\omega_1,\omega_2) = k(1-\overline{\omega}_1) (1-\overline{\omega}_2) (1-\omega_1)(1-\omega_2) = k\vert 1-\omega_1\vert^2\vert 1-\omega_2\vert^{2}\,,
     \]
and to the constant functions on~$\mathbb{T}^2_*$ given by
\[
\sigma_{L(k)}\equiv\sgn(k)=\begin{cases}
-1&\text{if~$k<0$;}\\
0&\text{if~$k=0$;}\\
1&\text{if~$k>0$,}
\end{cases}
\quad\text{and}\quad
\eta_{L(k)}\equiv\delta_{k0}=\begin{cases}
1&\text{if~$k=0$;}\\
0&\text{else}.
\end{cases}
\]
As a remark that will be used later, note that these generalized Seifert matrices also enable to compute the
Conway function of~$L(k)$ via the main result of~\cite{Cim}.
In these examples, we find
\begin{equation}
    \label{eq:nabla-twist}
\nabla_{L(k)}(t_1,t_2)=k(t_1-t_1^{-1})(t_2-t_2^{-1})\,.
\end{equation}
\end{example}

\begin{example}
\label{ex:torus}
For any~$\ell\in\Z$, let~$T(2,2\ell)$ denote the torus
link depicted in the left of Figure~\ref{fig:torus}.

    First note that for~$\ell=0$, the link~$T(2,2\ell)$ is just the $2$-components unlink whose signature is identically $0$ and whose nullity is identically $1$. Therefore, we can assume that $\ell$ does not vanish.
    In that case, a natural C-complex~$S$ is illustrated in the right of Figure~\ref{fig:torus}.
    The corresponding generalized Seifert matrices, with respect to the natural basis of~$H_1(S)$ given by cycles passing through adjacent clasps, are given by $A^{++}=-\sgn(\ell)T_\ell=(A^{--})^T$, where $T_\ell$ is the $(|\ell|-1)\times(|\ell|-1)$ matrix 
$$T_\ell=\begin{bmatrix}
	        1 & 0 & \ldots & 0 \\
	        1 & 1  & \ldots & 0\\
	        \vdots & \ddots & \ddots & 0\\
	        0 & \ldots & 1 & 1 \\
	        \end{bmatrix},$$
    and $A^{+-}=A^{-+}=0$.
    Without loss of generality, we can now assume that~$\ell$ is positive. Hence,~$\sigma_{T(2,2\ell)}(\omega_1,\omega_2)$ and~$\eta_{T(2,2\ell)}(\omega_1,\omega_2)$ are the signature and nullity of the matrix 
$$H(\omega_1,\omega_2)=(1-\overline{\omega}_1)(1-\overline{\omega}_2)(-T_\ell)+(1-\omega_1)(1-\omega_2)(-T_\ell)^T=\begin{bmatrix}
	        a & b  & \ldots  & 0 \\
	        \bar{b} & a & \ddots & \vdots\\
	        \vdots & \ddots & \ddots & b\\
	        0 & \ldots & \bar{b} & a \\
	        \end{bmatrix}\,,$$
         where~$a=-(1-\overline{\omega}_1)(1-\overline{\omega}_2)(1+\omega_1\omega_2)$ and~$b=-(1-\omega_1)(1-\omega_2)$.

         \begin{figure}[tbp]
        \centering
        \includegraphics[width=10cm]{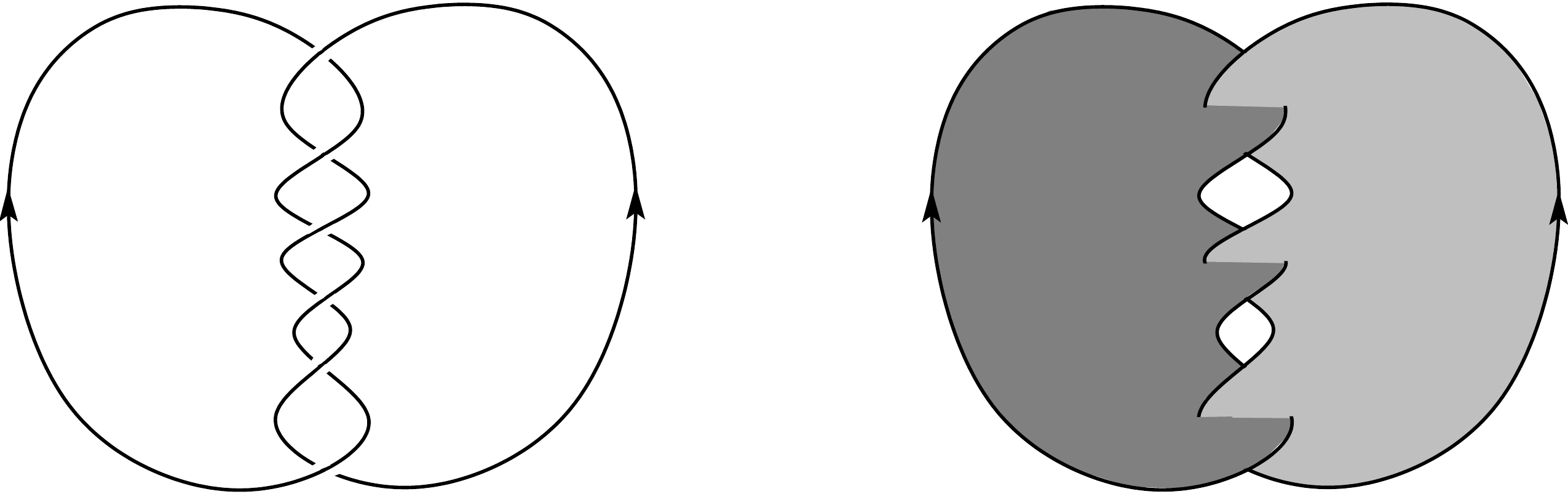}
        \caption{The link~$T(2,2\ell)$ (here with~$\ell=3$)
        together with an associated C-complex.}
        \label{fig:torus}
    \end{figure} 
    
    The eigenvalues of such a matrix are known to be the roots of the second type Chebyshev polynomial~\cite[Theorem 2.2]{KST}, and are given by
\begin{equation*}
a-2|b|\cos{\left(\frac{k\pi}{\ell}\right)}=-(1-\overline{\omega}_1)(1-\overline{\omega}_2)(1+\omega_1\omega_2)-2|1-\omega_1||1-\omega_2|\cos{\left(\frac{k\pi}{\ell}\right)}\,, \quad k=1,\ldots,\ell-1\,.
\end{equation*}
    Writing~$\omega_j=e^{2\pi i\theta_j}$ with~$\theta_j\in (0,1)$ and using the identity~$1-\omega_j=-2i\sin(\pi\theta_j)e^{i\pi\theta_j}$,
    these eigenvalues can be expressed as the positive factor~$8\sin{(\pi\theta_1)}\sin{(\pi\theta_2)}$ multiplied by
\[
\cos{(\pi(\theta_1+\theta_2))}-\cos{\left(\frac{k\pi}{\ell}\right)}\,, \quad k=1,\ldots,\ell-1\,.
\]
    Note that this expression is negative for~$\theta_1+\theta_2\in(\frac{k}{\ell},2-\frac{k}{\ell})$,
 it vanishes for~$\theta_1+\theta_2\in\{\frac{k}{\ell},2-\frac{k}{\ell}\}$,
 and it is positive otherwise.

    This leads to the following formulas, valid for any~$\ell\in\Z$. Writing~$\omega_1=e^{2\pi i \theta_1}$ and~$\omega_2=e^{2\pi i \theta_2}$ with~$\theta_1,\theta_1\in(0,1)$, we have
    \begin{equation}
        \label{eq:torus}
        \sigma_{T(2,2\ell)}(\omega_1,\omega_2)=\sgn(\ell)\cdot f_{|\ell|}(\theta_1+\theta_2)\,,
        \end{equation}
    where~$f_n\colon (0,2)\to\Z$ is determined by~$f_n(2-\theta)=f_n(\theta)$ and
\[
f_n(\theta)=\begin{cases}
                n-2k-1 &\text{if~$\frac{k}{n} < \theta < \frac{k+1}{n}$ with~$k=0,\ldots,n-1$};\\
                n-2k &\text{if~$\theta = \frac{k}{n}$ with~$k =1,\ldots,n-1$};\\
                1-n &\text{if~$\theta = 1$}.
            \end{cases}
\]
Furthermore, the nullity is equal to
\[
\eta_{T(2,2\ell)}(\omega_1,\omega_2)=\begin{cases}
        1 & \text{if~$(\omega_1\omega_2)^\ell=1$ and~$\omega_1\omega_2\neq 1$};\\
        0 & \text{else}\,.
        \end{cases}
\]
The example~$\ell=3$ is illustrated in Figure~\ref{fig:graph}.

Note that these results can also be obtained from the Levine-Tristram signature and nullity of~$T(2,2\ell)$ together with Equation~\eqref{eq:multi-LT} and the fact that~$\sigma_{T(2,2\ell)}$ is locally constant on the complement of the zeros of the Alexander polynomial~$\frac{(t_1t_2)^{\ell}-1}{t_1t_2-1}$.
\end{example}

\begin{figure}[ptb]
    \centering
    \begin{overpic}[width=3.5cm]{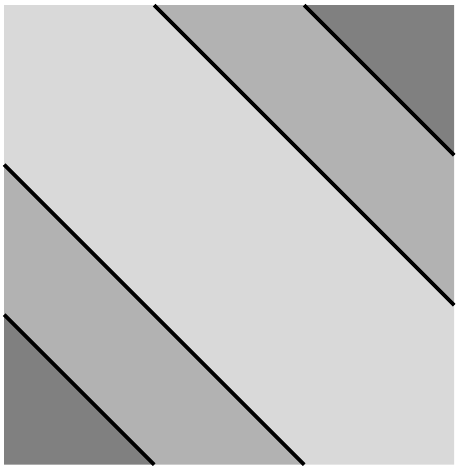}
    \put (8,8){$2$}
    \put (-7,32){$1$}
    \put (23,23){$0$}
    \put (-14,65){$-1$}
    \put (43,46){$-2$}
    \put (102,28){$-1$}
    \put (71,71){$0$}
    \put (102,61){$1$}
    \put (85,85){$2$}
    \end{overpic}
    \caption{The values of~$\sigma_L$ for~$L=T(2,2\ell)$ with~$\ell=3$, on the open torus~$\mathbb{T}^2_*\simeq(0,1)^2$. The function~$\eta_L$ is equal to~$1$ on the diagonals, and vanishes everywhere else.}
    \label{fig:graph}
\end{figure}
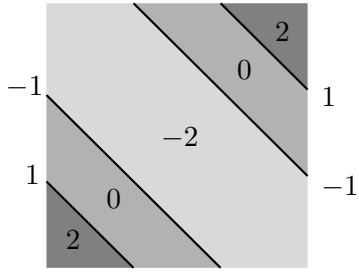

\medskip

We will make use of the following result, which is a direct consequence of~\cite[Corollary~4.2]{C-F}.

\begin{lemma}
    \label{lem:eta0}
Let~$L$ be a~$\mu$-colored link. If~$\omega\in\mathbb{T}^{\mu-1}_*$ is such that~$\Delta_L(\omega)\neq 0$, then~$\eta_L(\omega)$ vanishes.
\end{lemma}

\medskip

Note that when a variable~$\omega_j$ is equal to~$1$, then the full matrix~$H(\omega)$ vanishes, leading to a vanishing signature and ill-defined nullity. One of our achievements will be to provide a natural extension of these functions to the full torus, see Section~\ref{sub:extension} below. This uses an alternative point of view on the signature and nullity, that we now review.

\subsection{Signature and nullity via twisted intersection forms}
\label{sub:twisted}

We now briefly recall the four-dimensional viewpoint on the signature and nullity, following~\cite{CNT} and referring to Section~\ref{sec:algebr-preliminaries} for details.

\medskip

We first need to make a small detour into homological algebra.
Le~$X$ be a connected CW-complex endowed with a homomorphism~$\pi_1(X)\to\mathbb{Z}^\mu=\mathbb{Z}t_1\oplus\dots\mathbb{Z}t_\mu$ for some~$\mu\ge 1$.
Then, any~$\omega=(\omega_1,\dots,\omega_\mu)\in\mathbb{T}^\mu$ induces a group homomorphism~$\pi_1(X)\to\mathbb{C}^*$ by mapping~$t_i$ to~$\omega_i$. This in turn extends to a
ring homomorphism~$\phi_\omega\colon\mathbb{Z}[\pi_1(X)]\to\mathbb{C}$
such that~$\phi_\omega(g^{-1})=\overline{\phi_\omega(g)}$ for all~$g\in\pi_1(X)$, thus endowing the field~$\mathbb{C}$ with a structure of right-module over the group ring~$\mathbb{Z}[\pi_1(X)]$; we denote this module by~$\mathbb{C}^\omega$. The cellular
chain complex~$C(\widetilde{X})$ of the universal cover~$\widetilde{X}$ of~$X$ being a left module over this same ring, one can consider the complex vector spaces
\[
H_*(X;\mathbb{C}^\omega)\coloneqq H_*\left(\mathbb{C}^\omega\otimes_{\mathbb{Z}[\pi_1(X)]}C(\widetilde{X})\right)\,.
\]
This is one example of a construction known as the {\em homology of~$X$ with twisted coefficients\/}, see Section~\ref{sec:algebr-preliminaries}.

Coming back to low-dimensional topology, let us consider a compact oriented~$4$-manifold~$W$ endowed with a homomorphism~$\pi_1(W)\to\mathbb{Z}^\mu$.
As explained in Section~\ref{sec:algebr-preliminaries}, one can define a {\em twisted intersection pairing\/}
\[
Q_\omega\colon H_2(W;\mathbb{C}^\omega)\times H_2(W;\mathbb{C}^\omega)\longrightarrow\mathbb{C}
\]
that is Hermitian,
so one can consider the associated signature and nullity
\[
\sign_\omega(W)\coloneqq\sign(Q_\omega),\quad\nul_\omega(W)\coloneqq\nul(Q_\omega)\,.
\]

We are finally ready to come back to knots and links.
Let~$L=L_1\cup\dots\cup L_\mu$ be a colored link in~$S^3$.
A {\em bounding surface\/} for~$L$ is a union~$F=F_1\cup\dots\cup F_\mu$ of properly embedded, locally flat, compact,
connected oriented surfaces~$F_i\subset B^4$ which only intersect
each other transversally (in double points), and such that the oriented boundary~$\partial F_i$ is equal to~$L_i$.
These surfaces being locally flat, they admit tubular neighborhoods whose union we denote by~$\nu(F)$. Also, let us write~$V_F$ for the exterior~$B^4\setminus\nu(F)$ of~$F$ in~$B^4$, which intersects~$S^3=\partial B^4$ in the exterior~$X_L\coloneqq S^3\setminus\nu(L)$ of~$L$ in~$S^3$.

As one easily shows, the abelian group~$H_1(V_F;\mathbb{Z})$ is freely generated by the meridians of the surfaces~$F_1,\dots,F_\mu$,
and the inclusion induced homomorphism~$H_1(X_L;\mathbb{Z})\to H_1(V_F;\mathbb{Z})$ is an isomorphism.
As a consequence, we can apply the above technology to these spaces,
yielding in particular a complex vector space~$H_1(X_L;\mathbb{C}^\omega)$
and a~$\mathbb{C}^\omega$-twisted intersection pairing~$Q_\omega$
on~$H_2(V_F;\mathbb{C}^\omega)$ for any~$\omega\in\mathbb{T}^\mu$.

The following result is due to~\cite{CNT}, see also~\cite{C-F,CFT}.
It provides the promised four-dimensional viewpoint on the signature and nullity of a colored link.

\begin{proposition}[\cite{CNT}]
\label{prop:twisted}
For any~$\omega\in\mathbb{T}^\mu_*$ and any bounding surface~$F$ for~$L$,
we have
\[
\sigma_L(\omega)=\sign_\omega(V_F)\quad\text{and}\quad\eta_L(\omega)=\dim H_1(X_L;\mathbb{C}^\omega)\,.
\]
\end{proposition}

It is this point of view on the signature and nullity that we will use
in Sections~\ref{sec:Torres} and~\ref{sec:4D}.

\subsection{The Novikov-Wall theorem}
\label{sub:NW}

The goal of this section is to recall as briefly as possible
the statement of the {\em Novikov-Wall theorem\/}, which plays an important role in this work.

\medskip

Let~$Y$ be an oriented compact~4-manifold
and let~$X_0$ be an oriented compact~3-manifold embedded into~$Y$ so that~$\partial X_0=X_0\cap\partial Y$. Assume that~$X_0$ splits~$Y$ into two manifolds~$Y_-$
and~$Y_+$, with~$Y_-$ such that the induced orientation on its boundary restricted to~$X_0\subset\partial Y_-$
coincides with the given orientation of~$X_0$.
For~$\varepsilon=\pm$, denote by~$X_\varepsilon$ the compact~3-manifold~$\partial Y_\varepsilon\setminus\mathrm{Int}(X_0)$, and
orient it so that~$\partial Y_-=(-X_-)\cup X_0$ and~$\partial Y_+=(-X_0)\cup X_+$. Note that the orientations of~$X_0$,~$X_-$ and~$X_+$ induce the same orientation on the
surface~$\Sigma\coloneqq\partial X_0=\partial X_-=\partial X_+$, as illustrated in Figure~\ref{fig:NW}.

\begin{figure}[tbp]
    \centering
    \begin{overpic}[width=3.5cm]{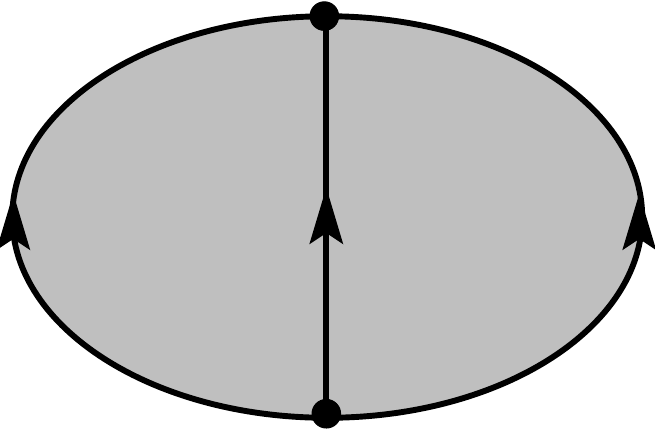}
    \put (-18,30){$X_-$}
    \put (55,25){$X_0$}
    \put (103,30){$X_+$}
    \put (23,30){$Y_-$}
    \put (75,30){$Y_+$}
    \put (45,68){$\Sigma$}
    \end{overpic}
    \caption{The setting of the Novikov-Wall theorem.}
    \label{fig:NW}
\end{figure}
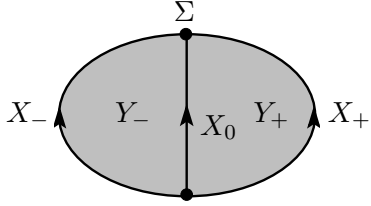

Assume further that~$Y$ is endowed with a homomorphism~$\psi\colon\pi_1(Y)\to\Z^\mu$ for some~$\mu\ge 1$.
As described in Section~\ref{sub:twisted}, any~$\omega\in\mathbb{T}^\mu$ then induces twisted coefficients~$\mathbb{C}^\omega$ on the homology of~$Y$.
Precomposing~$\psi$ with inclusion induced homomorphisms,
we also obtain twisted coefficients on the homology
of submanifolds of~$Y$, coefficients that we also denote by~$\mathbb{C}^\omega$.
Note that the twisted intersection form on~$H\coloneqq H_1(\Sigma;\mathbb{C}^\omega)$ is skew-Hermitian; we denote it by~$(a,b)\mapsto a\cdot b$.
Using Poincar\'e-Lefschetz duality, one checks that for any~$\varepsilon\in\{-,0,+\}$, the
kernel~$\mathcal{L}_\varepsilon$
of the inclusion-induced map~$H\to H_1(X_\varepsilon;\mathbb{C}^\omega)$ is a Lagrangian subspace of~$(H,\,\cdot\,)$.

Given three Lagrangian subspaces~$\mathcal{L}_-,\mathcal{L}_0,\mathcal{L}_+$ of a finite-dimensional complex vector space~$H$
endowed with a skew-Hermitian form~$(a,b)\mapsto a\cdot b$, the associated {\em Maslov index} is the integer
\[
\mathit{Maslov}(\mathcal{L}_-,\mathcal{L}_0,\mathcal{L}_+)=\sign(f)\,,
\]
where~$f$ is the Hermitian form on~$(\mathcal{L}_-+\mathcal{L}_0)\cap\mathcal{L}_+$ defined as follows. 
Given~$a,b\in(\mathcal{L}_-+\mathcal{L}_0)\cap\mathcal{L}_+$,
write~$a=a_-+a_0$ with~$a_-\in\mathcal{L}_-$ and~$a_0\in\mathcal{L}_0$ and set~$f(a,b)\coloneqq a_0\cdot b$.

\begin{theorem}[\cite{Wal}]
\label{thm:NW}
In the setting above and for any~$\omega\in\mathbb{T}^\mu$, we have
\[
\sign_\omega(Y)=\sign_\omega(Y_-)+\sign_\omega(Y_+)+\mathit{Maslov}(\mathcal{L}_-,\mathcal{L}_0,\mathcal{L}_+)\,.
\]
\end{theorem}

\begin{remark}
\label{rem:NW}
\begin{enumerate}
    \item This result was originally stated and proved by Wall~\cite{Wal} in the untwisted case, but the proof easily extends.
    \item The version above follows the convention of~\cite[Chapter~IV.3]{Tur}, which yields
a Maslov index equal to the opposite of the one appearing in~\cite{Wal}. This discrepancy is compensated
by a minus sign in the non-additivity theorem of~\cite{Wal}
which does not appear in Theorem~\ref{thm:NW}.
    \item\label{rem:NW:3} We have implicitly been using the ``outward vector first" convention for the induced orientation on the boundary of a manifold: this is necessary to obtain coincidence between the three and four-dimensional versions of the signatures (Proposition~\ref{prop:twisted}).
Another tacit convention is that the oriented meridian~$m_K$ of an oriented knot~$K$ should satisfy~$\lk(K,m_K)=1$; note that the sign of this linking number, and therefore the orientation of the meridian, depends on the orientation of the ambient 3-manifold.
Finally, the longitude~$\ell_K$ of the oriented~$K$ should obviously define the same generator of~$H_1(\nu(K))$ as~$K$. 
Assembling together these conventions,
we obtain that the orientation of~$\partial\nu(K)$ induced by the
orientation of~$X_K$ is such that the intersection form on~$\partial\nu(K)$ satisfied~$m_K\cdot\ell_K=-1$.
\item\label{rem:NW:4} \new{Note that the annihilator of~$f$ contains the space~$(\mathcal{L}_-\cap\mathcal{L}_+)+(\mathcal{L}_0\cap\mathcal{L}_+)$. As a consequence, the Maslov index vanishes as soon as this later space coincides with~$(\mathcal{L}_-+\mathcal{L}_0)\cap\mathcal{L}_+$.}
\end{enumerate}
\end{remark}

\subsection{Plumbed three-manifolds \texorpdfstring{\new{over~$\Z^\mu$}}{over Z^mu}}
\label{sec:plumbed}

The aim of this section is to \new{recall the definition of plumbed~$3$-manifolds, and to use them to define a closed~$3$-manifold~$M_L$
associated to an arbitrary~$\mu$-colored link~$L$, mildly extending~\cite[Construction~4.17]{toffoli}.
We also show that~$M_L$ is naturally (though not uniquely)
equipped with a homomorphism~$\varphi\colon \pi_1(M_L)\to\Z^\mu$,
thus defining an element in the bordism group~$\Omega_3(\Z^\mu)$.
Finally, we state a technical lemma about plumbed
3-manifolds, whose proof can be found in Appendix~\ref{ap:plumbed}.}

\medskip

We start by recalling the definition of these manifolds, following and slightly extending the presentation of~\cite[Section~4.2]{CNT}.

Let~$\Gamma=(V,E)$ be a finite unoriented graph. Following
the classical textbook~\cite{Serre}, we write~$E$ for the set of {\em oriented\/} edges, and~$s,t\colon E\to V$ for the source and target maps, respectively.
The graph is unoriented in the sense that the set~$E$ is endowed with an involution~$e\mapsto\overline{e}$ such that~$\overline{e}\neq e$ and~$s(\overline{e})=t(e)$ for all~$e\in E$.
We call such a graph~$\Gamma$ a {\em plumbing graph\/} if it has no loop (i.e. no edge~$e\in E$ such that~$s(e)=t(e)$) and if it is endowed
with the following decorations.
\begin{enumerate}
    \item[$\bullet$] Each vertex~$v\in V$ is
decorated by a compact, oriented, possibly disconnected surface~$F_v$, possibly with boundary.
    \item[$\bullet$] Each edge~$e\in E$ is labeled by a sign~$\varepsilon(e)=\pm 1$ such that~$\varepsilon(\overline{e})=\varepsilon(e)$, and comes with the specification of a connected component of~$F_{s(e)}$.
\end{enumerate}
Such a plumbing graph~$\Gamma$ determines an oriented 3-dimensional manifold~$P(\Gamma)$ via the following construction.
For each oriented edge~$e\in E$, we choose an embedded open disk~$D_e$
in the corresponding connected component of~$F_{s(e)}$ so that the
disks~$\{D_e\}_{e\in E}$ are disjoint. For each~$v\in V$, we then set
\[
F_v^\circ\coloneqq F_v\setminus\bigsqcup_{s(e)=v}D_e\,.
\]
The associated {\em plumbed $3$-manifold\/} is defined as
\[
P(\Gamma)\coloneqq\Big(\bigsqcup_{v\in V} F_v^\circ\times S^1\Big)/\sim\,,
\]
where each pair of edges~$e,\overline{e}\in E$ yields the following identification of~$F^\circ_{s(e)}\times S^1$ and~$F^\circ_{s(\overline{e})}\times S^1$ along one of their boundary components:
\begin{align}\label{eq:plumbing}
\begin{split}
    (-\partial D_e)\times S^1&\longrightarrow(-\partial D_{\overline{e}})\times S^1\\
    (x,y)&\longmapsto (y^{\new{-}\varepsilon(e)},x^{\new{-}\varepsilon(e)})\,.
    \end{split}
\end{align}
Note that since these homeomorphisms reverse the orientation, the resulting~$3$-manifold~$P(\Gamma)$ is endowed with an orientation which extends the orientation of each~$F^\circ_v\times S^1$. Note also
that the boundary of~$P(\Gamma)$ consists of one torus for each
boundary component of~$\bigsqcup_{v\in V}F_v$.

\medskip

	\new{
    We now come to the construction of the closed~3-manifold~$M_L$. Given a~$\mu$-colored link~$L$,
	consider the plumbing graph~$\Gamma_L$ defined as follows:
	\begin{itemize}
	\item The vertex set of~$\Gamma_L$ is given by the colors~$V=\{1,\dots,\mu\}$, the vertex~$i$ being decorated
	with the surface~$\bigsqcup_{K\subset L_i} D_K$ consisting of disjoint oriented closed discs indexed by the components of color~$i$.
	\item 
	Given two components~$K,K'$ of different colors, the corresponding discs are linked by~$\vert\lk(K,K')\vert$ edges, and every such edge~$e$ is decorated with the sign~$\varepsilon(e)=\sgn(\lk(K,K'))$.
	\end{itemize}
We will write~$P(L)$ for the associated plumbed manifold~$P(\Gamma_L)$. Note that the orientation reversing automorphisms of~$D_K\times S^1$ given by~$(x,y)\mapsto(x,y^{-1})$ define an orientation reversing homeomorphism~$P(L)\to P(\overline{L})$, where~$\overline{L}$ stands for the mirror image of~$L$, and thus an orientation preserving homeomorphism
\begin{equation}\label{eq:P-of-the-mirror}
    P(\overline{L})\simeq -P(L)\,.
\end{equation}}

\new{Note also that this oriented compact~$3$-manifold has boundary~$\partial P(L)=\bigsqcup_{K\subset L} \partial D_K\times S^1$.
Therefore, it is possible to glue~$P(L)$
and~$X_L=S^3\setminus\nu(L)$ along their homeomorphic boundaries, and we do so in the following way. For each component~$K\subset L$, recall that a {\em meridian\/} is an oriented simple closed curve~$m_K\subset \partial\nu(K)$ whose class vanishes in~$H_1(\nu(K))$ and satisfies~$\lk(m_K,K)=1$.
A {\em Seifert longitude\/} is an oriented simple closed curve~$\ell_K\subset \partial\nu(K)$ such that~$[\ell_K]=[K]\in H_1(\nu(K))$ and
\begin{equation}
\label{eq:Seifert}
\lk(\ell_K,L_i)\coloneqq\sum_{K'\subset L_i}\lk(\ell_K,K')=0\,.
\end{equation}
In other words, this is the longitude obtained by the intersection of~$\partial\nu(K)$ with a Seifert surface for the
sublink~$L_i$ of color~$i$, hence the terminology. 
Let us glue~$X_L$ and~$P(L)$ along their boundary via the homeomorphism~$\partial D_K\times S^1\simeq \partial\nu(K)$ obtained by mapping~$\ast_K\times S^1$ (for some~$\ast_K\in\partial D_K$) to a meridian~$m_K$, and~$\partial D_K\times\ast$ (for some~$\ast\in S^1$) to a Seifert longitude~$\ell_K$. The orientations on~$X_L$ and on~$P(L)$ can be seen to induce the same orientation on
the boundary tori (see Remark~\ref{rem:NW}.3). Therefore, we reverse the orientation of~$P(L)$ and define
\[
M_L\coloneqq X_L\cup_\partial -P(L)\,,
\]
which is an oriented closed~3-manifold.
We call it the {\em generalized Seifert surgery\/} on~$L$.}

\new{\begin{example}
\label{exs:ML}
\begin{enumerate}
\item\label{exs:ML-1} If~$L$ is an oriented link (interpreted as a~1-colored link), then~$M_L$ is the so-called {\em Seifert sugery\/}
on~$L$, as defined in~\cite[Definition~5.1]{NP}, hence the terminology. In such a case, and unless~$L$ is a knot, this manifold differs
from~\cite[Construction~4.17]{toffoli}.
\item\label{exs:ML-2} If~$L$ is the~$2$-colored Hopf link, then~$P_L$ is homeomorphic to~$X_L\simeq S^1\times S^1\times [0,1]$, and~$M_L$ is homeomorphic to the~3-dimensional torus.
\item\label{exs:ML-3} If~$L=L_1\cup\dots\cup L_\mu$ is a colored link with all linking numbers
vanishing, then~$M_L$ is the~$0$-surgery on~$L$. (This includes the case of knots.)
For example, if~$L$ is the 3-colored Borromean rings, then~$M_L$ is the~3-dimensional torus once again.
\end{enumerate}
\end{example}}

\new{The main point of this construction is that~$M_L$ is naturally (though not uniquely) a~{\em $\Z^\mu$-manifold\/}, i.e. is endowed with a homomorphism~$\pi_1(M_L)\to\Z^\mu$. More precisely, we have the following result.}

\new{\begin{lemma}
\label{lem:ML}
The homomorphism~$\varphi_X\colon H_1(X_L)\to\Z^\mu$ defined by~$\varphi_X([\gamma])=\left(\lk(\gamma,L_i)\right)_i$ extends to~$\varphi\colon H_1(M_L)\to\Z^\mu$ such that~$\varphi([\ast_i\times S^1])=t_i\in\Z^\mu$ for any~$\ast_i\in D_K$ with~$K\subset L_i$.
\end{lemma}
\begin{proof}
For any component~$K\subset L_i$,
let $\varphi_K\colon H_1(D_K^\circ\times S^1)\to\Z^\mu$ be defined by
\[
\varphi_K([\ast_i\times S^1])=t_i\,,\quad\varphi_K([\partial D_{e}\times\ast])=\varepsilon(e) \, t_j
\]
for all~$\ast_i\in D_K$ with~$K\subset L_i$ and~$\ast\in S^1$,
and for any edge~$e$ with~$s(e)$ the vertex~$K$ and~$t(e)$ a vertex~$K'\subset L_j$. Since this is consistent with the gluing~\eqref{eq:plumbing}, Mayer-Vietoris arguments show that the homomorphisms~$\varphi_K$ can be extended to a well-defined (though in general not unique) homomorphism~$\varphi_P\colon H_1(P(L))\to\Z^\mu$. For each~$\ast_i$ and~$\ast$ as above,
this map satisfies
\[
\varphi_P([\ast_i\times S^1])=t_i=\varphi_X([m_K])\,.
\]
Moreover, since the sign~$\varepsilon(e)$ of an edge~$e$ as above is equal to~$\sgn(\lk(K,K'))$, we get
\begin{align*}
\varphi_P([\partial D_K\times \ast])&=\sum_{e\sim K}\varphi_K([\partial D_{e}\times \ast])=\sum_{e\sim K}\varepsilon(e)\,t_j=
\sum_{j\neq i}\sum_{K'\subset L_j}\varepsilon(e)\vert\lk(K,K')\vert \,t_j\\
&=\sum_{j\neq i}\lk(K,L_j) \,t_j \stackrel{\eqref{eq:Seifert}}{=}\lk(\ell_K,L_i) \,t_i+\sum_{j\neq i}\lk(K,L_j)\, t_j=\sum_{j}\lk(\ell_K,L_j)\, t_j\\
&=\varphi_X([\ell_K])\,.
\end{align*}
Since this is consistent with the gluing~$M_L=X_L\cup_\partial -P(L)$, a Mayer-Vietoris argument concludes the proof.
\end{proof}}

\new{\begin{remark}
\begin{enumerate}
    \item\label{rems:M_L:1} If~$\Gamma$ is a plumbing graph with vertices~$F_1,\dots,F_\mu$,
    we call {\em meridional\/} any homomorphism~$\varphi_P\colon H_1(P(\Gamma))\to\Z^\mu$
    with~$\varphi_P([\ast_i\times S^1])=t_i$ for all~$\ast_i\in F_i$.
    By the arguments from the beginning of the proof of Lemma~\ref{lem:ML}, a plumbed manifold~$P(\Gamma)$ always admits a meridional homomorphism.
    It is unique when the graph~$\Gamma_L$ is a forest, but not unique in general.
    \item\label{rems:M_L:2} Similarly, we also call {\em meridional\/}
    a homomorphism~$\varphi\colon H_1(M_L)\to\Z^\mu$ as in Lemma~\ref{lem:ML},
    i.e. a homomorphism which sends meridians to the appropriate generator of~$\Z^\mu$.
    By Lemma~\ref{lem:ML}, the manifold~$M_L$ always admits a meridional homomorphism.
    It is unique when all linking numbers between components of different colors vanish, but it is not unique in general.
\end{enumerate}
\label{rems:M_L}
\end{remark}}

\medskip

Suppose that~\((M_{1},f_1)\) and~\((M_2,f_2)\) are \(3\)-dimensional, compact, oriented, connected \(\Z^{\mu}\)-manifolds. \new{They are said to be} \emph{\(\Z^{\mu}\)-bordant} if there exists a~\(4\)-dimensional compact oriented \(\Z^{\mu}\)-manifold \((W,f)\) such that \(\partial W = M_{1} \sqcup (-M_{2})\) and \(f_i\colon\pi_{1}(M_{i}) \to \Z^{\mu}\)
  factors through \(f\colon\pi_{1}(W)\to\Z^{\mu}\) for~$i=1,2$.
  \new{The set of corresponding equivalence classes forms an abelian group for the connected sum, usually
  denoted by~$\Omega_3(B\Z^\mu)$ or simply by~$\Omega_3(\Z^\mu)$.
  Of interest to us is the following well-known fact: there is an isomorphism
\begin{align}
\label{eq:bordism}
\begin{split}
    \Omega_3(\Z^\mu)&\longrightarrow H_3(\mathbb{T}^\mu;\Z)=\Z^{{\mu}\choose{3}}\\
    (M,f)&\longmapsto f_*([M])\,,
    \end{split}
\end{align}
where~$[M]\in H_3(M;\Z)$ is the fundamental class of~$M$,
and~$f_*\colon H_3(M;\Z)\to H_3(\mathbb{T}^\mu;\Z)$ is the map induced by
the homotopy class~$M\to B\Z^\mu=\mathbb{T}^\mu$ corresponding to~$f\colon\pi_1(M)\to\Z^\mu$. This follows from the Atiyah-Hirzebruch spectral sequence and the fact that the bordism group~$\Omega_3(\mathrm{pt})$ vanishes, see
e.g.~\cite[Section~3]{DNOP}.}

\new{
\begin{example}
\begin{enumerate}
        \item\label{exs:meridional:1} As we saw above, for~$L$ the~$2$-colored Hopf link, the manifold~$M_L$ is the~$3$-dimensional torus. Moreover, it admits a unique meridional homomorphism, namely the projection~$\varphi\colon H_1(\mathbb{T}^3)\to\Z^2$
        defined by the coloring and orientation of its components.
        The~$\Z^2$-manifold~$(M_L,\varphi)$ bounds the~4-manifold~$S^1\times S^1\times D^2$ over~$\Z^2$.
        \item\label{exs:meridional:2} For~$L$ the~$3$-colored Borromean rings, the manifold~$M_L$ is once again the~$3$-dimensional torus. It admits a unique meridional homomorphism, namely the isomorphism~$\varphi\colon H_1(\mathbb{T}^3)\simeq\Z^3$
        induced by the coloring and orientation of its components.
        By~\eqref{eq:bordism}, the~$\Z^3$-manifold~$(M_L,\varphi)$ generates the bordism group~$\Omega_3(\Z^3)\simeq\Z$.
        In particular, it does {\em not} bound any~$4$-manifold over~$\Z^3$.
    \end{enumerate}
    \label{exs:meridional}
\end{example}
}

  \medskip
  
We will need the following generalization of~\cite[Lemma 4.9]{CNT}.
Following their terminology, we call a plumbing graph
{\em balanced\/} if for any pair of vertices~$v,w\in V$, we have~$\sum_{e=(v,w)}\varepsilon(e)=0$, where the sum is over the set of edges~$e\in E$ with~$s(e)=v$ and~$t(e)=w$.

\new{\begin{lemma}
  \label{lem:Y}
Let~$G$ be a balanced plumbing graph with vertices given by closed oriented surfaces, and let~$\varphi_P\colon H_1(P(G))\to\Z^\mu$ be a meridional homomorphism such
that~$(P(G),\varphi_P)$ bounds over~$\Z^\mu$.
      Then, it bounds a compact connected oriented~\(\Z^{\mu}\)-manifold~\((Y,f)\) such that~\(\pi_{1}(Y) = \Z^{\mu}\),~\(f\) is an isomorphism and~$\sign_\omega(Y)=0$ for
      all~\(\omega \in \mathbb{T}^{\mu}\).
\end{lemma}}

\new{Its proof can be found in Appendix~\ref{ap:plumbed}, where it appears as Corollary~\ref{cor:surjective-ap}, together with several additional technical lemmas on plumbed 3-manifolds.}

\section{Limits of signatures: the 3D approach}
\label{sec:3D}

The purpose of this section is to evaluate the limits of multivariable signatures using their definition via C-complexes described in Section~\ref{sub:C-complex}. More precisely, we start in Section~\ref{sub:statement} by the statement of the results, together
with examples and consequences. The proof of the main
theorem is given in Section~\ref{sub:proof}.
    
\subsection{Main result and consequences}
\label{sub:statement}

Throughout this section, we assume for simplicity that the colored link~$L=L_1\cup L_2\cup\dots\cup L_\mu=:L_1\cup L'$ is a~$\mu$-component link, i.e. that each sublink~$L_i$ is a knot. Note however that we expect our methods to extend to the case of an arbitrary colored link,
see in particular Remark~\ref{rem:3d-LT}.

We shall adopt the notation~$\ell_j:=\lk(L_1,L_j)$ together with~$s_j:=\sgn(\ell_j)$ for~$2\le j\le\mu$, and~$\vert\ell\vert:=\vert\ell_2\vert+\dots+\vert\ell_\mu\vert$.
Also, we write~$\rho\colon\mathbb{T}^2\to\{-1,0,1\}$ for the symmetric function defined by
\begin{equation}
\label{eq:def-rho}
\rho(z_1,z_2):=
    \sgn\left[i(z_1z_2-1)(\overline{z}_1-1)(\overline{z}_2-1)\right]
\end{equation}
for~$z_1,z_2\in S^{1}$,
whose graph is sketched in Figure~\ref{fig:graph-s}.
Note that~$i(z_1z_2-1)(\overline{z}_1-1)(\overline{z}_2-1)$ is real for all~$z_1,z_2\in S^1$, so
its sign~$\rho(z_1,z_2)\in\{-1,0,1\}$ is well-defined.
Moreover, it satisfies the identity~$\rho(z_1,z_2)=-\rho(\overline{z}_1,\overline{z}_2)$ for all~$(z_1,z_2)\in\mathbb{T}^2$.
This extends to a function~$\rho\colon\mathbb{T}^n\to\Z$ via
\begin{equation}
    \label{eq:def-rho-mult}
\rho(z_1,\dots,z_n):=\sum_{j=1}^{n-1} \rho(z_{j},z_{j+1}\cdots z_{n})\,.
\end{equation}

We are now ready to state the main result of this section.

\begin{figure}[tbp]
    \centering
    \begin{overpic}[width=2.5cm]{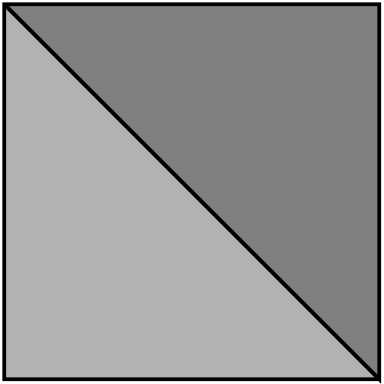}
    \put (65,70){$-1$}
    \put (28,28){$1$}
    \put (-10,90){$0$}
    \end{overpic}
    \caption{The values of~$\rho$ on~$\mathbb{T}^2$ represented as a square with opposite sides identified.}
    \label{fig:graph-s}
\end{figure}
   
\begin{theorem}
\label{thm:limit-of-signature-inequality}
    For a~\(\mu\)-component link~\(L=L_1\cup\dots\cup L_\mu=:L_1\cup L'\) and all~\(\omega'\in\mathbb{T}_*^{\mu-1}\), we have
    \[
        \left|\lim_{\omega_{1}\rightarrow 1^{\pm}}\sigma_{L}(\omega_{1},\omega')-\sigma_{L'}(\omega')\mp\rho_{\ell}(\omega')\right|\le \eta_{L'}(\omega')+\tau_{\ell}(\omega')-\rank A(L)\,,
        \]
         where~$A(L)$ denotes the multivariable Alexander module of~$L$, while~$\rho_\ell$ and~$\tau_\ell$ are given by
         \begin{equation}
         \label{eq:tau-rho}
            \rho_\ell(\omega')=
            \begin{cases}   
        \rho(\underbrace{\omega^{s_2}_2,\dots,\omega^{s_2}_2}_{\vert\ell_2\vert},\dots,\underbrace{\omega^{s_\mu}_\mu,\dots,\omega^{s_\mu}_\mu}_{\vert\ell_\mu\vert})&\text{if~$\vert\ell\vert>0$;}\cr0&\text{else,}
        \end{cases}
        \qquad
        \tau_\ell(\omega')=\begin{cases}
    1 &\text{if~$\omega_2^{\ell_2}\cdots\omega_\mu^{\ell_\mu}=1$;}\cr
    0 &\text{else,}
    \end{cases}
                  \end{equation}
                  for~$\omega'=(\omega_2,\dots,\omega_\mu)\in\mathbb{T}^{\mu-1}_*$.
    \end{theorem}

The function~$\rho_\ell$ can be presented via the closed formula~\eqref{eq:tau-rho}, but it also
admits the following elementary geometric description.

\begin{addendum}
Given any~$\ell=(\ell_2,\dots,\ell_\mu)\in\Z^{\mu-1}\setminus\{0\}$, let us denote by~$\Sigma_\ell$ the hypersurface
\[
\Sigma_\ell:=\{\omega'\in \mathbb{T}^{\mu-1}_*\,|\,\tau_\ell(\omega')=1\}=\{(\omega_2,\dots,\omega_\mu)\in \mathbb{T}^{\mu-1}_*\,|\,\omega_2^{\ell_2}\cdots\omega_\mu^{\ell_\mu}=1\}\,,
\]
which consists of~$\vert\ell\vert-1$ parallel hyperplans. 
Then, the function~$\rho_\ell\colon\mathbb{T}^{\mu-1}_*\to\Z$ is uniquely determined by the following properties:
\begin{enumerate}
    \item it is constant on the connected components of~$\mathbb{T}^{\mu-1}_*\setminus\Sigma_\ell$ and of~$\Sigma_\ell$;
    \item it takes its maximum value~$\vert\ell\vert-1$ when~$\omega_j\to 1^{s_j}$ for all~$j$ such that~$\ell_j\neq 0$;
    \item moving away from the component of~$\mathbb{T}^{\mu-1}_*\setminus\Sigma_\ell$ described by the second point above, it jumps by~$-1$ when entering a component of~$\Sigma_\ell$ and by~$-1$ when exiting it, eventually reaching its minimal value~$1-\vert\ell\vert$ on the component where~$\omega_j\to 1^{-s_j}$ for all~$j$ such that~$\ell_j\neq 0$. 
\end{enumerate}
\end{addendum}

We defer the proof of Theorem~\ref{thm:limit-of-signature-inequality} and of its addendum to Section~\ref{sub:proof}, and now explore some consequences and examples.

\begin{example}
	\label{ex:correction-2}
 Let us compute the function~$\rho_\ell$ in the case~$\mu=2$, with linking number~$\ell$ of sign~$s$,
 using the formula~\eqref{eq:tau-rho}.
 If~$\ell=0$, then~$\rho_\ell$ is identically zero. For~$\ell\neq 0$, we have
 \[
 \rho_\ell(\omega)=\rho(\overbrace{\omega^s,\dots,\omega^s}^{\vert\ell\vert})=\sum_{j=1}^{\vert\ell\vert-1}\rho(\omega^s,\omega^{s(\vert\ell\vert-j)})=s\cdot\sum_{j=1}^{\vert\ell\vert-1}\rho(\omega,\omega^j)\,,
 \]
where these (empty) sums are understood as vanishing if~$\vert\ell\vert=1$.
 Note that~$\rho(\omega,\omega^j)$ is determined by the following properties: it vanishes at all~$\omega\in S^1\setminus\{1\}$ such that~$\omega^{j}=1$ or~$\omega^{j+1}=1$, is equal to~$1$ for~$\omega\to 1^+$, and
 alternates sign at each zero.
 Writing~$\omega=\exp\left(2 \pi i \theta \right)$ with~$\theta\in(0,1)$, 
 this easily leads to
	\begin{equation}
 \label{eq:rho}
	    	 s\cdot \rho_\ell(\omega)
  = \begin{cases}
|\ell|-(2k+1) & \text{if~$\frac{k}{|\ell|} < \theta < \frac{k+1}{|\ell|}$ with~$k=0,1,\ldots,|\ell|-1$;} \\
\vert\ell\vert-2k & \text{if~$\theta = \frac{k}{|\ell|}$ with~$k = 1,2,\ldots,|\ell|-1$.}
	            \end{cases}
		\end{equation}
As expected, this coincides with the description given in
the addendum, where~$\Sigma_\ell$ consists of the~$\ell^\text{th}$-roots of unity in~$S^1\setminus\{1\}$.
The graph of~$\rho_\ell$ is illustrated in Figure~\ref{fig:rho} (in the case~$\ell=5$).
\end{example}

\begin{example}
 \label{ex:rho-3} 
 Let us now describe the function~$\rho_\ell$ in the case~$\mu=3$, this time using the addendum,
 and assuming for definiteness that the linking numbers~$\ell_2,\ell_3$ are non-negative.
By definition, the hyperplane~$\Sigma_\ell$ is given by
the restriction to~$\mathbb{T}^2_*\simeq (0,1)^2$ of
a torus link of type~$T(\ell_2,\ell_3)$, i.e.~$\vert\ell\vert-1$ parallel lines of slope~$-\ell_2/\ell_3$ dividing~$(0,1)^2$ into~$\vert\ell\vert$ connected components. In the bottom-left corner,~$\rho_\ell$ takes
the value~$\vert\ell\vert-1$, then~$\vert\ell\vert-2$
on the adjacent component of~$\Sigma_\ell$, then~$\vert\ell\vert-3$ on the next slab, and so on, until it reaches the value~$1-\vert\ell\vert$ at the top-right corner.
The examples~$\ell=(2,2)$ and~$\ell=(2,3)$ are drawn in Figure~\ref{fig:rho}.
\end{example}

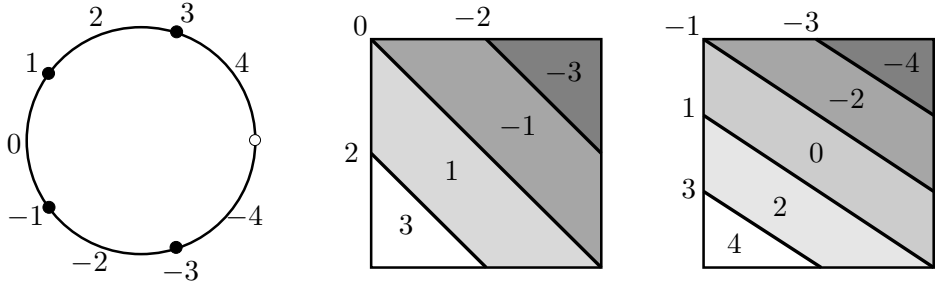
\begin{figure}[tbp]
        \centering
    \begin{overpic}[width=12cm]{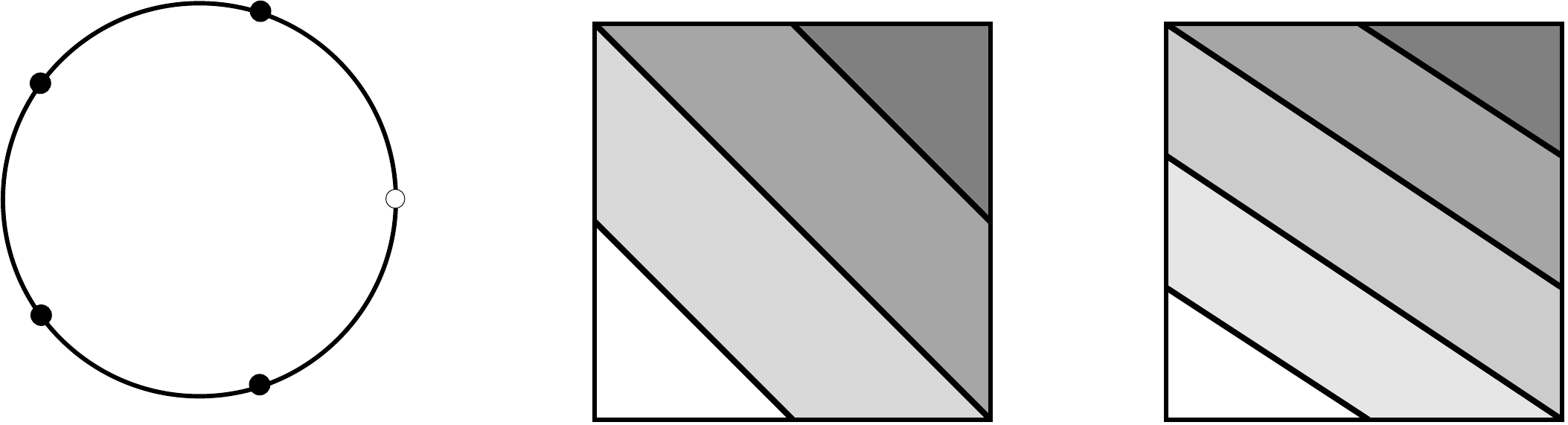}
    \put (23,22){$4$}
    \put (17,27.5){$3$}
    \put (7,27){$2$}
    \put (0,22){$1$}
    \put (-2,13){$0$}
    \put (-2,5){$-1$}
    \put (5,0){$-2$}
    \put (15,-1){$-3$}
    \put (22,5){$-4$}
    \put (41,4){$3$}
    \put (35,12){$2$}
    \put (46,10){$1$}
    \put (36,26){$0$}
    \put (52,15){$-1$}
    \put (47,27){$-2$}
    \put (57,21){$-3$}
    \put (77,2){$4$}
    \put (72,8){$3$}
    \put (82,6){$2$}
    \put (72,17){$1$}
    \put (86,12){$0$}
    \put (70,26){$-1$}
    \put (88,18){$-2$}
    \put (83,26.5){$-3$}
    \put (94,22){$-4$}
    \end{overpic}
        \caption{The graph of~$\rho_\ell$ for~$\ell=5$, for~$\ell=(2,2)$ and for~$\ell=(2,3)$.}
        \label{fig:rho}
    \end{figure}

Theorem~\ref{thm:limit-of-signature-inequality} determines the values of the limits~$\lim_{\omega_{1}\rightarrow 1^{\pm}}\sigma_{L}(\omega)$ in the following case.

 \begin{corollary}\label{cor:limit-of-signature-equality}
        If \(L=L_1\cup\dots\cup L_\mu=:L_1\cup L'\) is a \(\mu\)-component link, then we have
        \[\lim_{\omega_{1}\rightarrow 1^{\pm}}\sigma_{L}(\omega_{1},
        \omega')=\sigma_{L'}(\omega')\pm \rho_{\ell}(\omega')\]
        for all \(\omega'\in\mathbb{T}_*^{\mu-1}\) such that
        $\Delta_{L}(1,\omega')\neq 0$.
    \end{corollary}
    \begin{proof}
    By the Torres formula~\eqref{eq:Torres} together with Equation~\eqref{eq:tau-rho} for~$\tau_\ell$
    and Lemma~\ref{lem:eta0}, the assumption~$\Delta_{L}(1,\omega')\neq 0$ ensures that~$\tau_\ell(\omega')$ vanishes as well as~$\eta_{L'}(\omega')$. The result now follows from Theorem~\ref{thm:limit-of-signature-inequality}.
    \end{proof}

\begin{example}
\label{ex:sign-T}
Consider the torus link~$L=T(2,2\ell)$.
If~$\ell=0$, then~$L$ is the unlink and Theorem~\ref{thm:limit-of-signature-inequality} implies
the obvious result, namely~$\lim_{\omega_1 \to 1^{\pm}} \sigma_L(\omega_1,\omega_2) = 0$.
For~$\ell\neq 0$, Corollary~\ref{cor:limit-of-signature-equality}
yields
\[
\lim_{\omega_1 \to 1^{\pm}} \sigma_L(\omega_1,\omega_2) = \pm\rho_{\ell}(\omega_2)
\]
for all~$\omega_2\in S^1$ such that~$\omega_2^\ell\neq 1$.
For these exceptional values of~$\omega_2$, the inequality of Theorem~\ref{thm:limit-of-signature-inequality},
which reads
\[
\left| \lim_{\omega_1 \to 1^{\pm}} \sigma_L(\omega_1,\omega_2) \mp \rho_{\ell}(\omega_2) \right| \leq 1\,,
\]
is sharp,
but does not determine the value of the limits (compare~\eqref{eq:rho} and~\eqref{eq:torus}).
This is for a good reason, since these limits are actually not well-defined:
if~$\omega_2^\ell=1$, then~$\lim_{\omega\to(1^{\pm},\omega_2)}\sigma_L(\omega)$
depends on the way~$\omega$ converges to~$(1^{\pm},\omega_2)$. On these examples, Theorem~\ref{thm:limit-of-signature-inequality} is therefore optimal: it determines the limits when they exist, and gives a sharp estimate on their possible values when they are not well-defined.
\end{example}

	\begin{example}
 \label{ex:nonequality-signature-limit}
	    Consider the link~\(L(k)\) depicted in Figure~\ref{fig:link-Lk}.
	    The components of \(L(k)\) being unknotted and unlinked (i.e.~$\ell=0$), Theorem~\ref{thm:limit-of-signature-inequality} simply reads
	    \[\left| \lim_{\omega_1 \to 1^{\pm}} \sigma_{L(k)}(\omega_1,\omega_2)\right| \leq 1-\rank A(L)\,.\]
	 Recall from Example~\ref{ex:twist} that the signature of~$L(k)$ is constant equal to~$\sgn(k)$, while the nullity is constant equal to~$\delta_{k0}$. Hence, we see that the inequality above is sharp on this family of examples.
  However, we also see that Theorem~\ref{thm:limit-of-signature-inequality} does not determine the limit of the signature unless~$k=0$.
  In particular, it fails to determine this limit in the cases~$k=\pm 1$ of the Whitehead links.
  
  As we shall see in Example~\ref{ex:Wh}, the results of Section~\ref{sec:4D} do determine these limits.  
	\end{example}

We conclude this section with a short discussion of
further consequences, restricting our attention
to the~$2$-component case for simplicity.
These results will be extended to an
arbitrary number of components in Section~\ref{sub:lim-mult-1} using different methods.
 
	\begin{corollary}
	\label{cor:LT}
	    Suppose that \(L = L_1 \cup L_2\) is a two-component link such that \(\ell = \lk(L_1,L_2) \neq 0\). Then, for any~$\epsilon_1,\epsilon_2=\pm$,
	    the limit of~$\sigma_L(\omega_1,\omega_2)$ as~$\omega_1$ tends to~$1^{\epsilon_1}$ and ~$\omega_2$ to~$1^{\epsilon_2}$ exists and is given by
	    \begin{equation}
     \label{eq:cor-2}
	        	       \lim_{\omega_1 \to 1^{\epsilon_1},\omega_2 \to 1^{\epsilon_2}} \sigma_L(\omega_1,\omega_2) = \epsilon_1\epsilon_2\,(\ell-\sgn(\ell))\,.
	    	    \end{equation}
	    If~$\ell=0$, then the inequality
	    \[
	    |\sigma_L(\omega_1,\omega_2)|\le 1-\rank A(L)
	    \]
	    holds for all~$(\omega_1,\omega_2)$ in some neighborhood of~$(1,1)$ in~$\mathbb{T}^2_{\ast}$.
	    In particular, if~$\Delta_L$ vanishes, then the four limits exist and are equal to zero.
	\end{corollary}
	\begin{proof}
	    First, observe that the assumption \(\ell \neq 0\) guarantees that \(\Delta_L(1,1) \neq 0\). Hence, there exists a neighborhood \((1,1) \in U \subset S^1 \times S^1\) such that the signature function \(\sigma_L(\omega_1,\omega_2)\) is constant on each connected component of \(U \cap \mathbb{T}^2_{\ast} = U_1 \sqcup U_2 \sqcup U_3 \sqcup U_4\). 
	    These connected components correspond to the four possible limits
	    of~$\sigma_{L}(\omega_1,\omega_2)$, whose existence is now established.
	    Using Corollary~\ref{cor:limit-of-signature-equality} and the fact that the Levine-Tristram signature of a knot vanishes near~$\omega=1$,
	    we get
	    \[\lim_{\omega_1 \to 1^{\epsilon_1},\omega_2 \to 1^{\epsilon_2}} \sigma_L(\omega_1,\omega_2) = \lim_{\omega_1 \to 1^{\epsilon_1}} \left(\sigma_{L_1}(\omega_1)+
     \epsilon_2\rho_{\ell}(\omega_1)\right) = \lim_{\omega_1 \to 1^{\epsilon_1}}
     \epsilon_2\rho_{\ell}(\omega_1)\,.
	    \]
	    The result now follows from the explicit value of the correction term given in Equation~\eqref{eq:rho}.
	    
	    If~$\ell=0$, then a similar argument leads to the following fact: there exists a neighborhood~$U$ of~$(1,1)$ in~$\mathbb{T}^2_{\ast}$ such that the signature function satisfies
	    \[
	    |\sigma_L(\omega_1,\omega_2)|\le 1-\rank A(L)
	    \]
	    for all~$(\omega_1,\omega_2)\in U$. If the Alexander polynomial vanishes,
	    then the right-hand side of this inequality vanishes as well, leading to
	    the desired statement.
	\end{proof}
	
	\begin{remark}
 \label{rem:LT-2}
 \begin{enumerate}
     \item 
	In particular, using~\eqref{eq:multi-LT}, we obtain the fact that for a $2$-component link $L$ with non-vanishing linking number, or vanishing linking number and Alexander polynomial, the Levine-Tristram signature satisfies~$\lim_{\omega \to 1}\sigma_L(\omega)=-\sgn(\ell)$.
	For non-vanishing linking numbers, this coincides with the 2-component case of~\cite[Theorem 1.1]{BZ}. (See Remark~\ref{rem:B-Z} for a proof that the hypothesis are equivalent.).
	\item 
	The example of the twist links~$L(k)$ with~$k\neq 0$ given in Example~\ref{ex:twist}, whose signatures are constant equal to~$\sgn(k)$, shows that the equality~\eqref{eq:cor-2} does not hold in general when~$\ell=0$ and~$\Delta_L\neq 0$.
      \end{enumerate}
	\end{remark}

\subsection{Proof of Theorem~\ref{thm:limit-of-signature-inequality} and of the addendum}
\label{sub:proof}

    We wish to study the limits~$\omega_1\to 1^\pm$ of the signature~$\sigma_L$ of an arbitrary ordered link~\(L=L_{1}\cup\ldots\cup L_{\mu}=:L_1\cup L'\)
    for a fixed value of~$\omega'=(\omega_2,\dots,\omega_\mu)\in\mathbb{T}^{\mu-1}_*$, assuming~$\mu\ge 2$.

    \medskip
    
    Consider an associated C-complex~\(S=S_{1}\cup\ldots\cup S_{\mu}\).
    Without loss of generality (e.g. via the second move in~\cite[Lemma~2.2]{C-F}), it may be assumed that~$S':=S\setminus S_1$ is connected.
    Let~\(\mathcal{A}\)
    be a set of curves in~\(S'\) representing a basis of~\(H_{1}(S')\),
    and let~\(\mathcal{B}\) be a set of curves in~$S$ such that the classes of the elements of~\(\mathcal{A}\cup\mathcal{B}\) form a basis of~\(H_{1}(S)\). With respect to this basis, the Hermitian matrix
    \[
H(\omega_{1},\omega')=\sum_{\varepsilon\in\{\pm 1\}^\mu}\prod_{j=1}^\mu(1-\overline{\omega}_j^{\varepsilon_j})A^\varepsilon
\]
can be presented in a block form that we denote by
    \(H(\omega_{1},\omega')=\begin{bmatrix}
    C & D\\
    E & F\\
    \end{bmatrix}\).
   Observe that since 
    the curves~$x\in\mathcal{A}$ are disjoint from~$S_1$, the linking numbers~\(\lk(x^{\varepsilon},-)\)
    do not depend on~\(\varepsilon_{1}\). Therefore, the coefficients of the matrices~\(C,D\) and~\(E\) are multiples of~\((1-\omega_{1})(1-\overline{\omega}_{1})=\vert1-\omega_1\vert^2\)
    by polynomial functions of~$\omega_2,\dots,\omega_\mu$.
    Note also that~\(\frac{1}{\vert1-{\omega_{1}}\vert^2}C\) coincides with the matrix \(H'(\omega')\) obtained from the (connected) C-complex~\(S'\) with respect to the basis of~$H_1(S')$ represented by~$\mathcal{A}$.
    
    Now, consider the block-diagonal matrix~$P(\omega_1)=\begin{bmatrix}
    (1-\omega_{1})^{-1} \mathit{Id} & 0\\
    0 & (1-\omega_{1})^{-1/2}\mathit{Id}\\
     \end{bmatrix}$, and set
    \[
    \widehat{H}(\omega_1,\omega'):=
    P(\omega_1)H(\omega_1,\omega')P(\omega_1)^*\,.
    \] 
    By the considerations above, we get
    \[
    \lim_{\omega_{1}\rightarrow 1^{\pm}}\widehat{H}(\omega_{1},\omega')=
    \lim_{\omega_{1}\rightarrow 1^{\pm}}
    \begin{bmatrix}
    \frac{C}{\vert 1-\omega_1\vert^2} && \frac{D}{\vert 1-\omega_1\vert(1-\omega_1)^{1/2}} \\
    \frac{E}{\vert 1-\omega_1\vert(1-\overline{\omega}_1)^{1/2}} && \frac{F}{\vert 1-\omega_1\vert}
    \end{bmatrix}=
    \begin{bmatrix}
    H'(\omega') && 0 \\
    0 && F^\pm(\omega')
    \end{bmatrix}\,,
    \]
    with~$F^{\pm}(\omega'):=\lim_{\omega_1\to 1^\pm}\frac{F}{\vert 1-\omega_1\vert}$.
    The equality~$\lim_{\omega_1\to 1^\pm}\frac{1-\overline{\omega}_1}{\vert 1-\omega_1\vert}=\pm i$
    leads to
    \begin{equation}
       \label{eq:F}
       F^{\pm}(\omega')=\pm i\!\!\!\!\sum_{\varepsilon'\in\{\pm 1\}^{\mu-1}}
   \prod_{j=2}^{\mu}(1-\overline{\omega}_j^{\varepsilon_j})\big(A_{\mathcal{B}}^{(+1,\varepsilon')}-A_{\mathcal{B}}^{(-1,\varepsilon')}\big)\,,
    \end{equation}
    where~$A_{\mathcal{B}}^\varepsilon$ is the restriction
    of~\(A^{\varepsilon}\)
    to the subspace spanned by the classes of the curves in~$\mathcal{B}$.
    
Since~$\sign(\widehat{H}(\omega))=\sign(H(\omega))=\sigma_L(\omega)$ and~$\nul(\widehat{H}(\omega))=\nul(H(\omega))=\eta_L(\omega)$ for all~$\omega\in\mathbb{T}^\mu_*$,
Lemma~\ref{lemma:limit-signature} applied to~$H(t)=\widehat{H}(\exp(\pm 2\pi it),\omega')$ yields the inequality
\[
        \left|\lim_{\omega_{1}\rightarrow 1^{\pm}}\sigma_{L}(\omega_{1},\omega')-\sigma_{L'}(\omega')\mp\rho_{L}(\omega')\right|\le \eta_{L'}(\omega')+\tau_{L}(\omega')-\lim_{\omega_{1}\rightarrow 1}\eta_{L}(\omega_{1},\omega')\,,
        \]
  where
\begin{equation}
    \label{eq:def-rho-tau}
\rho_L(\omega'):=\sign\left(F^+(\omega')\right)\quad\text{and}\quad
\tau_L(\omega'):=\nul\left(F^+(\omega')\right)\,.
\end{equation}
By Lemma~\ref{lem:Alex}, it now only remains to show that the functions~$\rho_L$ and~$\tau_L$ defined via~\eqref{eq:F} and~\eqref{eq:def-rho-tau} coincide with the functions~$\rho_\ell$ and~$\tau_\ell$ defined via~\eqref{eq:tau-rho}, respectively.
(Here, we take the liberty to appeal to the forthcoming Lemma~\ref{lem:Alex}
based the four-dimensional point of view on the nullity; alternatively, the case~$\omega\in\mathbb{T}_*^\mu$ which suffices for our current purposes can
be obtained via the three-dimensional approach as a consequence of~\cite[Corollary~3.6]{C-F}.)

\medskip

Our demonstration of the equalities~$\rho_L=\rho_\ell$ and~$\tau_L=\tau_\ell$ rely on a sequence of lemmas. The proof of the first one is based on an observation of Cooper~\cite{Cooper}.

\begin{figure}[tbp]
        \centering
        \includegraphics[width=12.5cm]{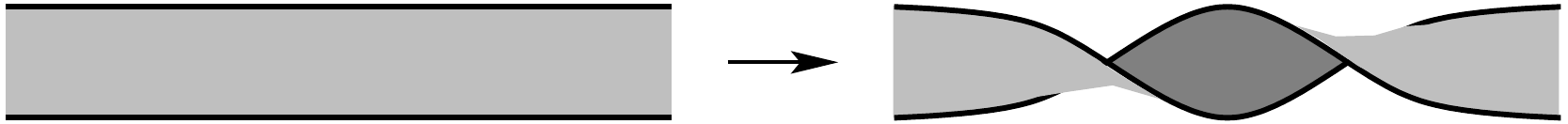}
        \caption{Twisting a band.}
        \label{band_twist}
    \end{figure}
    
\begin{lemma}
    \label{lemma:homotopy}
The functions~$\rho_{L}$ and~$\tau_{L}$ are invariant under link homotopy.   
\end{lemma}
\begin{proof}
Any crossing change between two strands of the same link component can be realised by twisting a band in a C-complex, as illustrated in Figure~\ref{band_twist}. Therefore, consider a C-complex~\(S\) and another
C-complex~\(\widehat{S}\) obtained from~$S$ by twisting a band.
    Since~$S$ and~$\widehat{S}$ are homotopy equivalent in an obvious way, their first homology groups are canonically isomorphic.
    Let us denote this isomorphism by~$H_1(S)\to H_1(\widehat{S}),x\mapsto\widehat x$.
    Then, one easily checks the equality 
    \[\lk(x^{(+1,\varepsilon')},y)-\lk(x^{(-1,\varepsilon')},y)=\lk(\widehat{x}^{(+1,\varepsilon')},
    \widehat{y})-\lk(\widehat{x}^{(-1,\varepsilon')},\widehat{y})\]
    for all~$x,y\in H_1(S)$ and~\(\varepsilon'\in\{\pm 1\}^{\mu-1}.\)
    Given the form of~$F^+(\omega')$ described in Equation~\eqref{eq:F}, we see
    that this matrix is invariant under band
    twisting. This shows that its signature~$\rho_{L}$ and nullity~$\tau_{L}$ are invariant under link homotopy.
    \end{proof}

    By Lemma~\ref{lemma:homotopy}, we can
    assume without loss of generality that~$L_1$
    is the unknot. By~\cite[Lemma~1]{Cim}, we can then find a C-complex~\(S=S_1\cup\dots\cup S_\mu\) for~\(L\) such that~\(S_{1}\) is a disk.
    Recall that we also assume that~$S'$ is connected.
    
    Note that since~$S$ is connected with~$\mu\ge 2$,
the number~$n$ of clasps involving~$S_1$ is strictly positive.
Let us number these clasps linearly from~$1$ to~$n$, starting with an arbitrary one and following the cyclic order along the oriented boundary~\(\partial S_{1}\). Let~$c(1),\dots,c(n)\in\{2,\dots,\mu\}$ denote the corresponding colors, and~$s(1),\dots,s(n)\in\{-1,1\}$ denote the corresponding signs. (By the {\em sign of a clasp\/} of color~$j$, we mean
its contribution to the linking number~$\lk(L_1,L_j)$.)

\begin{lemma}
    \label{lem:comb}
    For any~$\omega'\in\mathbb{T}^{\mu-1}_*$, the terms~$\rho_{L}(\omega')$ and~$\tau_{L}(\omega')$ are given by the signature and nullity of the tridiagonal Hermitian matrix~$F$ of size~$n-1$
with non-vanishing coefficients equal to
\begin{equation}
\label{equ:F}
F_{k,k-1}=\overline{F}_{k-1,k}=\frac{i}{1-\omega^{s(k)}_{c(k)}},\qquad F_{k,k}=\frac{i\cdot(\omega^{s(k)}_{c(k)}\omega^{s(k+1)}_{c(k+1)}-1)}{(1-\omega^{s(k)}_{c(k)})(1-\omega^{s(k+1)}_{c(k+1)})}\,.
\end{equation}
\end{lemma}
    \begin{proof}
Let~$S$ be a C-complex associated to a~$\mu$-component link~$L$, with~$S_1$ a disc, and~$S'$ connected. By definition of~$\rho_{L}$ and~$\tau_{L}$ (recall Equation~\eqref{eq:F}), we need to compute the subgroup of~$H_1(S)$ spanned by curves of~$\mathcal{B}$, as well as the numbers~$\lk(x^{(1,\varepsilon')},y)-\lk(x^{(-1,\varepsilon')},y)=:\lk(x^{(1,\varepsilon')}-x^{(-1,\varepsilon')},y)$
for all~$x,y$ in this subspace and all~$\varepsilon'\in\{\pm 1\}^{\mu -1}$.

Since~$S_1$ is a disc and~$S'$ is connected, an easy
homological computation shows that the family~$\mathcal{B}$ can be chosen to be~$n-1$ cycles, each passing through consecutive clasps around~$S_1$.
For definiteness, let us write~$\mathcal{B}=\{x_1,\dots,x_{n-1}\}$,
where for all~$1\le k\le n-1$, the cycle~$x_k$ enters~$S_1$ through the~$k^{\text{th}}$ clasp and exits~$S_1$ through the~$(k+1)^{\text{th}}$ one.

A straightforward computation leads to the following results,
valid for all~$\varepsilon'\in\{\pm 1\}^{\mu -1}$:
for any~$1\le k\le n-1$, we have
\[
\lk(x_k^{(1,\varepsilon')}-x_k^{(-1,\varepsilon')},x_k)=\begin{cases}
    -1&\text{if~$\varepsilon_k=s(k)$ and~$\varepsilon_{k+1}=s(k+1)$;}\cr
    +1&\text{if~$\varepsilon_k=-s(k)$ and~$\varepsilon_{k+1}=-s(k+1)$;}\cr
    0&\text{else}\,,
\end{cases}
\]
while for any~$2\le k\le n-1$, we have
\begin{align*}
    \lk(x_k^{(1,\varepsilon')}-x_k^{(-1,\varepsilon')},x_{k-1})&=\begin{cases}
    +1&\text{if~$\varepsilon_k=s(k)$;}\cr
    0&\text{else}\,,
\end{cases}\cr
\lk(x_{k-1}^{(1,\varepsilon')}-x_{k-1}^{(-1,\varepsilon')},x_{k})&=\begin{cases}
    -1&\text{if~$\varepsilon_k=s(k)$;}\cr
    0&\text{else}.
    \end{cases}
\end{align*}
One then checks that the coefficients of
the matrix~$F^+(\omega')$ defined by~\eqref{eq:F}
are equal to
\[
F^+(\omega')_{k,k'}:=i\!\!\!\!\sum_{\varepsilon'\in\{\pm 1\}^{\mu-1}}
   \prod_{j=2}^{\mu}(1-\overline{\omega}_j^{\varepsilon_j})\lk(x_k^{(1,\varepsilon')}-x_k^{(-1,\varepsilon')},x_{k'})=\prod_{j=2}^{\mu}\vert 1-\omega_j\vert^2\cdot F_{k,k'}
\]
for all~$1\le k,k'\le n-1$, with~$F_{k,k'}$ as in Equation~\eqref{equ:F}.
(This formula holds whether or not the involved clasps
have the same color.)
The scalar~$\prod_{j=2}^{\mu}\vert 1-\omega_j\vert^2$ being strictly positive for all~$\omega'\in\mathbb{T}_*^{\mu-1}$, the statement follows.
\end{proof}

\begin{lemma}
    \label{lem:invarianceF}
The functions~$\rho_{L}$ and~$\tau_{L}$ are invariant under the following transformations:
\begin{enumerate}
    \item removal of two adjacent clasps of the same color and opposite signs (as long as~$S$ remains connected);
    \item permutation of two adjacent clasps of different colors.
\end{enumerate}
\end{lemma}
\begin{proof}
By Lemma~\ref{lem:comb}, we only need to check that for all~$\omega'\in\mathbb{T}_*^{\mu-1}$, the signature and nullity of the
tridiagonal Hermitian matrix~$F$ given by Equations~\eqref{equ:F} are 
unchanged by these two transformations.

To show the invariance under the first transformation, observe that~$F_{k,k}$ vanishes if~$c(k)=c(k+1)$ and~$s(k)\neq s(k+1)$. Note also that~$F_{k,k-1}$ never vanishes for~$\omega'\in\mathbb{T}^{\mu-1}_*$. Renumbering the clasps starting with the~$(k+2)^\text{th}$ one (and thus ending with the~$(k+1)^\text{th}$ one), the corresponding matrix~$F$ is of the form
\[
F=\begin{bmatrix}
F'&\xi&0\cr \xi^*&\alpha&\lambda\cr 0&\overline{\lambda}&0
\end{bmatrix}\,,
\]
with~$\alpha\in\mathbb{R},\lambda\in\C^*$, and~$F'$ the matrix corresponding to the C-complex with both clasps removed. The fact that the signature and nullity of~$F$ and~$F'$ coincide is well-known, see e.g. the proof of the invariance of the Levine-Tristram signature in~\cite{Lic}.

We are left with the proof that the signature and nullity of~$F$ are unchanged when permuting two adjacent clasps
of different colors. Without loss of generality,
let us assume that these two clasps are the two last ones in the linear numbering~$1,\dots,n$, and let us denote the occurring variables by~$z_1:=\omega_{c(n-2)}^{s(n-2)},~z_2:=\omega_{c(n-1)}^{s(n-1)}$ and~$z_3:=\omega_{c(n)}^{s(n)}$. By Equation~\eqref{equ:F}, we thus need to compare two matrices of the form
\[
F=\begin{bmatrix}
F_0&\xi&0\cr
\xi^*&\frac{i\cdot(z_1z_2-1)}{(1-z_1)(1-z_2)} & \frac{-i}{1-\overline{z}_2} \cr
0&\frac{i}{1-z_2}&\frac{i\cdot(z_2z_3-1)}{(1-z_2)(1-z_3)}
\end{bmatrix}
\quad
\text{and}
\quad
F'=\begin{bmatrix}
F_0&\xi&0\cr \xi^*&\frac{i\cdot(z_1z_3-1)}{(1-z_1)(1-z_3)} & \frac{-i}{1-\overline{z}_3} \cr
0&\frac{i}{1-z_3}&\frac{i\cdot(z_2z_3-1)}{(1-z_2)(1-z_3)}
\end{bmatrix}\,.
\]
If~$z_2z_3=1$, then~$F$ and~$F'$ both have the signature and nullity of~$F_0$ by the first step, and the invariance holds.
If~$z_2z_3\neq 1$, then one can consider the matrices
\[
P=\begin{bmatrix}
\mathit{Id} &0 &0\cr
0 & 1 & 0\cr
0 & \frac{1-z_3}{1-z_2z_3} & 1
\end{bmatrix}\quad\text{and}\quad P'=\begin{bmatrix}
\mathit{Id} &0 &0\cr
0 & 1 & 0\cr
0 & \frac{1-z_2}{1-z_2z_3} & 1
\end{bmatrix}\,.
\]
A direct computation now leads to the equality
\[
P^*FP=\begin{bmatrix}
F_0&\xi&0\cr
\xi^*&\frac{i\cdot(z_1z_2z_3-1)}{(1-z_1)(1-z_2z_3)} & 0 \cr
0&0&\frac{i\cdot(z_2z_3-1)}{(1-z_2)(1-z_3)}
\end{bmatrix}=(P')^*F'P'\,,
\]
concluding the proof.
\end{proof}

The fact that~$\rho_L$ (resp.~$\tau_L$) coincides with~$\rho_\ell$ (resp.~$\tau_\ell$) of Equation~\eqref{eq:tau-rho} now follows from one last lemma.

\begin{lemma}
    \label{lem:signF}
    For any~$n\ge 1$ and~$z=(z_1,\dots,z_n)\in\mathbb{T}^n_*$, let~$G_n(z)$ denote the tridiagonal matrix of size~$n-1$ with non-vanishing coefficients equal to 
\begin{equation*}
    G_n(z)_{k,k-1}=\overline{G_n(z)}_{k-1,k}=\frac{i}{1-z_k}\quad\text{and}\quad G_n(z)_{k,k}=\frac{i\cdot(z_kz_{k+1}-1)}{(1-z_k)(1-z_{k+1})}\,.
\end{equation*}
Then, we have
\[
\sign(G_n(z))=\rho(z_1,\dots,z_n)\quad\text{and}\quad\nul(G_n(z))=\begin{cases} 1&\text{if~$z_1\cdots z_n=1$};\cr
0&\text{else},
\end{cases}
\]
with~$\rho$ defined by~\eqref{eq:def-rho} and~\eqref{eq:def-rho-mult}.
\end{lemma}
\begin{proof}
We proceed by induction on~$n\ge 1$. The case~$n=1$ holds with the right conventions (namely, that the signature and nullity of an empty matrix vanish), and the case~$n=2$ is straightforward. Hence, let us assume that the lemma holds up to~$n-1$, and consider~$G_n(z)$ with~$n\ge 3$
and~$z=(z_1,\dots,z_n)\in\mathbb{T}^n_*$.

If~$z_{n-1}$ and~$z_n$ satisfy~$z_{n-1}z_n=1$, then
the diagonal coefficient~$G_n(z)_{n-1,n-1}$ vanishes
while the off-diagonal ones~$G_n(z)_{n-1,n-2}=\overline{G_n(z)}_{n-2,n-1}$ do not vanish.
As a consequence, as in the first step of Lemma~\ref{lem:invarianceF}, the matrix~$G_n(z)$ has the same signature and nullity as~$G_{n-2}(z)$, which are known by the induction hypothesis. Since~$\rho(z_1,\dots,z_n)$ is easily seen to coincide with~$\rho(z_1,\dots,z_{n-2})$ if~$z_{n-1}z_n=1$ and similarly for the nullity, the lemma is checked in this case.

Let us now assume that~$z_{n-1}$ and~$z_n$ are such that~$z_{n-1}z_n\neq 1$. Then, as in the second step of
the proof of Lemma~\ref{lem:invarianceF}, one can consider the matrix
\[
P=\begin{bmatrix}
\mathit{Id} &0 &0\cr
0 & 1 & 0\cr
0 & \frac{1-z_n}{1-z_{n-1}z_{n}} & 1
\end{bmatrix}\,.
\]
A direct computation leads to the equality
\[
P^*G_n(z_1,\dots,z_n)P=G_{n-1}(z_1,\dots,z_{n-2},z_{n-1}z_n)\oplus\left(\frac{i\cdot(z_{n-1}z_n-1)}{(1-z_{n-1})(1-z_n)}\right)\,.
\]
By the induction hypothesis, we now get
\[
\sigma(G_n(z_1,\dots,z_n))=\rho(z_1,\dots,z_{n-2},z_{n-1}z_n)+\rho(z_{n-1},z_n)=\rho(z_1,\dots,z_n)
\]
and
\[
\eta(G_n(z_1,\dots,z_n))=\eta(G_{n-1}(z_1,\dots,z_{n-2},z_{n-1}z_n))=\begin{cases}
    1&\text{if~$z_1\cdots z_n=1$};\cr
    0&\text{else},
\end{cases}
\]
concluding the proof.
\end{proof}

We are now ready to conclude the proof of Theorem~\ref{thm:limit-of-signature-inequality}, i.e. to show  that~$\rho_L$ (resp.~$\tau_L$) coincides with~$\rho_\ell$ (resp.~$\tau_\ell$) of Equation~\eqref{eq:tau-rho}.

First note that, as a consequence of Lemma~\ref{lem:invarianceF}, the functions~$\rho_L$ and~$\tau_L$ only depend on the linking numbers~$\ell_2:=\lk(L_1,L_2),\dots,\ell_\mu:=\lk(L_1,L_\mu)$.
More precisely, if all these numbers vanish, then the transformations of Lemma~\ref{lem:invarianceF}
can be carried to the point where we are left with~$2$
claps, of the same color and opposite signs;
this leads to~$\rho_L=0=\rho_\ell$ and~$\tau_L=1=\tau_\ell$, so Theorem~\ref{thm:limit-of-signature-inequality} holds in such a case.
If the linking numbers do not all vanish, then via these two transformations, one can assume that the~$n:=\vert\ell\vert$ clasps are cyclically ordered around~$S_1$ as~$\vert\ell_2\vert$ clasps of color~$2$, followed by~$\vert\ell_3\vert$ clasps of color~$3$, and so on, ending
with~$\vert\ell_\mu\vert$ clasps of color~$\mu$.
Now, observe that~$\rho_L(\omega')$ and~$\eta_L(\omega')$
are the signature and nullity of the matrix~$F$ of Lemma~\ref{lem:comb}, which coincides with~$G_n(z)$ evaluated at
\[
z=(z_1,\dots,z_n)=(\omega_{c(1)}^{s(1)},\dots,\omega_{c(n)}^{s(n)})=(\underbrace{\omega^{s_2}_2,\dots,\omega^{s_2}_2}_{\vert\ell_2\vert},\dots,\underbrace{\omega^{s_\mu}_\mu,\dots,\omega^{s_\mu}_\mu}_{\vert\ell_\mu\vert})\,.
\]
The explicit form given by~\eqref{eq:tau-rho} now follows from Lemma~\ref{lem:signF},
concluding the proof of Theorem~\ref{thm:limit-of-signature-inequality}.

\medskip

Let us finally turn to the proof of the addendum yielding a more geometric description
of the function~$\rho_\ell$.

\begin{proof}[Proof of the addendum]
Fix~$\ell\in\Z^{\mu-1}\setminus\{0\}$. By Equation~\eqref{eq:tau-rho} and the lemmas above, the function~$\rho_\ell$ is equal to the signature of a matrix whose nullity is equal to~$1$ on~$\Sigma_\ell\subset\mathbb{T}^{\mu-1}_*$ and vanishes elsewhere.
This implies that~$\rho_\ell$ satisfies
the first point of the statement: it is constant
on the connected components of the complement of~$\Sigma_\ell$ in~$\mathbb{T}^{\mu-1}_*$,
and on the connected components of~$\Sigma_\ell\subset\mathbb{T}^{\mu-1}_*$. This also
implies that, when~$\omega'$ crosses a component
of~$\Sigma_\ell$, the function~$\rho_\ell(\omega')$
either jumps by~$\pm 2$ or stays constant, and always takes the average value on~$\Sigma_\ell$.

By the explicit form of~$\rho_\ell$ given in~\eqref{eq:def-rho},~\eqref{eq:def-rho-mult} and~\eqref{eq:tau-rho}, we see
that it satisfies the second point, i.e.
\[
\lim_{\omega_j\to 1^{s_j}}\rho_\ell(\omega_2,\dots,\omega_\mu)=\lim_{z_j\to 1^+}\rho(z_1,\dots,z_{\vert\ell\vert})=\sum_{k=1}^{\vert\ell\vert-1}\lim_{z\to 1^+}\rho(z,z^k)=\vert\ell\vert-1\,.
\]
We shall denote by~$\omega'=1^s$ this corner of the open torus~$\mathbb{T}^{\mu-1}_*$. Note that if some~$\ell_j$ vanishes, then this corner is not uniquely defined, but~$\rho_\ell$ being independent of~$\omega_j$, any value of~$s_j$ can be chosen.
By the symmetry property~$\rho_\ell(\overline{\omega}')=-\rho_\ell(\omega')$,
we obtain the fact that~$\rho_\ell$ takes the opposite value~$1-\vert\ell\vert$ at the opposite corner~$\omega'=1^{-s}$.

Now, consider the closed path~$\gamma\colon S^1\to\mathbb{T}^{\mu-1}$ defined by~$\gamma(z)=\left(z^{s_2},\dots,z^{s_\mu}\right)$,
which restricts to an open path in~$\mathbb{T}_*^{\mu-1}$ from the corner~$1^s$
to the opposite corner~$1^{-s}$.
Consider also for each~$\alpha\in S^1$ the hyperplan\new{e}
\[
\Sigma^{(\alpha)}_\ell:=\{(\omega_2,\dots,\omega_\mu)\in\mathbb{T}^{\mu-1}\,|\,\omega_2^{\ell_2}\cdots\omega_\mu^{\ell_\mu}=\alpha\}\,.
\]
This defines a foliation of the full torus~$\mathbb{T}^{\mu-1}$ by hyperplans,
with each leaf intersecting the path~$\gamma$
transversally in
\[
s_2\ell_2+\dots+s_\mu\ell_\mu=\vert\ell_2\vert+\dots+\vert\ell_\mu\vert=\vert\ell\vert
\]
points. Moreover, the hyperplan~$\Sigma^{(1)}_\ell$ is nothing but the closure of~$\Sigma_\ell\subset\mathbb{T}^{\mu-1}_*$ in the full torus~$\mathbb{T}^{\mu-1}$, and~$\Sigma_\ell$ intersects~$\gamma$ exactly~$\vert\ell\vert-1$ times.

As a consequence, the open path~$\gamma\colon S^1\setminus\{1\}\to\mathbb{T}^{\mu-1}_*$ 
meets each of the~$\vert\ell\vert$ connected components of the complement of~$\Sigma_\ell$ in~$\mathbb{T}^{\mu-1}_*$, and each of the~$\vert\ell\vert-1$ connected components of~$\Sigma_\ell$. Since~$\rho_\ell$ takes the value~$\vert\ell\vert-1$ near~$\omega'=1^s$, the value~$1-\vert\ell\vert$ near the opposite corner~$\omega'=1^{-s}$,
and jumps at most by~$\pm 2$ when crossing a
connected component of~$\Sigma_\ell$, it necessarily jumps by~$-2$ when crossing any
of these~$\vert\ell\vert-1$ components, thus determining its values on the full domain. This concludes the proof
of the addendum.
\end{proof}

\section{Torres-type formulas for the signature and nullity}
\label{sec:Torres}

The aim of this section is twofold. First,
in Section~\ref{sub:extension},
we extend the signature and nullity functions from~$\Tm=(S^1\setminus\{1\})^\mu$ to the full torus~$\mathbb{T}^{\mu}=(S^1)^\mu$.
Then, in Sections~\ref{sub:Torres-sign} and ~\ref{sub:Torres-null}, we devise Torres-type formulas
for these extended signatures and nullity, respectively. These results
\new{rely on several technical lemmas to be found in Appendix~\ref{ap:plumbed}, and}
are used in Section~\ref{sec:4D} to
study limits of signatures.

\subsection{Extension of the signature and nullity to the full torus}
\label{sub:extension}

Let~$L=L_1\cup\dots\cup L_\mu$ be a colored link in~$S^3$. Recall that the
associated signature~$\sigma_L$ and nullity~$\eta_L$ are~$\mathbb{Z}$-valued maps defined on~$\Tm=(S^1\setminus\{1\})^\mu$. The aim of this section is
to extend these maps in a natural way to the full torus~$\mathbb{T}^\mu=(S^1)^\mu$.

\medskip

\new{In a nutshell, these extended signature and nullity are defined as the twisted signature and nullity of a~$4$-dimensional~$\Z^\mu$-manifold~$W$ bounding the closed~$3$-manifold~$M_L$ endowed with a meridional homomorphism~$\varphi\colon H_1(M_L)\to\Z^\mu$ (recall Section~\ref{sec:plumbed}). There are several issues 
with this approach:
\begin{itemize}
    \item The meridional homomorphism~$\varphi\colon H_1(M_L)\to\Z^\mu$ is not unique (recall Remark~\ref{rems:M_L}.2).
    \item There exists~$\mu$-colored links~$L$ such that for any meridional homomorphism~$\varphi$, the~$\Z^\mu$-manifold $(M_L,\varphi)$ does not
    bound over~$\Z^\mu$ (recall Example~\ref{exs:meridional}.2).
\end{itemize}
However, these obstacles can be overcome as follows.}

\medskip

\new{
Given a~$\mu$-colored link~$L$, consider the associated closed~$3$-manifold~$M_L$ defined in Section~\ref{sec:plumbed}. By Lemma~\ref{lem:ML},
there exists a meridional homomorphism~$\varphi\colon H_1(M_L)\to\Z^\mu$. Let
us choose an arbitrary one, and consider the associated bordism class~$(M,\varphi)\in\Omega_3(\Z^\mu)$. The canonical isomorphism~$\Omega_3(\Z^\mu)\simeq\Z^{\mu\choose 3}$ given in~\eqref{eq:bordism} yields
\begin{align*}
    \Omega_3(\Z^\mu)&\longrightarrow \Z^{{\mu}\choose{3}}\\
    (M_L,\varphi)&\longmapsto \mu_L=\{\mu_L(ijk)\mid 1\le i<j<k\le \mu\}\,.
\end{align*}
Note that the integers~$\mu_L(ijk)$ are in general not invariants
      of the colored link~$L$, as they depend on the choice of the meridional homomorphism~$\varphi$. However, by Remark~\ref{rems:M_L}.2,
      these integers are well-defined if all linking numbers vanish.
      For example, the~$3$-colored Borromean rings~$B(123)$ endowed with the appropriate coloring and orientation yields~$\mu_{B(123)}(123)=1$ by Example~\ref{exs:meridional}.2.}
\new{
\begin{remark}
  \label{rem:Milnor}
  This brings to mind
      Milnor's triple linking numbers~\cite{Mil57}, whose notation we chose for a reason.
      Indeed, when each component of~$L$ is endowed with a different color, the fact that the integers~$\mu_L(ijk)$ coincide with (some refined version of) the aforementioned triple linking numbers can be extracted from~\cite{DNOP}.
\end{remark}
}

\new{
Now, consider the auxiliary~$\mu$-colored link~$L^\#$ defined as follows:
\begin{equation}
\label{eq:Lsharp}
    L^\#:=L\sqcup \bigsqcup_{i<j<k} -\mu_L(ijk)\cdot B(ijk)\,,
\end{equation}
where~$\sqcup$ denotes the distant sum,~$B(ijk)$ the Borromean rings endowed with orientations and colors so that~$\mu_{B(ijk)}(ijk)=1$, and~$n\cdot B(ijk)$ stands for the distant sum of~$\vert n\vert$ copies of~$B(ijk)$ (resp.~$B(jik)$) if~$n\ge 0$ (resp.~$n\le 0$). By construction, the corresponding manifold~$M_{L^\#}$ is given by the connected sum of~$M_L$
with three-dimensional tori endowed with color-induced homomorphisms~$H_1(\mathbb{T}^3)\hookrightarrow\Z^\mu$. Hence, the homomorphism~$\varphi\colon H_1(M_L)\to\Z^\mu$ extends uniquely to a meridional homomorphism~$\varphi^\#\colon H_1(M_{L^\#})\to\Z^\mu$, which by construction satisfies
\begin{equation}
\label{eq:MLbounds}
(M_{L^\#},\varphi^\#)=0\in\Omega_3(\Z^\mu)\,.
\end{equation}
Therefore, there exists a~$\Z^\mu$-manifold~$(W,\Phi)$ such that~$\partial W=M_{L^\#}$ and~$\varphi^\#$ factors through $\Phi\colon H_1(W)\to\Z^\mu$.}

\new{\begin{remark}
\label{rem:rho}
At this point, it would be possible to define our extended signature as the signature defect~$\sigma_\omega(W)-\sigma(W)$.
This would indeed be an invariant of~$L$, but it would in general
{\em not\/} extend the usual signatures.
For this to hold, we need to consider a specific type of~$4$-manifold bounding~$M_{L^\#}$.
\end{remark}}

To do so, let us consider a bounding surface~$F=F_{1}\cup\dots\cup F_{\mu}\subset B^4$ for~\new{$L^\#$} obtained by pushing a totally-connected C-complex for~$L^\#$ from~$S^3$ into~$B^4$. Let us write~$X_{L^\#}=S^3\setminus\nu(L^\#)$ and~$V_{F}=B^{4}\setminus\nu(F)$.
Note that we have~$\partial V_F=X_{L^\#}\cup-P(F)$, where~$P(F)$ is the boundary of a tubular neighborhood of \(F\) in \(B^{4}\).
Moreover, this latter manifold can be described as the plumbed manifold defined by the plumbing graph~$\Gamma_F$ with vertices given by the surfaces~$F_i$ and signed edges given by the signed intersections of these surfaces in~$B^4$ (or equivalently, the signed clasps of the C-complex). We refer the reader to~\cite[Section 4.3]{toffoli} for details.

Note that~$P(F)$ and~$P(L^\#)$ have the common boundary~$\partial\nu(L^\#)$.
Let us form the closed 3-manifold
\[
P(G)=P(F)\cup_{\partial}-P(L^\#)\new{\simeq P(F)\cup_{\partial}P(\overline{L^\#})}\,,
\]
where~$\overline{L^\#}$ denotes the mirror image of~$L^\#$.
Clearly,~$P(G)$ can be described as the plumbing manifold obtained
from the plumbing graph~$G$ given as follows. The vertices of~$G$ correspond to the colors~$\{1,\dots,\mu\}$, with the closed surface~$\widehat{F}_i$ associated to the color~$i$ obtained from~$F_i$ by capping it off with~$|L_i|$ 2-discs.
The edges of~$G$ are given by the signed intersections of the surfaces~$F_i$ in~$B^4$, and by~$\sum_{K\subset L_i,K'\subset L_j}|\lk(K,K')|$ edges between~$\widehat{F}_i$ and~$\widehat{F}_j$ with signs opposite to the signs of
the linking numbers.
By construction, the graph~$G$ is balanced \new{(recall the end of Section~\ref{sec:plumbed})}.

\new{Let us now focus on the meridional homomorphisms. By construction, the homomorphism $\varphi^\#\colon H_1(M_{L^\#})\to\Z^\mu$ extends~$\varphi_X\colon H_1(X_{L^\#})\to\Z^\mu$ and some~$\varphi_P\colon H_1(P(L^\#))\to\Z^\mu$, a fact that we will denote by~$\varphi^\#=\varphi_X\cup\varphi_P$. Also, we have an isomorphism~$H_1(V_F)\simeq\Z^\mu$ which extends~$\varphi_X$ (see e.g.~\cite[Lemma~3.1]{CNT}), and induces some~$\varphi_F\colon H_1(P(F))\to\Z^\mu$. Therefore, the maps~$\varphi_F$ and~$\varphi_P$
agree on~$H_1(\partial\nu(L^\#))$ and induce a meridional homomorphism~$\varphi_G\colon H_1(P(G))\to\Z^\mu$.
Moreover, since~$M_{L^\#}=X_{L^\#}\cup -P(L^\#)$ endowed with~$\varphi^\#=\varphi_X\cup\varphi_P$ bounds over~$\Z^\mu$ (recall~\eqref{eq:MLbounds}),
while~$\partial V_F=X_{L^\#}\cup-P(F)$ endowed with~$\varphi_X\cup\varphi_F$ bounds by construction, it follows that~$P(G)=P(F)\cup-P(L^\#)$ endowed with~$\varphi_G=\varphi_F\cup\varphi_P$ bounds as well. This is illustrated in Figure~\ref{fig:W}.}

\begin{figure}[tbp]
    \centering
    \begin{overpic}[width=3.5cm]{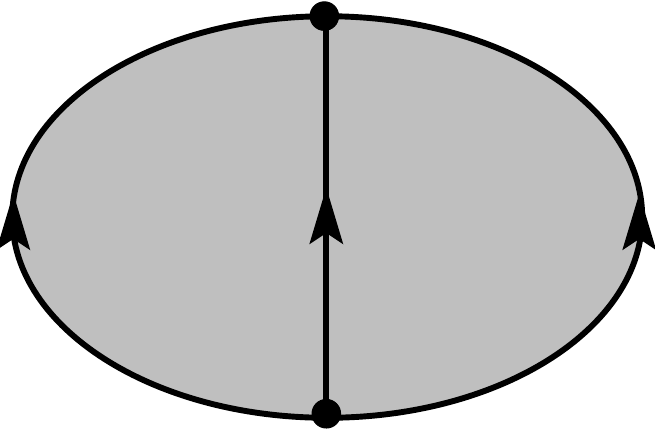}
    \put (-40,30){$\scriptstyle{(X_{L^\#},\varphi_X)}$}
    \put (52,18){$\scriptstyle{(P(F),\varphi_F)}$}
    \put (103,30){$\scriptstyle{(P(L^\#),\varphi_P)}$}
    \put (23,35){$V_F$}
    \put (70,35){$Y_F$}
    \put (43,68){$\scriptstyle{\partial\nu(L^\#)}$}
    \end{overpic}
    \caption{\new{Construction of the~$\Z^\mu$-manifold~$W_F$.}}
    \label{fig:W}
\end{figure}

\new{By Lemma~\ref{lem:Y}, the~$\Z^\mu$-manifold~$(P(G),\varphi_G)$ bounds a compact connected oriented \(\Z^{\mu}\)-manifold \((Y_F,f)\) such that~\(\pi_{1}(Y_F) = \Z^{\mu}\),~\(f\) is an isomorphism and~$\sign_\omega(Y_F)=0$ for
      all~\(\omega \in \mathbb{T}^{\mu}\).}
Note that the manifolds~$V_F$ and~$Y_F$ both admit~$P(F)$ as part of their
boundary, \new{with the meridional isomorphisms~$H_1(V_F)\simeq\Z^\mu$ and~$H_1(Y_F)\stackrel{f}{\simeq}\Z^\mu$ both restricting to~$\varphi_F$ on~$H_1(P(F))$.} Therefore, one can consider the 4-manifold
\begin{equation}\label{eq:the-W-F-mfd}
    W_{F}=V_{F}\cup_{P(F)} Y_{F}
\end{equation}
\new{equipped with a meridional homomorphism~$\Phi\colon H_1(W_F)\to\Z^\mu$, whose boundary is~$(M_{L^\#},\varphi^\#)$.}

\bigbreak

We are finally ready to extend the signature and nullity to the full torus.

\begin{definition}
\label{def:extension}
For any~$\omega\in\mathbb{T}^\mu$, set
\[
\sigma_F(\omega)=\sign_\omega(W_F)\quad\text{and}\quad\eta_F(\omega)=\nul_\omega(W_F)\new{-\delta_L(\omega)}\,,
\]
\new{where~$
\delta_L(\omega)=\sum_{i<j<k}\vert\mu_L(ijk)\vert+2\sum_{i<j<k;\,\omega_i=\omega_j=\omega_k=1}\vert\mu_L(ijk)\vert$.}
\end{definition}

A priori, these extended signatures and nullity might depend on the choice of the bounding surface $F$, \new{and of the meridional homomorphism~$\varphi$.} This is not the case, as demonstrated by the following statement.

\begin{theorem}
\label{thm:extension}
The maps~$\sigma_F\colon\mathbb{T}^\mu\to\mathbb{Z}$ and~$\eta_F\colon\mathbb{T}^\mu\to\mathbb{Z}$
only depend on the colored link~$L$, and extend the multivariable
signature and nullity~$\sigma_L\colon\mathbb{T}_*^\mu\to\mathbb{Z}$ and~$\eta_L\colon\mathbb{T}_*^\mu\to\mathbb{Z}$, respectively.
\end{theorem}

The proof of this theorem relies on
\new{several technical results to be found in Appendix~\ref{ap:plumbed}, as well as on}
the following algebraic lemma.

\begin{lemma}\label{lemma:triviality-tor-groups}
  Let \(\Lambda_{\mu}\) denote the group ring \(\C[\Z^{\mu}]\).
  For any \(\omega \in \mathbb{T}^\mu \setminus \{(1,\ldots,1)\}\) and any \(i \geq 0\), we have $\Tor_{i}^{\Lambda_{\mu}}(\C^{\omega},\C) = 0$.
  Furthermore, for \(\omega=(1,\dots,1)\), we have \(\Tor_i^{\Lambda_{\mu}}(\C^{\omega},\C) = \C^{\binom{\mu}{i}}\).
\end{lemma}
\begin{proof}
  This computation can be performed using Koszul resolutions, see e.g.~\cite[Chapter 4.5]{Weibel}.
  For any \(x \in \Lambda_{\mu}\) consider the chain complex~$K(x) \coloneqq \Lambda_{\mu} \xrightarrow{x} \Lambda_{\mu}$ concentrated in degrees~$1$ and~$0$.
  Let \(t_1,\ldots,t_{\mu}\) be the elements of \(\Lambda_{\mu}\) corresponding to the canonical basis of \(\Z^{\mu}\), so that \(\Lambda_{\mu} = \C[t_1^{\pm1},\dots,t_\mu^{\pm1}]\).
  Consider the Koszul complex
  \[
    K_\mu:= K(t_1-1)\otimes_{\Lambda_\mu} K(t_2-1)\otimes_{\Lambda_\mu}\cdots\otimes_{\Lambda_\mu} K(t_\mu-1)\,.
  \]
  By~\cite[Corollary~4.5.5]{Weibel}, the complex~\(K_{\mu}\) is a free resolution of~$\Lambda_{\mu}/(t_1-1,\dots,t_\mu-1)=\C$ over \(\Lambda_{\mu}\).
  Therefore, \(\Tor_i^{\Lambda_{\mu}}(\C^{\omega},\C) = H_{i}(\C^{\omega} \otimes_{\Lambda_{\mu}} K_{\mu})\).
  Since a tensor product of an acyclic complex with any other complex is again acyclic, we get that
  \(\Tor_{i}^{\Lambda_{\mu}}(\mathbb{C}^{\omega},\C)=0,\)
  for \(i\geq 0\) if there is some \(\omega_{j}\neq 1\).
  If~\(\omega_{j}=1\) for all~$j$, then all of the differentials in \(K_{\mu}\) vanish and we get~\(\Tor_i^{\Lambda_{\mu}}(\C^{\omega},\C) = \C^{\binom{\mu}{i}}\).  
\end{proof}

We will also use the following statement.

\begin{proposition}
\label{prop:nullity}
The nullity function~$\eta_F\colon\mathbb{T}^\mu\to\mathbb{Z}$ is given by
\[
\eta_F(\omega)=
\begin{cases}
\dim H_{1}(M_L;\C^{\omega})&\text{for~$\omega\neq(1,\dots,1)$};\\
\dim H_1(M_L;\mathbb{C})-\mu&\text{for~$\omega=(1,\dots,1)$}.
\end{cases}
\]
\end{proposition}
\begin{proof}
Recall that the twisted intersection form of~$W_F$ is defined as the composition of the maps
\[
 H_2(W_F;\mathbb{C}^\omega)\stackrel{j_*}{\longrightarrow} H_2(W_F,M_L;\mathbb{C}^\omega)\stackrel{\mathrm{PD}}{\longrightarrow} H^2(W_F;\mathbb{C}^\omega)\stackrel{\mathrm{ev}}{\longrightarrow} \hom(H_2(W_F;\mathbb{C}^\omega),\mathbb{C})^{\mathrm{tr}}\,,
\]
the second and third ones being the isomorphisms given by
Poincar\'e-Lefschetz duality and the universal coefficient theorem, see~Appendix~\ref{sec:algebr-preliminaries} for more details.
Therefore, we have
\[
\eta_F(\omega)\new{+\delta_L(\omega)}=\nul_\omega(W_F)=\dim\ker(j_*)=\dim\mathrm{coker}(j_*)
\]
since $H_2(W_F;\mathbb{C}^\omega)$ and $H_2(W_F,M_{\new{L^{\#}}};\mathbb{C}^\omega)$ have the same dimension.
By the exact sequence of the pair $(W_F,M_{L^\#})$,
we get
\begin{equation}
\label{eq:nullity}
    \eta_F(\omega)\new{+\delta_L(\omega)}=\dim\ker\left(H_1(M_{L^\#};\mathbb{C}^\omega)\stackrel{i_*}{\longrightarrow} H_1(W_F;\mathbb{C}^\omega)\right)
\end{equation}
for all $\omega\in\mathbb{T}^\mu$.

Now, recall from~\cite[Proposition 3.1]{CFT} that since~\(F\) is totally connected, we have~\(\pi_{1}(V_{F}) = \Z^{\mu}\).
  The homomorphism~$\pi_{1}(Y_{F})\to\Z^{\mu}$ being an isomorphism, and the meridional
  homomorphism~$\pi_{1}(P(F))\to\Z^{\mu}$ being onto,
the Seifert-van Kampen theorem implies that~\(\pi_{1}(W_{F}) \cong \Z^{\mu}\). Indeed, one easily checks that since the diagonal homomorphism~$\pi_{1}(P(F))\to\Z^{\mu}$ is onto and the other maps to~$\Z^\mu$ are isomorphisms, the following diagram satisfies the universal property of the pushout (or amalgamated product):
  \begin{center}
        \begin{tikzcd}
    \pi_1(V_F) \arrow[r, "\cong"] & \Z^\mu \\
    \pi_1(P(F))\arrow[u] \arrow[r] & \pi_1(Y_F)\,.\arrow[u, "\cong"]
    \end{tikzcd}
        \end{center}
As a consequence, the \(\Z^{\mu}\)-cover \(\widetilde{W_{F}}\) of \(W_{F}\) satisfies \(\pi_{1}(\widetilde{W_{F}}) \cong H_{1}(\widetilde{W_{F}}) = 0\).
    Therefore, the Universal Coefficient Spectral Sequence (see Theorem~\ref{thm:UCSS}) implies that for any \(\omega \in \mathbb{T}^\mu \setminus \{(1,\ldots,1)\}\), we have an exact sequence
    \[H_{1}(\widetilde{W_{F}};\C) \otimes_{\Lambda_{\mu}} \C^{\omega} \to H_{1}(W_{F};\C^{\omega}) \to \Tor_{1}^{\Lambda_{\mu}}(\C,\C^{\omega}) \to 0.\]
    Since \(\widetilde{W_{F}}\) is simply-connected, Lemma~\ref{lemma:triviality-tor-groups} implies that \(H_{1}(W_{F};\C^{\omega}) = 0\). This, together with Equation~\eqref{eq:nullity}, yields the result
    \[
    \new{\eta_F(\omega)+\delta_L(\omega)=\dim H_{1}(M_{L^\#};\C^{\omega})}
    \]
    for~$\omega\neq(1,\dots,1)$.
    
For~$\omega=(1,\dots,1)$, Equation~\eqref{eq:nullity} shows that~$\eta_F(\omega)$ is equal to the dimension of the kernel of the inclusion induced map~$i_*\colon H_1(M_{L^\#};\mathbb{C})\to H_1(W_F;\mathbb{C})$.
Since the homomorphism~$\pi_1(W_F)\to\mathbb{Z}^\mu$ is an isomorphism compatible with the surjective meridional homomorphism~$\pi_1(M_{L^\#})\to\mathbb{Z}^\mu$,
the map~$i_*$ is surjective onto~$H_1(W_F;\mathbb{C})\simeq\mathbb{C}^\mu$. \new{This yields
\[
\eta_F(\omega)+\delta_L(\omega)=\dim H_{1}(M_{L^\#};\C)-\mu
\]
for~$\omega=(1,\dots,1)$.}

\new{Finally, a straightforward Mayer-Vietoris argument shows that if~$\mathbb{T}(ijk)$ denotes the~3-torus corresponding to the 3-colored Borromean ring~$B(ijk)$, then
\[
\dim H_{1}(M_{L}\,\#\,\mathbb{T}(ijk);\C^{\omega})=
\begin{cases}
\dim H_{1}(M_L;\C^{\omega})+3&\text{if~$\omega_i=\omega_j=\omega_k=1$};\\
\dim H_1(M_L;\mathbb{C}^{\omega})+1&\text{else}.
\end{cases}
\]
The proposition now follows from the definition~\eqref{eq:Lsharp} of~$L^\#$ together with the three equations displayed above.}
\end{proof}

\begin{proof}[Proof of Theorem~\ref{thm:extension}]
Given a colored link~$L$ and~$\omega\in\mathbb{T}^\mu$,
\new{we have chosen a meridional homomorphism~$\varphi\colon H_1(M_L)\to\Z^\mu$, thus defining an element~$(M_L,\varphi)\in\Omega^3(\Z^\mu)\simeq\Z^{\mu\choose 3}$. This allowed
us to construct an auxiliary colored link~$L^\#$, and a~$\Z^\mu$-manifold~$(W_F,\Phi)$ with boundary~$(M_{L^\#},\varphi^\#)$.}
Consider the associated signature defect
\[
\dsign_\omega(W_F)\coloneqq\sign_\omega(W_F)-\sign(W_F)\,.
\]
By~\cite[Corollary~2.11]{CNT}, this number only depends on~$\partial W_F=M_{L^\#}=X_{L^\#}\cup -P(L^\#)$ if~\(\omega\in\Tm\)
\new{(and possibly on the choice of the meridional homomorphism).}
Moreover, this proof easily extends to general~\(\omega\in\mathbb{T}^{\mu}\) since the arguments hold for any twisted coefficient system, and if the twisted coefficient system is trivial, then the signature defect vanishes
by definition.

\new{
As explained in Appendix~\ref{ap:plumbed}, this invariant actually coincides with the opposite of the~$\rho$-invariant~$\rho(M_{L^\#},\chi_\omega\circ\varphi^\#)$, where~$\chi_\omega\colon\Z^\mu\to S^1$
is the homomorphism determined by~$\chi_\omega(t_i)=\omega_i$ (see e.g.~\cite[Proposition~4.1]{CNT}).
Moreover, this invariant is additive under connected sum (this follows from~\cite[Theorem~3.9]{toffoli}), and vanishes on 3-manifolds
of the form~$\Sigma\times S^1$ with~$\Sigma$ a closed oriented surface. In particular, it vanishes on the~3-torus~$M_{B(ijk)}$
associated with the 3-colored Borromean rings~$B(ijk)$. Hence, we get
\[
-\dsign_\omega(W_F)=\rho(M_{L^\#},\chi_\omega\circ\varphi^\#)=\rho(M_{L},\chi_\omega\circ\varphi)\,,
\]
which does not depend on the choice of~$\varphi$: this is one of our
technical results, namely Corollary~\ref{cor:rho-ind}.}
Since~$\dsign_\omega(W_F)=\sigma_F(\omega)-\sign(W_F)$ by definition,
it only remains to check that the untwisted signature~$\sign(W_F)$ only depends on~$L$.

This can be verified by applying the Novikov-Wall theorem to the decomposition
\[
W_{F}=V_{F}\cup_{P(F)}Y_{F}\,.
\]
Since the inclusion induced map~$H_2(\partial V_F)\to H_2(V_F)$ is surjective (see e.g. the proof of~\cite[Proposition~3.3]{CNT}) the intersection form on~$H_2(V_F)$ vanishes, and so does~$\sign(V_F)$. Also, we have~$\sign(Y_F)=0$ by Lemma~\ref{lem:Y}.
Hence, the signature of~$W_F$ coincides with the Maslov index associated to this gluing.
Using standard techniques, one easily checks that the three associated Lagrangians are
fully determined by the linking numbers \new{of~$L^\#$.
The fact that the resulting Maslov index only depends on (the linking numbers of)~$L$
is a consequence of Lemma~\ref{lem:Maslov-ind}, so~$\sigma_F$ is an invariant.}

Let us now assume that~$\omega$ lies in~$\Tm$. Then,
the manifold~$W_F$ is obtained by gluing~$V_F$ and~$Y_F$
    along the plumbed 3-manifold~$P(F)$, whose boundary~$\partial P(F)$
    is easily seen to be~$\mathbb{C}^\omega$-acyclic.
    As a consequence, Novikov-Wall additivity applies, and we get
    \[
    \sigma_F(\omega)=\sign_\omega(W_F)=\sign_\omega(V_F)+\sign_\omega(Y_F)
    \]
for all~$\omega\in\mathbb{T}^\mu_*$. Since~$\sign_\omega(Y_F)=0$ by Lemma~\ref{lem:Y} and~$\sign_\omega(V_F)=\sigma_L(\omega)$ by definition, we recover the
equality~$\sigma_F(\omega)=\sigma_L(\omega)$ for all~$\omega\in\mathbb{T}^\mu_*$.

We now turn to the nullity. Since $M_L$ only depends on $L$, Proposition \ref{prop:nullity} \new{and Lemma~\ref{lem:kernel-dim}} immediately imply that $\eta_F(\omega)$ is an invariant for all $\omega\in\mathbb{T}^\mu$.
    Let us finally assume that $\omega$ belongs to $\mathbb{T}^\mu_*$. In that case, the spaces $P(L)$ and $P(F)$ are clearly $\mathbb{C}^\omega$-acyclic,
    see e.g. the proof of Lemma~\ref{lem:P(G)}.
Since~$M_L$ is obtained by gluing~$X_L$ to~$P(L)$ along the~$\mathbb{C}^\omega$-acyclic space~$\partial X_L$,
the Mayer-Vietoris exact sequence implies that the inclusion of~$X_L$ in~$M_L$ induces isomorphisms in homology with coefficients in~$\mathbb{C}^\omega$. Hence,
Proposition \ref{prop:nullity} yields
\[
\eta_{F}(\omega)=\dim H_{1}(M_L; \C^{\omega})=\dim H_{1}(X_L; \C^{\omega})=\eta_L(\omega)\,.
\]
This completes the proof.
\end{proof}

Since~$\sigma_F$ and~$\eta_F$ are invariants of $L$ that extend~$\sigma_L$ and~$\eta_L$, we can denote them by these same symbols
\[
\sigma_L\colon\mathbb{T}^\mu\longrightarrow\mathbb{Z}\quad\text{and}\quad \eta_L\colon\mathbb{T}^\mu\longrightarrow\mathbb{Z}\,.
\]

The extension of these invariants raises a natural question,
namely: do the properties of the original invariant propagate to these extended versions~?
For example, Proposition~2.5 of~\cite{C-F} states that
if a~$\mu$-colored link~$L'$ is obtained from a~$(\mu+1)$-colored link~$L$ by identifying the colors of its sublinks~$L_\mu$ and~$L_{\mu+1}$, then the corresponding signatures and nullities are related by
\begin{align*}
\sigma_{L'}(\omega_1,\dots,\omega_\mu)&=\sigma_{L}(\omega_1,\dots,\omega_\mu,\omega_\mu)-\lk(L_\mu,L_{\mu+1})\,,\\
\eta_{L'}(\omega_1,\dots,\omega_\mu)&=\eta_{L}(\omega_1,\dots,\omega_\mu,\omega_\mu)\,
\end{align*}
for all~$(\omega_1,\dots,\omega_\mu)\in\mathbb{T}_*^\mu$.
We expect these formulas to extend to~$\mathbb{T}^\mu\setminus\{(1,\dots,1)\}$.
Also, Theorem~4.1 of~\cite{C-F} asserts that the signature and nullity of a~$\mu$-colored link~$L$ are piecewise continuous along strata of~$\mathbb{T}_*^\mu$ defined via the Alexander ideals of~$L$.
Once again, we expect such a result to hold using what could be
described as {\em multivariable Hosokawa ideals\/}. (We refer to Remark~\ref{rem:B-Z} for the appearance of the classical Hosokawa polynomial~\cite{Hos} in our theory.)
Finally, it is known the signature and nullity are invariant under concordance when restricted to some explicit dense subset of~$\mathbb{T}_*^\mu$, see~\cite[Corollary~3.13]{CNT}, a result that we
also expect to hold for the extended signatures.

However, we shall postpone the answer to these questions to later
study~\new{\cite{CFP}}, focusing in the present article on the object of its title.

\subsection{Torres formulas for the signature}
\label{sub:Torres-sign}

The aim of this section is to relate the signature of a colored link~$L=L_1\cup L_2\cup\dots\cup L_\mu=:L_1\cup L'$ evaluated at~$\omega=(1,\omega')\in\mathbb{T}^\mu$ with the signature of~$L'$ evaluated at~$\omega'\in\mathbb{T}^{\mu-1}$. On any given example, the techniques used below allow us to find a relation. However, such a fully general Torres formula does not admit an easily presentable closed form (see Remark~\ref{rem:Torres-general} below). For this reason, we shall make several natural assumptions.

First of all, we restrict ourselves to~$\omega'\in\mathbb{T}_*^{\mu-1}$.
Also, we assume that~$L$ belongs to one of the following three classes (which include all ordered links):
\begin{enumerate}
    \item oriented links (i.e. 1-colored links);
    \item $\mu$-colored links~$L=L_1\cup L'$ with~$\mu\ge 2$ and~$\lk(K,K')=0$ for all~$K\subset L_1$ and~$K'\subset L'$;
    \item $\mu$-colored links~$L=L_1\cup L'$ with no~$K\subset L_1$ such that~$\lk(K,K')=0$ for all~$K'\subset L'$.
\end{enumerate}

To state the corresponding Torres formulas, we need several preliminary
notations.
Let us assume that~$L=L_1\cup L'$ is \emph{algebraically split},
i.e. that~$\lk(K,K')=0$ for all~$K\subset L_1$ and~$K'\subset L'$,
and fix~$\omega=(1,\omega')\in\mathbb{T}^\mu$ with~$\omega'\in\mathbb{T}^{\mu-1}$.
Then, we have~$H_1(\partial\nu(L);\mathbb{C}^\omega)=H_1(\partial\nu(L_1);\mathbb{C})$,
so this space admits the natural basis~$\{m_K,\ell_K\}_{K\subset L_1}$, with~$m_K$ a meridian of~$\partial\nu(K)$ and~$\ell_K$ a longitude of~$\partial\nu(K)$, chosen so that~$\lk(L_1,\ell_K)=0$
and~$m_K\cdot\ell_K=-1$ (recall Remark~\ref{rem:NW}).
Since the kernel of the inclusion induced map~$H_1(\partial\nu(L);\mathbb{C}^\omega)\to H_1(X_L;\mathbb{C}^\omega)$
is half-dimensional, it is freely generated by~$n\coloneqq|L_1|$ elements~$x_1,\dots,x_n$ that can be expressed as
\begin{equation}
\label{eq:xj}
    x_j=\sum_{K\subset L_1}\alpha_{jK}m_K+\beta_{jK}\ell_K
\end{equation}
for some complex numbers~$\alpha_{jK},\beta_{jK}$ depending on~$L$ and on~$\omega$. Let~$\mathcal{F}=(f_{ij})$ be the~$n\times n$ matrix defined by
\begin{equation}
    \label{eq:defF}
    f_{ij}=-\sum_{K\subset L_1}\alpha_{iK}\overline{\beta_{jK}}\,.
\end{equation}
The kernel generated by the~$x_j$s being isotropic with respect to the sesquilinear intersection form, we have~$x_i\cdot x_j=0$ for all~$1\le i,j\le n$, implying that~$F$ is a Hermitian matrix.

\medskip

We are finally ready to state the main result of the section:
three Torres-type formulas for the three cases displayed above.

\begin{theorem}
\label{thm:Torres}
\begin{enumerate}
    \item If~$L$ is a (1-colored) oriented link, then~$\sigma_L(1)=\sgn(\mathit{Lk}_L)$, with~$\mathit{Lk}_L$ the linking matrix defined by~\eqref{eq:Lk}.
    \item If~$L=L_1\cup L'$ is an algebraically split~$\mu$-colored link with~$\mu\ge 2$, then for any~$\omega'\in\mathbb{T}_*^{\mu-1}$, we have~$\sigma_L(1,\omega')=\sigma_{L'}(\omega')+\sign(\mathcal{F})$, with~$\mathcal{F}$ the Hermitian matrix defined by~\eqref{eq:defF}.
    \item If $L=L_1\cup L'$ is a $\mu$-colored link with no~$K\subset L_1$ such that~$\lk(K,K')=0$ for all~$K'\subset L'$, then~$\sigma_L(1,\omega')=\sigma_{L'}(\omega')$ for all~$\omega'\in\mathbb{T}_*^{\mu-1}$.
\end{enumerate}
\end{theorem}

Let us point out that this result
immediately implies the following (general) Torres-type formula in the case where~$L_1$ is a knot.

\begin{corollary}
\label{cor:Torres-sign}
Let~$L=L_1\cup L'$ be a~$\mu$-colored link with~$L_1$ a knot. If~$\mu=1$, then~$\sigma_L(1)=0$. If~$\mu\ge 2$, then for all~$\omega'\in\mathbb{T}^{\mu-1}_*$, we have
\[
\sigma_L(1,\omega')=
\begin{cases}
\sigma_{L'}(\omega')-\sgn(\alpha_L(\omega')\overline{\beta_L(\omega')})&\text{if $\lk(L_1,K')=0$ for all~$K'\subset L'$};\\
\sigma_{L'}(\omega')&\text{else},
\end{cases}
\]
where~$\alpha_L(\omega'),\beta_L(\omega')\in\mathbb{C}$ are such that~$\alpha_L(\omega')m+\beta_L(\omega')\ell$ generates
the kernel of the inclusion induced map~$H_1(\partial\nu(L_1);\mathbb{C}^\omega)\to H_1(X_L;\mathbb{C}^\omega)$.\qed
\end{corollary}


Before starting the proof of Theorem~\ref{thm:Torres}, several remarks are in order.

\begin{remark}
 \label{rem:slope}
 \begin{enumerate}
     \item In theory, the matrix~$\mathcal{F}$ appearing in the algebraically split case can be computed from a diagram for the link~$L$. Indeed, one can first compute the Wirtinger presentation of the link group from the diagram, then use Fox calculus
     to determine the homology group~$H_1(X_L;\mathbb{C}^\omega)$, and eventually
     find a basis of the kernel in~$H_1(\partial X_L;\mathbb{C}^\omega)$ of the inclusion induced map.
\item If~$L_1$ is a knot, then much more can be said.
In such a case indeed,
following the terminology of~\cite{DFL}, the correction term~$-\sgn(\alpha_L(\omega')\overline{\beta_L(\omega')})$ is nothing but the sign of the \emph{slope}
\[
(L_1/L')(\omega'):=-\frac{\alpha_L(\omega')}
{\beta_L(\omega')}\in\mathbb{R}\cup\{\infty\}\,,
\]
with the convention that~$\sgn(\infty)=0$. By~\cite[Theorem~3.2]{DFL}, it can be computed via
\begin{equation}
    \label{eq:slope}
    (L_1/L')(\omega')=-\frac{\frac{\partial\nabla_L}{\partial t_1}(1,\sqrt{\omega'})}{2\nabla_{L'}(\sqrt{\omega'})}
\end{equation}
as long as this fraction makes sense (i.e. as long as both the numerator and denominator do not both vanish), where~$\nabla$ stands
for the Conway function.
\item If~$L_1$ is a knot, then the resulting formula (stated in the introduction as Theorem~\ref{thm:intro3}) should be compared with~\cite[Lemma~4.9]{DFL} which deals with the
`literal' extension of the signature.
The latter statement illustrates how this `naive' extension is in general not well-defined.
 \end{enumerate}
\end{remark}

\begin{example}
\label{ex:Whitehead}
Consider the links~$L=L(k)=L_1\cup L_2$ given in Figure~\ref{fig:link-Lk}, and let us assume~$k\neq 0$ (i.e. that~$L(k)$ is non-trivial).
As was computed in Example~\ref{ex:twist}, we have~$\sigma_L(\omega)=\sgn(k)$ for all~$\omega\in\mathbb{T}^2_*$.
Since~$\lk(L_1,L_2)=0$ and~$L_2=L'$ is a trivial knot, Corollary~\ref{cor:Torres-sign} leads to
\[
\sigma_L(1,\omega')=-\sgn(\alpha_L(\omega')\overline{\beta_L(\omega')})=\sgn((L_1/L')(\omega'))
\]
for all~$\omega'=\omega_2\neq 1$.
The well-known value~$\nabla_{L'}(t)=(t-t^{-1})^{-1}$ 
together with Equation~\eqref{eq:nabla-twist} for~$\nabla_{L}$
then enable us to compute
\begin{equation}
\label{eq:slopeL(k)}
(L_1/L')(\omega')=-k\left(\sqrt{\omega'}-\frac{1}{\sqrt{\omega'}}\right)^2=4k\sin(\pi\theta)^2
\end{equation}
for~$\omega'=e^{2i\pi\theta}$. Hence, we have ~$\sigma_L(\omega)=\sgn(k)$ for all~$\omega\in\mathbb{T}^2\setminus\{(1,1)\}$: on these examples, the signature extends continuously across the axes.
\end{example}

\begin{example}
\label{ex:Torres-sign-T}
In the case of the torus link~$L=T(2,2\ell)$, Theorem~\ref{thm:torres-formula-nullity} simply stated that the signature extends to~$\sigma_L(1,\omega)=\sigma_L(\omega,1)=0$ for all~$\omega\in\mathbb{T}^1_*$. This extension is trivial, but nevertheless very natural.
Indeed, by Example~\ref{ex:torus},
it coincides with the average of the limits on either sides of the axes.
\end{example}

\new{
\begin{example}
    \label{ex:Bor}
    Let~$B=L_1\cup L'$ denote the 3-colored Borromean rings. Since it is amphicheiral, its (non-extended) signature vanishes on~$\mathbb{T}^3_*$. Since~$\nabla_B=(t_1-t_1^{-1})(t_2-t_2^{-1})(t_3-t_3^{-1})$ while~$\nabla_{L'}=0$, Equation~\eqref{eq:slope} implies that its slope is infinite. By Corollary~\ref{cor:Torres-sign}, the extended signature~$\sigma_B(1,\omega')$ vanishes as well for all~$\omega'\in\mathbb{T}^2_*$.
\end{example}}

\begin{proof}[Proof of Theorem~\ref{thm:Torres}]
Let us start with an arbitrary~$\mu$-colored link~$L=L_1\cup L_2\cup\dots\cup L_\mu=L_1\cup L'$ and an element~$\omega=(1,\omega')$
of~$\mathbb{T}^\mu$ with~$\omega'\in\mathbb{T}_*^{\mu-1}$.
\new{We fix a meridional homomorphism~$\varphi\colon H_1(M_L)\to\Z^\mu$, hence defining~$\mu_L=(M_L,\varphi)\in\Z^{\mu\choose 3}$ and~$L^\#=(L^\#)_1\cup(L^\#)'$
as in~\eqref{eq:Lsharp}. Note that~$(L^\#)_1$ is the distant union of~$L_1$ with a trivial 1-colored link, while~$(L^\#)'$ is the distant union of~$L'$ with a trivial link and Borromean rings.}
Let~$F=F_1\cup F'$ be a surface in~$B^4$ bounding~\new{$L^\#=(L^\#)_1\cup(L^\#)'$}, obtained by pushing a totally connected C-complex
inside~$B^4$,
and let~$W_F=V_F\cup Y_F$ and~$W_{F'}=V_{F'}\cup Y_{F'}$ be the corresponding~$4$-manifolds (recall Section~\ref{sub:extension}).
The idea is now to apply the Novikov-Wall theorem to the decompositions:
\begin{enumerate}
    \item $V_{F'}=V_{F}\cup \nu(F^\circ_1)$, yielding~$\sign_\omega(V_F)=\sign_{\omega'}(V_{F'})$;
    \item $W_{F'}=V_{F'}\cup Y_{F'}$, yielding~$\sign_{\omega'}(W_{F'})=\sign_{\omega'}(V_{F'})$;
    \item $W_F=V_F\cup Y_F$, yielding~$\sign_{\omega}(W_F)=\sign_{\omega}(V_{F})+\sign_{\omega}(Y_{F})+\mathcal{M}$ for some Maslov index~$\mathcal{M}$.
\end{enumerate}
Since we know that~$\sign_\omega(Y_F)=0$ by
Lemma~\ref{lem:Y}, these three claims imply
the equality
\[
\sigma_L(\omega)=\sigma_{L'}(\omega')+\mathcal{M}\,.
\]
 
 \medskip
 
We start with the first claim, namely the fact that the Novikov-Wall theorem
applied to the decomposition~$V_{F'}=V_{F}\cup \nu(F^\circ_1)$
yields to equality~$\sign_\omega(V_F)=\sign_{\omega'}(V_{F'})$.
First note that if~$\mu=1$, then this amounts to proving that~$\sign(V_F)$ vanishes, a well-known fact (see e.g.~\cite[Proposition~3.3]{CNT}). Therefore, we can assume~$\mu\ge 2$.
Since~$F^\circ_1$ is a surface with boundary, the 4-manifold~$\nu(F^\circ_1)\simeq F^\circ_1\times D^2$ has the homotopy type of a 1-dimensional CW-complex, and its signature vanishes.

To compute the correction term, first note that the 3-manifold~$M_1\coloneqq V_F\cap\nu(F^\circ_1)$ is equal to~$F_1^\circ\times S^1$, with boundary~$\Sigma\coloneqq\partial\nu(\new{(L^\#)_1})\cup\bigsqcup_e T_e$, where~$\{T_e\}_e$ denotes the
tori corresponding to the intersections of~$F_1$ with the other surfaces.
Since~$\omega'$ belongs to~$\mathbb{T}^{\mu-1}_*$, these tori are~$\mathbb{C}^\omega$-acyclic, leading to~$H_1(\Sigma;\mathbb{C}^\omega)=H_1(\partial\nu(\new{(L^\#)_1});\mathbb{C}^\omega)$.
Clearly, this space is freely generated by~$\{m_K,\ell_K\}_{K\in\new{\mathcal{K}^\#_1}}$
with indices ranging over the set
\begin{equation}
\label{eq:Ksharp}
\new{\mathcal{K}^\#_1}=\{K\subset \new{(L^\#)_1}\,|\,\omega_2^{\lk(K,\new{(L^\#)_2})}\cdots\omega_\mu^{\lk(K,\new{(L^\#)_\mu})}=1\}\,.
\end{equation}

Now, observe that since we assumed~$F$ connected and~$\mu\ge 2$, the surface~$F_1$ intersects the rest of the bounding surface, so~$H_0(F_1^\circ;\mathbb{C}^\omega)$ vanishes.
By the K\"unneth formula, we get~$H_1(M_1;\mathbb{C}^\omega)\simeq H_1(F_1^\circ;\mathbb{C}^\omega)$. This implies that the meridians~$\{m_K\}_{K\in\mathcal{K}_1}$ lie in
the kernel of the inclusion induced map~$H_1(\Sigma;\mathbb{C}^\omega)\to H_1(M_1;\mathbb{C}^\omega)$. Since the dimension of this kernel is equal to the cardinal of~$\new{\mathcal{K}_1^\#}$, these meridians freely generate this kernel.

To determine the second Lagrangian, observe that since~$\nu(F_1)$ is homeomorphic to~$F_1\times D^2$, we have
\[
M_2\coloneqq\partial\nu(F_1)\setminus M_1\simeq\left(\nu(\new{(L^\#)_1})\cup(F_1\times S^1)\right)\setminus(F_1^\circ\times S^1)=\nu(\new{(L^\#)_1})\cup\bigsqcup_e(D^2\times S^1)\,,
\]
where the solid tori are indexed by the double points in~$F_1$.
Since~$\omega'$ belongs to~$\mathbb{T}^{\mu-1}_*$, these tori are~$\mathbb{C}^\omega$-acyclic, and we have~$H_1(M_2;\mathbb{C}^\omega)=H_1(\nu(\new{(L^\#)_1});\mathbb{C}^\omega)$,
a space freely generated by~$\{\ell_K\}_{K\in\new{\mathcal{K}_1^\#}}$. As a consequence, 
the Lagrangian given by the kernel of the inclusion induced map~$H_1(\Sigma;\mathbb{C}^\omega)\to H_1(M_2;\mathbb{C}^\omega)$
admits the basis~$\{m_K\}_{K\in\new{\mathcal{K}_1^\#}}$, and coincides with
the first Lagrangian. Therefore, the Maslov correction term vanishes, completing the proof of the first claim.

\medskip

\new{The second claim is clear: since~$\omega'$ belongs to~$\mathbb{T}^{\mu-1}_*$,
the equality~$\sigma_{\omega'}(W_{F'})=\sigma_{\omega'}(V_{F'}) $ follows from Theorem~\ref{thm:extension}.}

\medskip

We now turn to the third and last step, i.e. the application of the Novikov-Wall theorem
to the decomposition~$W_F=V_F\cup Y_F$ along~$P(F)$.
Since the orientation on~$W_F$ induces an orientation on~$V_F$ and~$Y_F$ such that~$\partial Y_F=P(F)\cup -P(\new{L^\#})$ and~$\partial V_F=X_{\new{L^\#}}\cup-P(F)$, we have
\[
\sign_\omega(W_F)=\sign_\omega(V_F)+\sign_\omega(Y_F)+\mathit{Maslov}(\new{\mathcal{L}^\#_-,\mathcal{L}^\#_0,\mathcal{L}^\#_+})\,,
\]
where~\new{$\mathcal{L}^\#_-$ (resp.~$\mathcal{L}^\#_0,\mathcal{L}^\#_+$}) denotes the kernel of the inclusion induced maps from~$H_1(\partial X_{\new{L^\#}};\mathbb{C}^\omega)$ to~$H_1(P(\new{L^\#});\mathbb{C}^\omega)$ (resp.~$H_1(P(F);\mathbb{C}^\omega)$,~$H_1(X_{\new{L^\#}};\mathbb{C}^\omega)$).
As above \new{(recall in particular~\eqref{eq:Ksharp})}, the space~$H_1(\partial X_{\new{L^\#}};\mathbb{C}^\omega)$ is freely generated by~$\{m_K,\ell_K\}_{K\in\new{\mathcal{K}^\#_1}}$,
and it remains to compute the three Lagrangians.

\new{Before doing so, first note that~$\mathcal{K}^\#_1$ is the disjoint union of~$\mathcal{K}_1$ with the components of the trivial link~$(L^\#)_1\setminus L_1$, where
\[
\mathcal{K}_1=\{K\subset L_1\,|\,\omega_2^{\lk(K,L_2)}\cdots\omega_\mu^{\lk(K,L_\mu)}=1\}\,.
\]
This leads to the equality
\begin{equation*}
    H_1(\partial X_{L^\#};\mathbb{C}^\omega)=H_1(\partial X_{L};\mathbb{C}^\omega)\oplus\bigoplus_{K\in (L^\#)_1\setminus L_1}\left(\C m_K\oplus\C\ell_K\right)\,,
\end{equation*}
which will allow us to compare the Lagrangians~$\mathcal{L}^\#_-,\mathcal{L}^\#_0,\mathcal{L}^\#_+$ with their counterparts~$\mathcal{L}_-,\mathcal{L}_0,\mathcal{L}_+$ defined with~$L$ instead of~$L^\#$.}

By Lemma~\ref{lem:P(G)} applied to~$P(L)$, we know that~$\mathcal{L}_-$ admits
the basis~$\{c_K\}_{K\in\mathcal{K}_1}$, where
\[
c_K=\begin{cases}
\ell_K&\text{if~$\lk(K,K')=0$ for all~$K'\subset L'$;}\cr
m_K&\text{else.}
\end{cases}
\]
\new{The same argument applied to~$P(L^\#)$ yields the relation
\[
\mathcal{L}_-^\#=\mathcal{L}_-\oplus\bigoplus_{K\in (L^\#)_1\setminus L_1}\C\ell_K\,.
\]
S}ince~$F$ is connected, Lemma~\ref{lem:P(G)} applied to~$P(F)$ shows that~$\new{\mathcal{L}^\#_0}$ is freely generated by~$\{m_K\}_{K\in\new{\mathcal{K}^\#_1}}$ if~$\mu\ge 2$.
\new{Writing~$\mathcal{L}_0$ for the subspace freely generated by~$\{m_K\}_{K\in\mathcal{K}_1}$, we now get
\[
\mathcal{L}_0^\#=\mathcal{L}_0\oplus\bigoplus_{K\in (L^\#)_1\setminus L_1}\C m_K\,.
\]
F}or~$\mu=1$ \new{(in which case~$L^\#=L$)}, one last use of Lemma~\ref{lem:P(G)} shows that~$\new{\mathcal{L}^\#_0=\mathcal{L}_0}$ admits the vectors~$\sum_K\ell_K$ and~$\{m_K-m_{K_0}\}_{K\subset L}$ as
a basis, with~$K_0$ any fixed component of~$L=L_1$.
By definition, the third Lagrangian~$\mathcal{L}_+$
admits a basis~$\{x_j\}_j$, which can be described in coordinates as in Equation~\eqref{eq:xj} above.
\new{Finally, an easy Mayer-Vietoris argument shows that the Lagrangian~$\mathcal{L}^\#_+$ splits as
\begin{align*}
\mathcal{L}^\#_+&=\ker\left(H_1(\partial X_{L^\#};\mathbb{C}^\omega)\to H_1(X_{L^\#};\mathbb{C}^\omega)\right)\\
&=\ker\left(H_1(\partial X_{L};\mathbb{C}^\omega)\to H_1(X_{L};\mathbb{C}^\omega)\right)\oplus \ker\left(H_1(\partial X_{(L^\#)_1\setminus L_1};\mathbb{C})\to H_1(X_{L^\#\setminus L};\mathbb{C}^\omega)\right)\\
&=\mathcal{L}_+\oplus\bigoplus_{K\in (L^\#)_1\setminus L_1}\C m_K\,,
\end{align*}
where the last equality follows from the definition of~$\mathcal{L}_+$ together with the fact that the slope of the Borromean rings is infinite (recall Example~\ref{ex:Bor}).
By the four equalities displayed above, we get
\[
\mathit{Maslov}(\mathcal{L}^\#_-,\mathcal{L}^\#_0,\mathcal{L}^\#_+)=\mathit{Maslov}(\mathcal{L}_-,\mathcal{L}_0,\mathcal{L_+})+\sum_{K\in (L^\#)_1\setminus L_1}\underbrace{\mathit{Maslov}(\C\ell_K,\C m_K,\C m_K)}_{=0}\,,
\]
and we are left with the computation of~$\mathit{Maslov}(\mathcal{L}_-,\mathcal{L}_0,\mathcal{L_+})$.}

This is the point where a presentable closed formula becomes out of reach, and we focus on the three cases as in the statement of the theorem.
Let us first assume that~$\mu=1$. In this case, the computation
of the Maslov index can be performed as in the proof of~\cite[Lemma~5.4]{NP}, leading to~$\mathit{Maslov}(\mathcal{L}_-,\mathcal{L}_0,\mathcal{L}_+)=\sign(Lk_L)$. Let us now assume that~$L$ satisfies the condition of the third point. This precisely means that the Lagrangians~$\mathcal{L}_-$ and~$\mathcal{L}_0$ coincide, leading to the Maslov index vanishing.
Let us finally assume that~$L=L_1\cup L'$ is algebraically split with~$\mu\ge 2$, and recall the notation of Equation~\eqref{eq:xj}.
As explained in Section~\ref{sub:NW}, the Maslov index
is given by the signature of the form~$f$ on~$(\mathcal{L}_-+\mathcal{L}_0)\cap\mathcal{L}_+$ defined as follows:
if~$a=a_-+a_0\in(\mathcal{L}_-+\mathcal{L}_0)\cap\mathcal{L}_+$ with~$a_-\in\mathcal{L}_-,a_0\in\mathcal{L}_0$
and~$b\in(\mathcal{L}_-+\mathcal{L}_0)\cap\mathcal{L}_+$, then~$f(a,b)=a_0\cdot b$. Since~$L=L_1\cup L'$ is algebraically split,~$\mathcal{L}_-$ is freely generated by~$\{\ell_K\}_{K\subset L_1}$,~$\mathcal{L}_0$ is freely generated by~$\{m_K\}_{K\subset L_1}$, and we have~$(\mathcal{L}_-+\mathcal{L}_0)\cap\mathcal{L}_+=\mathcal{L}_+$.
Therefore, we get
\[
f(x_i,x_j)=\Big(\sum_{K\subset L_1}\alpha_{iK}m_K\Big)\cdot\Big(\sum_{K\subset L_1}\alpha_{jK}m_K+\beta_{jK}\ell_K\Big)=
\sum_{K\subset L_1}\alpha_{iK}\overline{\beta_{jK}}
\,(\overbrace{m_K\cdot\ell_K}^{-1})=f_{ij}\,,
\]
using the third part of Remark~\ref{rem:NW}. This concludes the proof.
\end{proof}

\begin{remark}
\label{rem:Torres-general}
\begin{enumerate}
    \item There is no obstacle to relating~$\sigma_L(\omega)$ and~$\sigma_{L'}(\omega')$ in the general setting of an arbitrary colored link~$L=L_1\cup L'$. Indeed,
    the proof above leads to the formula
    \[
\sigma_L(\omega)=\sigma_{L'}(\omega')+\mathit{Maslov}(\mathcal{L}_-,\mathcal{L}_0,\mathcal{L}_+)\,,
\]
where~$\mathcal{L}_-,\mathcal{L}_0,\mathcal{L}_+$ are explicit Lagrangians of an explicit symplectic vector space. The issue is that, outside of the three cases highlighted in Theorem~\ref{thm:Torres}, there does
not seem to be a self-contained closed formula for this Maslov index.
    \item The same can be said of the restriction to~$\omega'\in\mathbb{T}_*^{\mu-1}$: Lemma~\ref{lem:P(G)} \new{holds for}
    arbitrary values of~$\omega\in\mathbb{T}^{\mu}$,
    leading to formulas of the form displayed above
    valid for any~$\omega=(1,\omega')$ with~$\omega'\in\mathbb{T}^{\mu-1}$.
   Once again, it is not difficult to give explicit description of the corresponding Lagrangian subspaces, but their Maslov index does not admit
   a simple closed formula in general.
\end{enumerate}
\end{remark}

\subsection{Torres formulas for the nullity}
\label{sub:Torres-null}

As we did for the signature in the previous section, we now want to relate the nullity of a~$\mu$-colored link \(L=L_{1}\cup\ldots\cup L_{\mu}\) at \(\omega=(1,\omega')\) to the nullity of \(L'=L_{2}\cup\ldots\cup L_{\mu}\) at \(\omega'\).
As in Theorem~\ref{thm:Torres}, we will assume~$\omega'\in\mathbb{T}^{\mu-1}_{*}$, and
 will restrict our attention to the same three cases (the second case being slightly less general in the statement below).

\begin{theorem}
\label{thm:torres-formula-nullity}
    \begin{enumerate}
    \item If~$L$ is a (1-colored) oriented link, then~$\eta_L(1)=\nul(\mathit{Lk}_L)-1$.
    \item If~$L=L_1\cup L'$ is algebraically split with~$\mu\ge 2$ and~$L_1$ is a knot, then for all~\(\omega'\in\mathbb{T}^{\mu-1}_{*}\),
\[
		\eta_L(1,\omega')=
		\begin{cases}
			\eta_{L'}(\omega')+1&\text{if $(L_1/L')(\omega')=0$}\\
			\eta_{L'}(\omega')-1&\text{if $(L_1/L')(\omega')=\infty$}\\
			\eta_{L'}(\omega')&\text{else.}
		\end{cases}
  \]
        \item If there is no \(K\subset L_{1}\) such that \(\lk(K,K')=0\) for all \(K'\subset L'\), then we have
    \[
    \eta_L(1,\omega')=\eta_{L'}(\omega')-|L_1|+\sum_{K\subset L_1}\sum_{K'\subset L'} |\lk(K,K')|
    \]
    for all~\(\omega'\in\mathbb{T}^{\mu-1}_{*}\),
    where~$|L_1|$ denotes the number of components of~$L_1$, and the sums run over all components~$K$ of~\(L_{1}\) and~$K'$ of~\(L'\).
    \end{enumerate}
\end{theorem}

\begin{example}
\label{ex:null-L(k)}
Consider the family of twist links~$L(k)=L_1\cup L_2$ of Figure~\ref{fig:link-Lk}. Since these links are algebraically split
with unknotted components and slope given by~\eqref{eq:slopeL(k)}, we find that~$\eta_{L(k)}$ extends continuously to the constant function equal to~$\delta_{k0}$ on the whole of~$\mathbb{T}^2\setminus\{(1,1)\}$.
\end{example}

\begin{example}
 \label{ex:null-T}
 Consider the torus link~$L=T(2,2\ell)$ studied in Example~\ref{ex:torus}, assuming~$\ell\neq 0$.
 By the third case of Theorem~\ref{thm:torres-formula-nullity},
 we get
 \[
\eta_{L}(1,\omega)=\eta_L(\omega,1)=\vert\ell\vert -1
 \]
 for all~$\omega\in\mathbb{T}^1_*=S^1\setminus\{1\}$.
 Such a value might seem surprising, as the nullity is at most~$1$
 on~$\mathbb{T}^2_*$. As we shall see, such a high number is necessary
 to account for the different values of the limits of the signatures
 when approaching~$1$ from different sides, see Remark~\ref{ex:diff-limits}.
\end{example}

\begin{proof}[Proof of Theorem~\ref{thm:torres-formula-nullity}]
Let~\(L=L_{1}\cup\ldots\cup L_{\mu}=:L_1\cup L'\) be a~$\mu$-colored link and let us fix~\(\omega=(1,\omega')\)~ with~\(\omega'\in\mathbb{T}^{\mu-1}_{*}\).
By Proposition~\ref{prop:nullity}, the nullity of~\(L\) at~\(\omega\) is equal to the dimension of~\(H_{1}(M_{L};\mathbb{C}^\omega)\) if~$\mu\ge 2$, and to~$\dim H_{1}(M_{L};\mathbb{C})-1$ if~$\mu=1$.
Recall also that by Theorem~\ref{thm:extension} and the assumption~\(\omega'\in\mathbb{T}^{\mu-1}_{*}\), we have
\begin{equation}
\label{eq:X_L'}
\eta_{L'}(\omega')=\dim H_{1}(M_{L'}; \C^{\omega'})=\dim H_{1}(X_{L'};\C^{\omega'})\,.
\end{equation}
Hence, we are left with the computation of the difference between the dimensions of~$H_{1}(M_{L};\mathbb{C}^\omega)$
and of~$H_{1}(X_{L'};\mathbb{C}^{\omega'})$.

To do so, we apply the Mayer-Vietoris exact sequence to
the decompositions
\[
X_{L'}=X_{L}\cup_{\partial\nu(L_1)} \nu(L_{1})
\quad\text{and}\quad M_{L}=X_{L}\cup_{\partial\nu(L)} -P(L)\,.
\]
Let us start with the first decomposition, which leads to the exact sequence
\begin{align*}
H_{1}(\partial\nu(L_{1});\mathbb{C}^\omega)&\rightarrow H_{1}(X_{L};\mathbb{C}^\omega)\oplus H_{1}(\nu(L_{1});\mathbb{C}^\omega)\rightarrow H_{1}(X_{L'};\mathbb{C}^{\omega'})\\
&\rightarrow H_{0}(\partial\nu(L_{1});\mathbb{C}^\omega)\rightarrow H_{0}(X_{L};\mathbb{C}^\omega)\oplus H_{0}(\nu(L_{1});\mathbb{C}^\omega)\,.
\end{align*}
Note that both spaces~$H_{0}(\partial\nu(L_{1});\mathbb{C}^\omega)$ and~$H_{0}(\nu(L_{1});\mathbb{C}^\omega)$ have dimension equal to the cardinal of
\[
\mathcal{K}_1=\{K\subset L_1\,|\,\omega_2^{\lk(K,L_2)}\cdots\omega_\mu^{\lk(K,L_\mu)}=1\}\,,
\]
so the last arrow above is injective. By exactness, the second is therefore surjective.
Similarly, the space~$H_{1}(\partial\nu(L_{1});\mathbb{C}^\omega)$ has dimension~$2|\mathcal{K}_1|$, with a natural
basis consisting of the meridiens and longitudes of
elements of~$\mathcal{K}_1$. Writing~$V_m$ and~$V_\ell$ for
the subspaces spanned by these meridiens and longitudes, respectively,
the map induced by the inclusion of~$\partial\nu(L_1)$ in~$\nu(L_1)$ restricts to an isomorphism~$V_\ell\simeq H_{1}(\nu(L_{1});\mathbb{C}^\omega)$ and to the zero map on~$V_m$.
As a consequence, we have the exact sequence
\begin{equation*}
V_{m}\rightarrow H_{1}(X_{L};\mathbb{C}^\omega)\rightarrow H_{1}(X_{L'};\mathbb{C}^{\omega'})\rightarrow 0\,,
\end{equation*}
which together with~\eqref{eq:X_L'}, yields the equality
\begin{equation}
\label{eq:X_L}
\dim H_{1}(X_{L};\mathbb{C}^\omega)=\eta_{L'}(\omega')+|\mathcal{K}_1|-\dim\ker(V_{m}\rightarrow H_{1}(X_{L};\mathbb{C}^\omega))\,.
\end{equation}
Next, consider the decomposition~$M_{L}=X_{L}\cup_{\partial\nu(L)} -P(L)$, which yields the exact sequence
\begin{align*}
0\rightarrow\ker{\iota}&\rightarrow H_{1}(\partial\nu(L);\mathbb{C}^\omega)\xrightarrow{\iota} H_{1}(X_{L};\mathbb{C}^\omega)\oplus H_{1}(P(L);\mathbb{C}^\omega)\rightarrow H_{1}(M_{L};\mathbb{C}^\omega) \cr
&\rightarrow H_{0}(\partial\nu(L);\mathbb{C}^\omega)\rightarrow H_{0}(X_{L};\mathbb{C}^\omega)\oplus H_{0}(P(L);\mathbb{C}^\omega)\rightarrow H_{0}(M_{L};\mathbb{C}^\omega)\rightarrow 0\,.
\end{align*}
Since the Euler characteristic of this sequence is 0, we can write
\[\beta_{1}(M_{L})=\dim\ker(\iota)-\beta_{1}(\partial\nu(L))+\beta_{1}(X_{L})+\beta_{1}(P(L))+\beta_{0}(\partial\nu(L))-\beta_{0}(X_{L})-\beta_{0}(P(L))+\beta_{0}(M_{L})\,,
\]
where~$\beta_i$ denotes the~$i^\mathrm{th}$ Betti number.
Now, we can simplify this equation via the following observations.
\begin{itemize}
    \item We have \(\beta_{0}(X_{L})=\beta_{0}(M_{L})\) since both of these spaces are connected and both are either trivially or non-trivially twisted.
    \item As already mentioned, we have~\(\beta_{0}(\partial\nu(L))=|\mathcal{K}_{1}|\) and~\(\beta_{1}(\partial\nu(L))=2|\mathcal{K}_{1}|\).
    \item The Betti number~\(\beta_{0}(P(L))\) is given by the number of components~\(K\subset L_{1}\) that are algebraically split from~$L'$, i.e. such that~$\lk(K,K')=0$ for all~$K'\subset L'$.
    \item We finally turn to~$\beta_{1}(P(L))$: using the assumption~$\omega'\in\mathbb{T}^{\mu-1}_*$
    together with the arguments of the proof of Lemma~\ref{lem:P(G)},
    we find that each disk in~$P(L)$ contributes the number of punctures on it minus~1 if there are punctures, and contributes~1 otherwise; in other words, we have
    \[
    \beta_{1}(P(L))=\beta_{0}(P(L))+\sum_{K\subset L_1}\Big(\big(\sum_{K'\subset L'} |\lk(K,K')|\big)-1\Big)\,,
    \]
    where the first sum runs over components~$K$ of \(L_{1}\) that are not algebraically split from~$L'$,
    and the second sum runs over all components~$K'$ of~\(L'\).
\end{itemize}
Using these observations together with Equation~\eqref{eq:X_L}, we get the following
general result:
\begin{equation}
\beta_{1}(M_{L})=\eta_{L'}(\omega') +\dim\ker(\iota) - \dim\ker(V_{m}\rightarrow H_{1}(X_{L}))+\sum_{K\subset L_1}\Big(\big(\sum_{K'\subset L'} |\lk(K,K')|\big)-1\Big)\,.\label{eq:general-formula-nullity}\end{equation}
We now consider the three particular cases appearing in the statement.

Let us first assume that for each~\(K\subset L_1\), there exists~$K'\subset L'$ with~$\lk(K,K')\neq 0$. Then, we know by Lemma~\ref{lem:P(G)} that~\(V_{m}\rightarrow H_{1}(P(L);\mathbb{C}^\omega)\) is trivial while~\(V_{\ell}\rightarrow H_{1}(P(L);\mathbb{C}^\omega)\) is an isomorphism. This implies the equality~\(\ker(\iota)=\ker(V_{m}\rightarrow H_{1}(X_{L};\mathbb{C}^\omega))\).  Putting this into formula~\eqref{eq:general-formula-nullity} yields
    \[
    \eta_{L}(\omega)=\beta_{1}(M_{L})=\eta_{L'}(\omega') +\sum_{K\subset L_1}\Big(\big(\sum_{K'\subset L'} |\lk(K,K')|\big)-1\Big)\,
    ,\]
with the first sum now running over all components of~$L_1$. This gives the third case in the statement.

Let us now assume that~$\lk(K,K')=0$ for all components~$K\subset L_1$ and~$K'\subset L'$.
In such a case, we know from Lemma~\ref{lem:P(G)} that~\(V_{\ell}\rightarrow H_{1}(P(L);\mathbb{C}^\omega)\) is trivial while~\(V_{m}\rightarrow H_{1}(P(L);\mathbb{C}^\omega)\) is an isomorphism, leading to~\(\ker(\iota)=\ker(V_{\ell}\rightarrow H_{1}(X_{L};\mathbb{C}^\omega))\).
    Therefore, the general formula~\eqref{eq:general-formula-nullity} simplifies to
    \begin{equation}
    \label{eq:split}
        \beta_{1}(M_{L})=\eta_{L'}(\omega')+\dim\ker(V_{\ell}\rightarrow H_{1}(X_{L};\mathbb{C}^\omega))-\dim\ker(V_m\rightarrow H_{1}(X_{L};\mathbb{C}^\omega))\,,
    \end{equation}
    as the final sum now runs over an empty set.
    In the special case where~\(L_{1}\) is a knot, the dimensions of these kernels are determined by the slope~\((L_{1}/L)(\omega')\) by definition (recall Remark~\ref{rem:slope}).
    This leads to the second case in the statement.
    
Finally, let us assume that~$L=L_1$ is a~1-colored link.
In that case, the link~$L'$ being empty and the coefficients trivial, we have~$\eta_{L'}=\dim H_{1}(X_{L'})=0$
and~\(V_{m}\rightarrow H_{1}(X_{L})\) is an
isomorphism. Therefore, the general formula~\eqref{eq:split} yields
\[
\eta_L(1)=\beta_{1}(M_{L})-1=\dim\ker(V_{\ell}\rightarrow H_{1}(X_{L}))-1\,.
\]
It remains to recall that the morphism~\(V_{\ell}\rightarrow H_{1}(X_{L})\) is presented by the matrix~\(\mathit{Lk}_L\),
leading to the first case in the statement,
and concluding the proof.
\end{proof}

\begin{remark}
\label{rem:Torres-general-nullity}
\begin{enumerate}
    \item As in the case of the signature, there is no obstacle to relating~$\eta_L(\omega)$ and~$\eta_{L'}(\omega')$ in the general setting of an arbitrary colored link~$L=L_1\cup L'$. Indeed,
    the proof above leads to Equation~\eqref{eq:general-formula-nullity},
    where the involved dimensions can be computed via Fox calculus
    on any given example. However, the general case does not yield a tractable closed formula.
    \item 
    Similarly, given any~$\omega'\in\mathbb{T}^{\mu-1}$ (and not necessarily in~$\mathbb{T}_*^{\mu-1}$), one could in theory relate the nullity of~\(L\) at~\((1,\omega')\) to the nullity of~\(L'\) at~$\omega'$. However, that general case does not yield any tractable closed formula, since the homology of the plumbed manifolds is decidedly more complicated in that case.
\end{enumerate}
\end{remark}

\section{Limits of signatures: the 4D approach}
\label{sec:4D}

The aim of this section is to use the 4-dimensional approach of Section~\ref{sec:Torres} to evaluate limits of signatures.
It is divided as follows. In Section~\ref{sub:prelim-lemma},
we give the general strategy together with two preliminary lemmas. Then, the case of the Levine-Tristram signature is studied in Section~\ref{sub:LT},
limits of multivariable signatures of colored links with all variables tending to~$1$ in Section~\ref{sub:lim-mult-1} and 
more general limits
of multivariable signatures in Section~\ref{sub:lim-mult}.
Finally, Section~\ref{sub:comparison} contains a discussion of the comparison of the three and four-dimensional approaches.

\bigbreak

\subsection{Preliminary lemmas, and the general strategy}
\label{sub:prelim-lemma}

The general idea of the 4D approach for evaluating limits of signatures is to apply Lemma~\ref{lemma:limit-signature} to a matrix
of the intersection form on~\(H_{2}(W_{F},\mathbb{C}^{\omega})\),
and then to use the Torres formulas for the signature and nullity.
For this idea to go through, we need to show that every element of~$\mathbb{T}^{\mu}\setminus\{(1,\cdots,1)\}$ admits an open neighborhood~\(U \subset \mathbb{T}^{\mu} \setminus \{(1,\cdots,1)\}\) such
that the intersection form on~\(H_{2}(W_{F},\mathbb{C}^{\omega})\) can be given by a common matrix~\(H_U(\omega)\) for all~$\omega\in U$.
This is a consequence of the following lemma.

\begin{lemma}\label{lemma:representing-int-forms}
  Let~\(\Lambda_{\mu}\) denote the group ring~\(\C[\Z^{\mu}]\), and let~\(Q(\Lambda_{\mu})\) be its fraction field.
  Suppose that~\((W,\psi)\) is a compact connected oriented~\(4\)-manifold over~\(\Z^{\mu}\) with connected boundary, such that the composition
  \[H_1(\partial W) \to H_1(W) \xrightarrow{\psi} \Z^{\mu}\]
  is surjective and~\(H_{1}(W;\Lambda_{\mu}) = 0\).
  Then, for any~\(j=1,\ldots,\mu\),
  there exists a Hermitian matrix~\(H_{j}\) over~\(Q(\Lambda_{\mu})\) such that for any \(\omega \in U_{j}\coloneqq\{\omega \in \mathbb{T}^{\mu} \colon \omega_{j} \neq 1\}\), the intersection form
  \[Q_{\omega} \colon H_{2}(W;\C^{\omega}) \times H_{2}(W;\C^{\omega}) \to \C\]
  is represented by \(H_{j}(\omega)\).
Furthermore, if~$\mu=1$, then~$Q_\omega$ is represented by a Hermitian matrix~$H(\omega)$ for all~$\omega\in S^1$.
\end{lemma}
We defer the proof of Lemma~\ref{lemma:representing-int-forms} to Appendix~\ref{sec:repr-inters-forms}.
We will also need the following lemma.

\begin{lemma}
\label{lem:Alex}
For any~$\mu$-colored link~$L$ and any~$\omega\in\mathbb{T}^\mu\setminus\{(1,\dots,1)\}$, the nullity~$\eta_L(\omega)$ is bounded below by the rank of the Alexander module~$H_1(X_L;\Lambda_\mu)$ of~$L$.  
\end{lemma}
\begin{proof}
Let~$L$ be an arbitrary~$\mu$-colored link, and fix~$\omega\in\mathbb{T}^\mu\setminus\{(1,\dots,1)\}$.
Consider the Universal Coefficient Spectral Sequence from~Theorem~\ref{thm:UCSS}
\[E^{2}_{p,q} = \Tor_{p}^{\Lambda_{\mu}}\left(\C^{\omega}, H_{q}(\new{M_L};\Lambda_{\mu})\right) \Rightarrow H_{p+q}(\new{M_L}; \C^{\omega}).\]
Since~\(H_{0}(\new{M_L};\Lambda_{\mu}) \cong \C\),
this spectral sequence yields an exact sequence
\[\Tor_{2}^{\Lambda_{\mu}}(\C^{\omega},\C) \to \C^{\omega} \otimes_{\Lambda_{\mu}} H_{1}(\new{M_L}; \Lambda_{\mu}) \to H_{1}(\new{M_L};\C^{\omega}) \to \Tor_{1}^{\Lambda_{\mu}}(\C^{\omega},\C) \to
  0\,.\]
Since we assumed~$\omega\neq(1,\dots,1)$, the Tor terms vanish by Lemma~\ref{lemma:triviality-tor-groups} and we have an
isomorphism~$H_1(\new{M_L};\mathbb{C}^\omega) \cong \mathbb{C}^{\omega} \otimes_{\Lambda_\mu} H_1(\new{M_L};\Lambda_{\mu})$.
Therefore, Proposition~\ref{prop:nullity} leads to
\begin{equation*}
\eta_L(\omega)=\dim_\mathbb{C} H_1(M_L;\mathbb{C}^\omega)=\dim_\mathbb{C}\left(\mathbb{C}^{\omega} \otimes_{\Lambda_\mu}H_1(\new{M_L};\Lambda_{\mu})\right)\ge \rank_{\Lambda_\mu} H_1(\new{M_L};\Lambda_{\mu})\,.
\end{equation*}
Finally, since the modules~$H_*(P(L);\Lambda_\mu)$ and~$H_*(\partial\nu(L);\Lambda_\mu)$ are torsion, the Mayer-Vietoris exact sequence for~$\new{M_L}=X_L\cup P(L)$ implies that the ranks of~$H_1(\new{M_L};\Lambda_\mu)$ and~$H_1(X_L;\Lambda_\mu)$ coincide,
concluding the proof.
\end{proof}

Since the (rank of the) Alexander module will appear quite often,
we now fix a notation for it following~\cite{Hillman}.
For any given~$\mu$-colored link~$L$, let us denote by
\[
A(L)\coloneqq H_1(X_L;\Lambda_\mu)
\]
the associated Alexander module over the ring~$\Lambda_\mu$.

\medskip

We are now ready to prove a preliminary version of our main result.

\begin{proposition}
\label{prop:limit}
For any~$\mu$-colored link~$L$, the inequality
\[
\left|\lim_{\omega_1\to 1^\pm}\sigma_L(\omega_1,\omega')-\sigma_{L}(1,\omega')\right|\le
\eta_{L}(1,\omega')-\rank A(L)
\]
holds for all~$\omega'\in\mathbb{T}^{\mu-1}\setminus\{(1,\dots,1)\}$ \new{with at most one coordinate equal to~$1$}.
\end{proposition}
\begin{proof}
Fix an arbitrary~$\mu$-colored link~$L$, and in case~$\mu>1$, some~$\omega'\in\mathbb{T}^{\mu-1}\setminus\{(1,\dots,1)\}$ \new{with at most one coordinate equal to~$1$}.
Set~$\omega_t=(e^{\pm it},\omega')\in\mathbb{T}^\mu$, with~$t$ a non-negative real number. \new{Let~$\varphi\colon H_1(M_L)\to\Z^\mu$ be a meridional homomorphism, set~$\mu_L=(M_L,\varphi)\in\Z^{\mu\choose 3}$ and let}~$W_F$ be the manifold associated with~$\new{L^\#}$
as in Section~\ref{sub:extension}. Recall form the proof of Proposition~\ref{prop:nullity} that the meridional homomorphism~$\pi_1(W_F)\to\mathbb{Z}^\mu$ is an isomorphism, which
implies that we are in the setting of Lemma~\ref{lemma:representing-int-forms}.
Hence, the intersection form on~$H_2(W_F;\mathbb{C}^{\omega_t})$
can be given by a matrix~$H(t)$
for all~$t\ge 0$.
Indeed, following the notation of Lemma~\ref{lemma:representing-int-forms}, one can take~$H(t)=H(\omega_t)$ if~$\mu=1$ and~$H(t)=H_j(\omega_t)$ for \new{some}~$j>1$ if~$\mu>1$.
The statement
now follows from Lemma~\ref{lemma:limit-signature} applied to~$H(t)$, Definition~\ref{def:extension}, Lemma~\ref{lem:Alex},
\new{and from the equality~$\delta_L(1,\omega')=\lim_{\omega_1\to 1^\pm}\delta_L(\omega_1,\omega')$
which is a direct consequence of the assumption on~$\omega'$.}
\end{proof}

\subsection{Limits of the Levine-Tristram signature}
\label{sub:LT}

Given an oriented link, recall the definition of the associated linking matrix~$\mathit{Lk}_L$
from Equation~\eqref{eq:Lk}.

\begin{theorem}
\label{thm:LT}
For any oriented link~$L$, we have
\[
\left|\lim_{\omega\to 1}\sigma_L(\omega)-\sign(\mathit{Lk}_L)\right|\le
\nul(\mathit{Lk}_L)-1-\rank A(L)\,.
\]
\end{theorem}
\begin{proof}
Let~$L$ be an arbitrary oriented link. The~$\mu=1$ case of Proposition~\ref{prop:limit} reads
\[
\left|\lim_{\omega\to 1^\pm}\sigma_L(\omega)-\sigma_{L}(1)\right|\le
\eta_{L}(1)-\rank A(L)\,,
\]
and the statement now follows immediately from the first points of Theorems~\ref{thm:Torres} and~\ref{thm:torres-formula-nullity}.
\end{proof}

The following corollary is immediate.

\begin{corollary}
\label{cor:lim-equal}
For any oriented link~$L$, we have the inequality~$\rank A(L)\le \nul(\mathit{Lk}_L)-1$. Moreover,~$\lim_{\omega\to 1}\sigma_L(\omega)=\sign(\mathit{Lk}_L)$ whenever the equality holds.\qed
\end{corollary}

A first class of links for which the equality~$\rank A(L)=\nul(\mathit{Lk}_L)-1$ holds
is when the right-hand side vanishes, yielding the following result.

\begin{corollary}
\label{cor:BZ}
If~$L$ is an oriented link such that~$\nul(\mathit{Lk}_L)=1$, then~$\lim_{\omega\to 1}\sigma_L(\omega)=\sign(\mathit{Lk}_L)$.\qed
\end{corollary}

\begin{remark}
\label{rem:B-Z}
As we now show, the condition~$\nul(\mathit{Lk}_L)=1$ is equivalent to~$(t-1)^m$ not dividing the non-vanishing Alexander polynomial~$\Delta_L(t)$ in~$\mathbb{Z}[t,t^{-1}]$, thus recovering the main result of~\cite{BZ}.
Since this is clearly true for knots, we assume without loss of generality that~$m\ge 2$. Recall that in such a case, the Hosokawa polynomial of~$L$~\cite{Hos} is defined by
\[
\nabla_L(t)=
\frac{\Delta_L(t,\dots,t)}{(t-1)^{m-2}}\in\mathbb{Z}[t,t^{-1}]\,.
\]
By~\cite[Theorem~2]{Hos}, the value of~$\nabla_L(1)$ is equal, up to a sign, to the determinant of the reduced linking matrix~$\widetilde{\mathit{Lk}_L}$ obtained from~$\mathit{Lk}_L$ by
deleting one row and the corresponding column. Therefore, we see that~$\mathit{Lk}_L$ has nullity~$1$ if and only if~$0\neq\det(\widetilde{\mathit{Lk}_L})=\pm\nabla_L(1)$, which is equivalent
to~$(t-1)$ not dividing~$\nabla_L(t)=\frac{\Delta_L(t,\dots,t)}{(t-1)^{m-2}}=\frac{\Delta_L(t)}{(t-1)^{m-1}}$, and to~$(t-1)^m$ not dividing~$\Delta_L(t)$.
\end{remark}

Another class of links for which the equality~$\rank A(L)=\nul(\mathit{Lk}_L)-1$ holds
is when the left-hand side is maximal, i.e. equal to~$m-1$. This is easily seen to be the case for boundary links (see e.g.~\cite{Hillman},
or~\cite[Corollary~3.6]{C-F}), immediately leading to the following result.

\begin{corollary}
\label{cor:boundary}
If~$L$ is a boundary link, then~$\lim_{\omega\to 1}\sigma_L(\omega)$ vanishes.\qed
\end{corollary}

As another direct consequence of Theorem~\ref{thm:LT}, we obtain
the following corollary, which refines
the last part of~\cite[Theorem~2.1]{G-L},
namely the inequality~$\left|\lim_{\omega\to 1}\sigma_L(\omega)\right|\le m-1$.

\begin{corollary}
\label{cor:G-L}
For any~$m$-component oriented link~$L$, we have
\[
\left|\lim_{\omega\to 1}\sigma_L(\omega)\right|\le m-1-\rank A(L)\,.
\]
\end{corollary}
\begin{proof}
By the triangle inequality together with Theorem~\ref{thm:LT}, we get
\begin{align*}
\left|\lim_{\omega\to 1}\sigma_L(\omega)\right|&\le\left|\lim_{\omega\to 1}\sigma_L(\omega)-\sign(\mathit{Lk}_L)\right|+\left|\sign(\mathit{Lk}_L)\right|\\
&\le \nul(\mathit{Lk}_L)+\left|\sign(\mathit{Lk}_L)\right|-1-\rank A(L)\le m-1-\rank A(L)\,,
\end{align*}
yielding the proof.
\end{proof}

\begin{remark}
    \label{rem:3d-LT}
    Actually, it is not difficult to obtain Theorem~\ref{thm:LT} (and its corollaries) using the three-dimensional method of Section~\ref{sec:3D}.
\end{remark}

\subsection{Limits of multivariable signatures with all variables tending to 1}
\label{sub:lim-mult-1}

The results of Section~\ref{sub:LT} allow us to study and in some case,
determine, the limits of multivariable signatures of colored links with
all variables tending to~$1$ simultaneously. More involved limits are treated in Section~\ref{sub:lim-mult}.

\medskip

Let~$L=L_1\cup\dots\cup L_\mu$ be an arbitrary~$\mu$-colored link.
For any choice of signs~$\epsilon=(\epsilon_1,\dots,\epsilon_\mu)\in\{\pm 1\}^\mu$, let us write
\[
\lim_{\omega_j\to 1^{\epsilon_j}}\sigma_L(\omega_1,\dots,\omega_\mu)\coloneqq\lim_{\omega\to 1^+}\sigma_L(\omega^{\epsilon_1},\dots,\omega^{\epsilon_\mu})\,.
\]
Also, let us denote by~$L^\epsilon$ the oriented link given by~$\epsilon_1L_1\cup\dots\cup\epsilon_\mu L_\mu$,
where~$+L_i=L_i$ and~$-L_i$ stands for the link~$L_i$ endowed with the opposite orientation. 

\begin{theorem}
\label{thm:lim-all-1}
For any colored link~$L=L_1\cup\dots\cup L_\mu$ and any signs~$\epsilon_1,\dots,\epsilon_\mu\in\{\pm 1\}$, we have the inequality
\[
\Big|\lim_{\omega_j\to 1^{\epsilon_j}}\sigma_L(\omega_1,\dots,\omega_\mu)-\sign(\mathit{Lk}_L^\epsilon)-\sum_{i<j}\epsilon_i\epsilon_j \lk(L_i,L_j)\Big|\le
\nul(\mathit{Lk}_L^\epsilon)-1-\rank A(L^\epsilon)\,,
\]
where~$\mathit{Lk}_L^\epsilon$ is the linking matrix of the oriented link~$L^\epsilon=\epsilon_1L_1\cup\dots\cup\epsilon_\mu L_\mu$.
\end{theorem}
\begin{proof}
Let us fix an arbitrary colored link~$L=L_1\cup\dots\cup L_\mu$ and signs~$\epsilon=(\epsilon_1,\dots,\epsilon_\mu)\in\{\pm 1\}^\mu$,
and let~$L^\epsilon$ be the associated oriented link defined above.
Applying Theorem~\ref{thm:LT} to~$L^\epsilon$, we get
\[
\left|\lim_{\omega\to 1}\sigma_{L^\epsilon}(\omega)-\sign(\mathit{Lk}_L^\epsilon)\right|\le
\nul(\mathit{Lk}_L^\epsilon)-1-\rank A(L^\epsilon)\,.
\]
By Propositions~2.5 and~2.8 of~\cite{C-F}, we have
\[
\sigma_{L^\epsilon}(\omega)=\sigma_{L^\epsilon}(\omega,\dots,\omega)-\sum_{i<j} \lk(\epsilon_iL_i,\epsilon_jL_j)=\sigma_L(\omega^{\epsilon_1},\dots,\omega^{\epsilon_\mu})-\sum_{i<j}\epsilon_i\epsilon_j \lk(L_i,L_j)\,,
\]
concluding the proof.
\end{proof}

The following corollary is an immediate consequence of Theorem~\ref{thm:lim-all-1}.

\begin{corollary}
\label{cor:lim-all-1}
Let~$L$ be an oriented link.
Then, for any coloring~$L_1\cup\dots\cup L_\mu$ of~$L$ and any signs~$\epsilon=(\epsilon_1,\dots,\epsilon_\mu)\in\{\pm 1\}^\mu$
such that the associated linking matrix~$\mathit{Lk}_L^\epsilon$ has nullity equal to~$1$, we have
\[
\lim_{\omega_j\to 1^{\epsilon_j}}\sigma_L(\omega_1,\dots,\omega_\mu)=\sign(\mathit{Lk}_L^\epsilon)+\sum_{i<j}\epsilon_i\epsilon_j \lk(L_i,L_j)\,.\qed
\]
\end{corollary}

\begin{example}
\label{ex:LT-2}
Let us consider the case of a 2-component~2-colored link~$L=L_1\cup L_2$. Writing~$\ell\coloneqq\lk(L_1,L_2)$ and fixing~$\epsilon=(\epsilon_1,\epsilon_2)\in\{\pm 1\}^2$, the corresponding linking matrix is given by
\[
\mathit{Lk}_L^\epsilon=
\begin{bmatrix}
-\epsilon_1\epsilon_2\ell&\epsilon_1\epsilon_2\ell\cr \epsilon_1\epsilon_2\ell&-\epsilon_1\epsilon_2\ell
\end{bmatrix}\,,
\]
which has nullity~$1$ if and only if~$\ell\neq 0$. In such as case, we have~$\sign(\mathit{Lk}_L^\epsilon)=-\epsilon_1\epsilon_2\,\sgn(\ell)$, and Corollary \ref{cor:lim-all-1} leads to
\[
\lim_{\omega_1\to 1^{\epsilon_1}\omega_2\to 1^{\epsilon_2}}\sigma_L(\omega_1,\omega_2)=\epsilon_1\epsilon_2(\ell-\sgn(\ell))\,.
\]
On the other hand, if~$\ell=0$ and the Alexander polynomial~$\Delta_{L^\epsilon}(t)=(t-1)^{-1}\Delta_L(t^{\epsilon_1},t^{\epsilon_2})$ vanishes, then Theorem~\ref{thm:lim-all-1} yields that the limit vanishes as well.
Finally, if~$\ell=0$ but~$\Delta_L(t^{\epsilon_1},t^{\epsilon_2})\neq 0$, then we can only conclude that the limit belongs to~$\{-1,0,1\}$ (and to~$\{-1,1\}$ for parity reasons).
This recovers the results of Corollary~\ref{cor:LT}, obtained via C-complexes.
\end{example}

We conclude this section with one last result,
which is a multivariable extension of Corollary~\ref{cor:G-L}.
Its proof being almost identical, it is left to the reader.

\begin{corollary}
\label{cor:G-L-mutl}
For any~$m$-component colored link~$L=L_1\cup\dots\cup L_\mu$ and any signs~$\epsilon_1,\dots,\epsilon_\mu\in\{\pm 1\}$, we have
\[
\Big|\lim_{\omega_j\to 1^{\epsilon_j}}\sigma_L(\omega_1,\dots,\omega_\mu)\Big|\le m-1+\Big|\sum_{i<j}\epsilon_i\epsilon_j \lk(L_i,L_j)\Big|-\rank A(L^\epsilon)\,.\qed
\]
\end{corollary}

\subsection{Limits of multivariable signatures}
\label{sub:lim-mult}

Everything is now in place to show the following theorem.
\begin{theorem}
\label{thm:main}
Let \(L=L_{1}\cup L_2\cup\ldots\cup L_{\mu}=:L_1\cup L'\) be a colored link with~$\mu\ge 2$ and~$L_1=:K$ a knot. Let us consider~\(\omega=(\omega_1,\omega')\in\mathbb{T}^{\mu}\) with~$\omega'\in\mathbb{T}^{\mu-1}_*$.
\begin{enumerate}
    \item 
If there exists a component~$K'\subset L'$ with~$\lk(K,K')\neq 0$, then we have:
\[
\left|\lim_{\omega_1\to 1^\pm}\sigma_L(\omega)-\sigma_{L'}(\omega')\right|\le
\eta_{L'}(\omega')-1+\sum_{K'\subset L'}|\lk(K,K')|-\rank  A(L)\,.
\]
\item If ~$\lk(K,K')=0$ for all components~$K'\subset L'$, then there is a well-defined slope~$(K/L)(\omega')\in\mathbb{R}\cup\{\infty\}$ for any~$\omega'\in\mathbb{T}^{\mu-1}_*$, and we have
\[
\left|\lim_{\omega_1\to 1^\pm}\sigma_L(\omega)-\sigma_{L'}(\omega')-s(\omega')\right|\le
\eta_{L'}(\omega')+\varepsilon(\omega')-\rank A(L)\,,
\]
where
\[
s(\omega')=
\begin{cases}
+1&\text{if $(K/L')(\omega')\in(0,\infty)$}\\
-1&\text{if $(K/L')(\omega')\in(-\infty,0)$}\\
0&\text{if $(K/L')(\omega')\in\{0,\infty\}$}
\end{cases}
\quad\text{and}\quad
\varepsilon(\omega')=
\begin{cases}
+1&\text{if $(K/L')(\omega')=0$}\\
-1&\text{if $(K/L')(\omega')=\infty$}\\
0&\text{else.}
\end{cases}
\]
\end{enumerate}
\end{theorem}
\begin{proof}
    Let~$L=L_1\cup L'$ be a~$\mu$-colored link with~$\mu\ge 2$ and~$L_1=K$ a knot. For any~$\omega'\in\new{\mathbb{T}^{\mu-1}_*}$, Proposition~\ref{prop:limit} yields
\[
\left|\lim_{\omega_1\to 1^\pm}\sigma_L(\omega_1,\omega')-\sigma_{L}(1,\omega')\right|\le
\eta_{L}(1,\omega')-\rank A(L)\,.
\]
The statement now follows from our Torres formulas, namely Corollary~\ref{cor:Torres-sign} (together with Remark~\ref{rem:slope}) and the second and third points of Theorem~\ref{thm:torres-formula-nullity}.
\end{proof}

This result is very powerful in the algebraically split case.

\begin{corollary}
\label{cor:alg-split-equality}
Let \(L=L_{1}\cup L_2\cup\ldots\cup L_{\mu}=:L_1\cup L'\) be a colored link with~$\mu\ge 2$ and~$L_1=:K$ a knot such that~$\lk(K,K')=0$ for all~$K'\subset L'$. Then, we have
\[
\lim_{\omega_1\to 1^+}\sigma_L(\omega_1,\omega')=\lim_{\omega_1\to 1^-}\sigma_L(\omega_1,\omega')=\sigma_{L'}(\omega')+\sgn\left(-\frac{\frac{\partial\nabla_L}{\partial t_1}(1,\sqrt{\omega'})}{\nabla_{L'}(\sqrt{\omega'})}\right)
\]
for all~$\omega'\in\mathbb{T}^{\mu-1}_*$ such that~$\nabla_{L'}(\sqrt{\omega'})\neq 0$ and~$\frac{\partial\nabla_L}{\partial t_1}(1,\sqrt{\omega'})\neq 0$.
\end{corollary}
\begin{proof}
For~$L=K\cup L'$ and~$\omega'\in\mathbb{T}^{\mu-1}_*$ as in the statement,
we have~$\Delta_{L'}(\omega')=\pm\nabla_{L'}(\sqrt{\omega'})\neq 0$,
which implies that~$\eta_{L'}(\omega')=0$ via Lemma~\ref{lem:eta0}.
By~\cite[Theorem~3.2]{DFL} (recall Remark~\ref{rem:slope}), the 
associated slope can be computed via
    \[
    (K/L')(\omega')=-\frac{\frac{\partial\nabla_L}{\partial t_1}(1,\sqrt{\omega'})}{2\nabla_{L'}(\sqrt{\omega'})}\,.
    \]
The assumption that~$\frac{\partial\nabla_L}{\partial t_1}(1,\sqrt{\omega'})\neq 0$ (which is equivalent to~$\frac{\partial\Delta_L}{\partial t_1}(1,\omega')\neq 0$) thus implies that this slope does not vanish, yielding~$\varepsilon(\omega')=0$.
Therefore, the right-hand side of the inequality in the second case of Theorem~\ref{thm:main} vanishes, leading to the result.
\end{proof}

\begin{example}
\label{ex:Wh}
For the 2-component links~$L(k)=K\cup L'$ of
Example~\ref{ex:twist}, Corollary~\ref{cor:alg-split-equality} together with Equation~\eqref{eq:slopeL(k)} leads to
\[
\lim_{\omega_1\to 1^+}\sigma_L(\omega_1,\omega_2)=\lim_{\omega_1\to 1^-}\sigma_L(\omega_1,\omega_2)=\sgn(k)\,.
\]
Recall from Example~\ref{ex:nonequality-signature-limit} that this result can not be obtained via Theorem~\ref{thm:limit-of-signature-inequality}.
\end{example}

Theorem~\ref{thm:main} is also powerful in the special case of total
linking number equal to~1,
as it easily implies the following result (using Lemma~\ref{lem:eta0}).

\begin{corollary}
\label{cor:l=1-equality}
Let \(L=L_{1}\cup L_2\cup\ldots\cup L_{\mu}=:L_1\cup L'\) be a colored link with~$L_1=:K$ a knot such that~$\sum_{K'\subset L'}\vert \lk(K,K')\vert=1$. Then, we have
\[
\lim_{\omega_1\to 1^+}\sigma_L(\omega_1,\omega')=\lim_{\omega_1\to 1^-}\sigma_L(\omega_1,\omega')=\sigma_{L'}(\omega')
\]
for all~$\omega'\in\mathbb{T}^{\mu-1}_*$ such that~$\Delta_{L'}(\omega')\neq 0$.\qed
\end{corollary}

On the other hand, Theorem~\ref{thm:main} is quite
weak in case of large linking numbers, as it does not distinguish between
the two possible limits. However, it does immediately provide the following upper bound on the difference of these limits.

\begin{corollary}
\label{cor:diff-lim}
Assuming the notation of Theorem~\ref{thm:main},
we have the inequalities
\[
\left|\lim_{\omega_1\to 1^+}\sigma_L(\omega)-\lim_{\omega_1\to 1^-}\sigma_L(\omega)\right|\le
2\Big(\eta_{L'}(\omega')-1+\sum_{K'\subset L'}|\lk(K,K')|-\rank  A(L)\Big)
\]
in case~1, and 
\[
\left|\lim_{\omega_1\to 1^+}\sigma_L(\omega)-\lim_{\omega_1\to 1^-}\sigma_L(\omega)\right|\le
2\Big(\eta_{L'}(\omega')+\varepsilon(\omega')-\rank A(L)\Big)
\]
in case~2.\qed
\end{corollary}

Note that the inequalities of Corollary~\ref{cor:diff-lim} are often sharp.
In other (slightly vague) words, even in the case of large linking numbers, Theorem~\ref{thm:main} is often ``as good as it can possibly be without distinguishing the two different limits''. This is made more precise by the following remark.

\begin{remark}
    \label{ex:diff-limits}
Let us assume that~$L=L_1\cup L'$ is an ordered link with~$\Delta_L(1,t_2,\dots,t_\mu)\neq 0$. (By the Torres formula, this is equivalent to~$\lk(L_1,L_j)$ not all vanishing
and~$\Delta_{L'}\neq 0$). Then Corollary~\ref{cor:diff-lim} reads
\[
\left|\lim_{\omega_1\to 1^+}\sigma_L(\omega)-\lim_{\omega_1\to 1^-}\sigma_L(\omega)\right|\le
2\Big(\eta_{L'}(\omega')-1+\sum_{j=2}^\mu\vert\lk(L_1,L_j)\vert\Big)
\]
for all~$\omega'\in\mathbb{T}^{\mu-1}_*$.
By Corollary~\ref{cor:limit-of-signature-equality}, Lemma~\ref{lem:eta0}, and the addendum to Theorem~\ref{thm:limit-of-signature-inequality},
this is a sharp inequality for~$\omega_j$ close to~$1^{s_j}$, where~$s_j$ denotes the sign of~$\lk(L_1,L_j)$.
\end{remark}

We conclude this section with an application of these results to the
limit of the Levine-Tristram signature of 2-component links.
To do so, let us first recall that by
the Torres formula~\eqref{eq:Torres}
adapted to the Conway function (see e.g. Equation~(5.3) of~\cite{Har}), any 2-component link~$L$ with vanishing linking number has Conway function of the form
\[
\nabla_L(t_1,t_2)=(t_1-t^{-1}_1)(t_2-t^{-1}_2)f(t_1,t_2)
\]
for some~$f\in\mathbb{Z}[t^{\pm 1}_1,t^{\pm 1}_2]$.
\begin{corollary}
\label{cor:LT-2}
If~$L=L_1\cup L_2$ is a 2-component oriented link with linking number~$\ell$
and two-variable Conway function~$\nabla_L$, then its Levine-Tristram signature satisfies
\[
\lim_{\omega\to 1}\sigma_L(\omega)=\begin{cases}
-\sgn(\ell)&\text{if~$\ell\neq 0$, or if~$\nabla_L=0$ (in which case~$\ell=0$);}\\
\sgn(f(1,1))&\text{if~$\ell=0$,~$\nabla_L\neq 0$ and~$f(1,1)\neq 0$};\\
\pm 1\new{\text{ or $0$}}&\text{if~$\ell=0$,~$\nabla_L\neq 0$ and~$f(1,1)= 0$},
\end{cases}
\]
where in the last two cases, we have~$\nabla_L(t_1,t_2)=(t_1-t^{-1}_1)(t_2-t^{-1}_2)f(t_1,t_2)\in\Z[t^{\pm 1}_1,t^{\pm 1}_2]$.\qed
\end{corollary}
\begin{proof}
Recall that by Equation~\eqref{eq:multi-LT}, the Levine-Tristram
and 2-variable signatures of a 2-component link~$L$ are related via~$\sigma_L(\omega)=\sigma_L(\omega,\omega)-\ell$
for all~$\omega\in S^1\setminus\{1\}$. Therefore, we need to compute the limit of~$\sigma_L(\omega_1,\omega_2)$ with both variables tending to~$1$.

If~$\ell\neq 0$, then we know that~$\lim_{\omega\to 1}\sigma_L(\omega,\omega)=
\ell-\sgn(\ell)$ by Corollary~\ref{cor:lim-all-1} (see also Example~\ref{ex:LT-2}, Corollary~\ref{cor:LT}, and the first point of Remark~\ref{rem:LT-2}),
yielding the result. We can therefore assume~$\ell=0$, and use Theorem~\ref{thm:main}. In this case, it reads
\begin{equation}
\label{equ:4D-2}
\left|\lim_{\omega_1\to 1^\pm}\sigma_L(\omega_1,\omega_2)-\sigma_{L_2}(\omega_2)-s(\omega_2)\right|\le
\eta_{L_2}(\omega_2)+\varepsilon(\omega_2)-\rank A(L)\,,
\end{equation}
with~$s(\omega_2)$ and~$\varepsilon(\omega_2)$ determined by the slope~$(L_1/L_2)(\omega_2)$ as described in the statement.
Recall also that this slope is equal to
\begin{equation}
\label{eq:DFL}
        (L_1/L_2)(\omega_2)=-\frac{\frac{\partial\nabla_L}{\partial t_1}(1,\sqrt{\omega_2})}{2\nabla_{L_2}(\sqrt{\omega_2})}
\end{equation}
whenever this fraction is not~$\frac{0}{0}$.
Taking the limit~$\omega_2\to 1^{\pm}$ in~\eqref{equ:4D-2} yields
\[
\left|\lim_{\omega\to 1^\pm}\sigma_L(\omega,\omega)-\lim_{\omega\to 1^\pm}s(\omega)\right|\le
\lim_{\omega\to 1^\pm}\varepsilon(\omega)-\rank A(L)\,.
\]
As one easily sees, the limit of the slope can be computed using Equation~\eqref{eq:DFL} together with the Torres formula
for the Conway function: if~$L=L_1\cup L_2$ is a 2-component link
with vanishing linking number, then we have~$\nabla_L(t_1,t_2)=(t_1-t_1^{-1})(t_2-t_2^{-1})f(t_1,t_2)$ for some~$f\in\Z[t^{\pm 1}_1,t^{\pm 1}_2]$, and~$\lim_{\omega\to 1}(L_1/L_2)(\omega)=f(1,1)$ up to a positive multiple.
As a consequence, we have the equality
\[
\lim_{\omega\to 1}\sigma_L(\omega)=\lim_{\omega\to 1^{\pm}}\sigma_L(\omega,\omega)=\sgn(f(1,1))
\]
in all possible cases, except possibly if~$f(1,1)=0$ while~$\Delta_L\neq 0$. In this later case, the inequality reads~$|\lim_{\omega\to 1}\sigma_L(\omega)|\le 1$.
\end{proof}

\begin{remark}
\label{rem:Conway}
\begin{enumerate}
    \item A family of links of the second kind is given by links of the form of the Whitehead link (or any~$L(k)$ of Figure~\ref{fig:link-Lk} with~$k\neq 0$) connected summed with two arbitrary knots.
    \item The third and last case can also happen.
Indeed, it is known that the Torres conditions are sufficient for~$\ell=0$, see e.g.~\cite{Platt}. Hence, there is no additional condition, in particular on the possible values of~$f(1,1)$.
\item Given the fact that both~$\sigma_L$ and~$\nabla_L$ can be computed from generalized Seifert matrices, it is plausible that Corollary~\ref{cor:LT-2} can also be obtained using the methods of Section~\ref{sec:3D}.
\end{enumerate}
\end{remark}

\subsection{Comparison of the three and four-dimensional approaches}
\label{sub:comparison}

In this short final section, we compare the
4D-results of Section~\ref{sec:4D} with the 3D-results of Section~\ref{sec:3D} regarding
the limits of multivariable signatures, namely Theorem~\ref{thm:main} and Theorem~\ref{thm:limit-of-signature-inequality} together with their corollaries.

Throughout this section, we assume that~$L=L_1\cup\dots\cup L_\mu=:L_1\cup L'$
 is an ordered link.
 
\medskip

Let us first assume that the total linking number~$\vert\ell\vert\coloneqq\sum_{j=2}^\mu\vert\lk(L_1,L_j)\vert$ is equal to~$1$.
In such as case, the functions~$\rho_\ell$ and~$\tau_\ell$ of~\eqref{eq:tau-rho} are identically zero. Therefore, Theorem~\ref{thm:limit-of-signature-inequality} and Theorem~\ref{thm:main} yield precisely the same result, namely the inequality
\[
\left|\lim_{\omega_1\to 1^\pm}\sigma_L(\omega_1,\omega')-\sigma_{L'}(\omega')\right|\le
\eta_{L'}(\omega')-\rank A(L)
\]
for all~$\omega'\in\mathbb{T}^{\mu-1}_*$.
In particular, Corollary~\ref{cor:l=1-equality} should be understood as
special case of Corollary~\ref{cor:limit-of-signature-equality}.

\medskip

If the linking numbers satisfy~$|\ell|>1$, then Theorem~\ref{thm:main} is in general rather weak
for determining the limits of signatures. Indeed, and as already explained in Section~\ref{sub:lim-mult}, the inequality
\[
\left|\lim_{\omega_1\to 1^\pm}\sigma_L(\omega)-\sigma_{L'}(\omega')\right|\le
\eta_{L'}(\omega')-1+\sum_{j=2}^\mu|\lk(L_1,L_j)|-\rank  A(L)
\]
is plagued by the fact that it does not distinguish between
the limits~$\omega_1\to 1^+$ and~$\omega_1\to 1^-$. In that case,
Theorem~\ref{thm:limit-of-signature-inequality} is obviously stronger, as it determines the limits of signatures for generic~$\omega'$
(recall Corollary~\ref{cor:limit-of-signature-equality}).

\medskip

On the other hand, Theorem~\ref{thm:main} outcompetes its three-dimensional contender in the algebraically split case~$\lk(L_1,L_2)=\dots=\lk(L_1,L_\mu)=0$.
Indeed, Theorem~\ref{thm:limit-of-signature-inequality} reads
\[
\left| \lim_{\omega_1 \to 1^{\pm}} \sigma_L(\omega_1,\omega') - \sigma_{L'}(\omega')\right| \leq \eta_{L'}(\omega') + 1 -\rank A(L)
\]
for all~$\omega'\in\mathbb{T}^{\mu-1}_*$.
A much stronger statement is obtained via Theorem~\ref{thm:main}, as it determines the
limits of signatures for generic~$\omega'$
(recall Corollary~\ref{cor:alg-split-equality}).

\medskip

Therefore, and in our opinion quite remarkably,
the three- and four-dimensional approaches turn out
to give complementary results.

\appendix{}

\section{Plumbed three-manifolds}
\label{ap:plumbed}

The aim of this appendix is \new{state and prove several technical lemmas on plumbed manifolds that play a crucial role in establishing the invariance of
the extended signature and nullity (Theorem~\ref{thm:extension}).}

\medskip

\new{Let~$\Gamma$ be an arbitrary plumbing graph
with vertices decorated by surfaces~$F_1,\dots,F_\mu$
(recall Section~\ref{sec:plumbed}), and let us write~$F_i=\bigsqcup_j F_{i,j}$ for the connected components of~$F_i$. Recall that each (oriented) edge~$e$ of~$\Gamma$, say with source~$s(e)=F_i$, comes with the specification of a connected component of~$s(e)$, say~$F_{i,j}$; we shall denote this fact by~$s(e)=(i,j)$, or equivalently by~$t(\overline{e})=(i,j)$.}

\new{Let us fix~\(\omega=(\omega_1,\dots,\omega_\mu)\in\mathbb{T}^\mu\).
Without loss of generality, we assume that there exists~$0\le k\le\mu$ such that~$\omega_i=1$ for~$1\le i\le k$ and~$\omega_i\neq 1$ for~$k<i\le\mu$. Let us denote by~\(\varphi_{P,\omega}\colon H_{1}(P(\Gamma))\to\mathbb{C}^*\) the composition of a  meridional homomorphism~$\varphi_P\colon H_1(P(\Gamma))\to\Z^\mu$ with the homomorphism~$\chi_\omega\colon\mathbb{Z}^\mu\to\mathbb{C}^*$ determined by~\(t_{i}\mapsto \omega_{i}\).
This induces twisted coefficients that we denote by~$\C^\omega$.
}

\new{The aim of the first lemma is to describe
in full generality the kernel of the inclusion induced map~$H_1(\partial P(\Gamma);\mathbb{C}^{\omega})\to H_1(P(\Gamma);\mathbb{C}^{\omega})$. This extends~\cite[Lemma~4.7]{CNT}, which corresponds to the case where each~$F_i$ is connected and~$\omega=(1,\dots,1)$. To do so,}
consider the set~\new{$\mathcal{K}\coloneqq \mathcal{K}_1\sqcup\dots\sqcup \mathcal{K}_k$, where}
\[
\mathcal{K}_i\coloneqq\{K\subset\partial F_i\,|\,\varphi_{P,\omega}([K])=1\}
\]
is the set of boundary components of~$F_i$ mapped to~$1$ by~$\varphi_{P,\omega}$.
The decomposition~$F_i=\bigsqcup_j F_{i,j}$ yields a partition~$\mathcal{K}_i=\bigsqcup_j \mathcal{K}_{i,j}$ of these boundary components.
Finally, for any~$K\in\mathcal{K}_i$, we denote by~$m_K\in H_1(\partial P(\Gamma);\mathbb{C}^{\omega})$ the class of the corresponding meridian.

\begin{lemma}
\label{lem:P(G)}
The kernel of the inclusion induced map~$H_1(\partial P(\Gamma);\mathbb{C}^{\omega})\to H_1(P(\Gamma);\mathbb{C}^{\omega})$
is~$|\mathcal{K}|$-dimensional, freely generated by the union of:
\begin{enumerate}
    \item[$\bullet$] for each\new{~$1\le i\le k$ and}~$j$ such that~$\varphi_{P,\omega}(H_1(F_{i,j}^\circ))\neq\{1\}$, the set~$\{m_K\,|\,K\in\mathcal{K}_{i,j}\}$;
    \item[$\bullet$] for each\new{~$1\le i\le k$ and}~$j$ such that~$\varphi_{P,\omega}(H_1(F_{i,j}^\circ))=\{1\}$, the set
    \[
    \{m_K-m_{K^0_{i,j}}\,|\,K\in\mathcal{K}_{i,j}\setminus\{K^0_{i,j}\}\}\,,
    \]
    where~$K^0_{i,j}$ is any fixed element of~$\mathcal{K}_{i,j}$;
    \item[$\bullet$] for each\new{~$1\le i\le k$ and}~$j$ such that~$\varphi_{P,\omega}(H_1(F_{i,j}^\circ))=\{1\}$, the element
    \[
        [\partial F_{i,j}]\new{-\sum_{s(e)=(i,j)}\varepsilon(e)\,m_{K^0_{t(e)}}}\,,
        \]
    \new{where the sum is over all edges~$e$ of~$\Gamma$ with~$s(e)=(i,j)$.}
\end{enumerate} 
    \end{lemma}
    
\begin{proof}
By construction, the boundary of~$P(\Gamma)$ consists of disjoint tori indexed
by the boundary components~$K\subset \partial F$. These tori are~$\mathbb{C}^\omega$-acyclic, except
possibly the ones indexed by~$K\subset\partial F_i$ \new{with~$1\le i\le k$.}
For such a torus, its meridian~$m_K$ gets mapped to~$\omega_i=1$ by assumption,
and its longitude~$\ell_K$ to~$\varphi_{P,\omega}([K])$. Therefore, the space~$H_1(\partial P(\Gamma);\mathbb{C}^{\omega})$ is freely generated by~$\{m_K,\ell_K\}_{K\in\mathcal{K}}$.
By the standard Poincar\'e-Lefschetz duality argument, the kernel of the inclusion induced map~$\iota\colon H_1(\partial P(\Gamma);\mathbb{C}^{\omega})\to H_1(P(\Gamma);\mathbb{C}^{\omega})$
is~$|\mathcal{K}|$-dimensional, and it only remains to \new{check that the~$|\mathcal{K}|$
linearly independent elements in the statement belong to this kernel.}

By a Mayer-Vietoris argument applied to~$P(\Gamma)=\bigcup_i F_i^\circ\times S^1$ (see the proof~\cite[Lemma~4.7]{CNT} for the untwisted case), we have an exact sequence
\[
\bigoplus_{e\in E}H_1(T_e;\mathbb{C}^\omega)\stackrel{\iota_t-\iota_s}{\longrightarrow}\bigoplus_{i=1}^{\mu}H_1(F_i^\circ\times S^1;\mathbb{C}^\omega)\longrightarrow H_1(P(\Gamma);\mathbb{C}^{\omega})\longrightarrow\bigoplus_{e\in E}H_0(T_e;\mathbb{C}^\omega)\,,
\]
where~$T_e\subset P(\Gamma)$ is the torus corresponding to the edge~$e\in E$ \new{and~$\iota_t,\iota_s$ are induced by the inclusion
of~$T_e$ into~$F_{t(e)}\times S^1,F_{s(e)}\times S^1$, respectively.
As observed in the proof of~\cite[Lemma~4.7]{CNT}, the inclusion of~$\partial F_i\subset\partial P(\Gamma)$ into~$P(\Gamma)$ factors through the space~$\bigsqcup_{i} F_{i}^\circ \times S^1=\bigsqcup_{i,j} F_{i,j}^\circ \times S^1$, yielding the commutative diagram}
\[
\begin{tikzcd}
\displaystyle\bigoplus_{e\in E}H_1(T_e;\mathbb{C}^\omega)\arrow[r,"\iota_t-\iota_s"]&\displaystyle\bigoplus_{i,j}H_1(F_{i,j}^\circ\times S^1;\mathbb{C}^\omega)\arrow[r] &H_1(P(\Gamma);\mathbb{C}^{\omega})\\
& & H_1(\partial P(\Gamma);\mathbb{C}^{\omega})\arrow[ul,"f"]\arrow[u,"\iota"]\,.
\end{tikzcd}
\]
\new{By exactness, we have~$\ker(\iota)=\{x\in H_1(\partial P(\Gamma);\mathbb{C}^{\omega})\mid f(x)\in \operatorname{Im}(\iota_t-\iota_s)\}$,
and we are left with the proof that our~$|\mathcal{K}|$
elements belong to this subspace.}

If~$\varphi_{P,\omega}(H_1(F_{i,j}^\circ))\neq\{1\}$, we can use the K\"unneth theorem to get
\[
H_1(F_{i,j}^\circ\times S^1;\mathbb{C}^\omega)\simeq H_1(F_{i,j}^\circ;\mathbb{C}^\omega)\oplus\left(H_0(F_{i,j}^\circ;\mathbb{C}^\omega)\otimes\mathbb{C}[\ast\times S^1]\right)=H_1(F_{i,j}^\circ;\mathbb{C}^\omega)\,,
\]
since~$H_0(F_{i,j}^\circ;\mathbb{C}^\omega)=0$ in this case (see e.g.~\cite[Lemma~2.6]{CNT}).
\new{Since each~$m_K$ with~$K\subset \partial F_{i,j}$ satisfies~$f(m_K)=[\ast\times S^1]=0$,
the set~$\{m_K\,|\,K\in\mathcal{K}_{i,j}\}$
lies in~$\ker(f)\subset\ker(\iota)$.}

\new{Let us now assume that~$\varphi_{P,\omega}(H_1(F_{i,j}^\circ))=\{1\}$, and argue as in the proof of~\cite[Lemma~4.7]{CNT} once again. Since~$F_{i,j}$ is connected, all~$m_K$ with~$K\subset\partial F_{i,j}$ are mapped by~$f$ to the same element in $H_1(F_{i,j}^\circ\times S^1;\mathbb{C}^\omega)=H_1(F_{i,j}^\circ\times S^1;\mathbb{C})$. Therefore,
the set~$\{m_K-m_{K^0_{i,j}}\,|\,K\in\mathcal{K}_{i,j}\}$
belongs to~$\ker(f)\subset\ker(\iota)$. As for the last element, the isomorphism~$H_1(F_{i,j}^\circ\times S^1;\mathbb{C})\simeq H_1(F_{i,j}^\circ;\mathbb{C})\oplus\C$ ensures that its image by~$f$ satisfies
\[
f\Big([\partial F_{i,j}]-\sum_{s(e)=(i,j)}\varepsilon(e)\,m_{K^0_{{t(e)}}}\Big)=\sum_{s(e)=(i,j)} \big([\partial D_e]-\varepsilon(e)\,m_{K^0_{{t(e)}}}\big)\,,
\]
while the construction (in particular~\eqref{eq:plumbing}) yields
\[
[\partial D_e]-\varepsilon(e)\,m_{K^0_{{t(e)}}}=\iota_s([\partial D_e])-\iota_t([\partial D_e])=(\iota_t-\iota_s)([-\partial D_e])\,.
\]
In particular, the image under~$f$ of this last element lies in the image of~$\iota_t-\iota_s$. This concludes the proof.}
\end{proof}

\new{Let us now fix an arbitrary~\(\mu\)-colored link~\(L\).
Let~$P(L)$ be the plumbed~$3$-manifold associated with the plumbing graph~\(\Gamma_{L}\), and let~$M_L = X_{L} \cup_{\partial} -P(L)$ be the generalized Seifert surgery on~$L$ (recall Section~\ref{sec:plumbed}).
By Lemma~\ref{lem:ML}, we have meridional homomorphisms \(\varphi_P\colon H_{1}(P(L))\rightarrow\mathbb{Z}^\mu\) and $\varphi\colon H_1(M_L)\to\Z^\mu$, which allow to
define homomorphisms
\[\varphi_{P,\omega} \colon H_{1}(P(L)) \to \C^{\ast}, \quad \varphi_{\omega} \colon H_{1}(M_{L}) \to \C^{\ast}\]
for any~\(\omega \in \mathbb{T}^{\mu}\),
and twisted coefficients systems that we denote by~$\C^\omega$.
As noted in Remark~\ref{rems:M_L}, these meridional homomorphisms are generally not unique, so the notation~$\C^\omega$ might seem inappropriate. This abuse of notation is justified by the following lemma and its proof.}

\begin{lemma}
\label{lem:kernel-dim}
    \new{For any~$\omega\in\mathbb{T}^\mu$, the dimension of~$H_1(M_L;\mathbb{C}^{\omega})$ does not depend on the choice of the meridional homomorphism~$\varphi$.}
\end{lemma}
\begin{proof}
\new{The proof is split into three steps.
First, we prove that \(H_{1}(P(L);\C^\omega)\) is independent of the choice of the meridional homomorphism~\(\varphi_{P}\).
Next, we show that the same holds for the inclusion induced map~$H_{1}(\partial P(L);\C^\omega) \to H_{1}(P(L);\C^{\omega})$.
In the last step, we combine the information from the first two steps to prove the lemma.}

\new{Let \(\varphi_{P} \colon H_{1}(P(L)) \to \Z^{\mu}\) be a meridional homomorphism.
Consider the Mayer-Vietoris sequence associated to the presentation of~\(P(L)\) as~\(P(\Gamma_{L})\):
\[
\begin{tikzcd}[ampersand replacement=\&]
    \bigoplus_{e \in E} H_{1}(T_{e};\C^{\omega}) \arrow{d} \& \& \bigoplus_{K\subset L} H_{0}(D^{\circ}_{K} \times S^1;\C^{\omega})\,. \\
    \bigoplus_{K\subset L} H_{1}(D^{\circ}_{K} \times S^1;\C^{\omega}) \arrow{r} \& H_{1}(P(L);\C^{\omega}) \arrow{r} \& \bigoplus_{e \in E} H_{0}(T_{e};\C^{\omega}) \arrow{u}
\end{tikzcd}
\]
Observe that the terms in the left and right column do not depend on the particular choice of~\(\varphi_{P}\), as long as it is meridional.
By the exactness of the above sequence it follows that \(H_{1}(P(L);\C^{\omega})\) is independent of the choice of~\(\varphi_{P}\) (as long as it is meridional).}

\new{
We now show that the kernel of the inclusion induced map~$H_1(\partial P(L);\mathbb{C}^{\omega})\to H_1(P(L);\mathbb{C}^{\omega})$
does not depend on the choice of~$\varphi_P$.
Indeed, by definition of a meridional homomorphism (recall Remark~\ref{rems:M_L}.1), the value of any such homomorphism~$\varphi_P$ is determined on each class of the form~$[\partial D_e]$. Since the surfaces~$F_{i,j}$ attached to the vertices of~$\Gamma_L$ are discs~$D_K$, the value of~$\varphi_P$ on each~$[K]=[\partial D_K]$ and on each~$H_1(F^\circ_{i,j})=H_1(D^\circ_K)$ is determined as well.
Therefore, by Lemma~\ref{lem:P(G)}, the kernel of the inclusion induced map~$H_1(\partial P(L);\mathbb{C}^{\omega})\to H_1(P(L);\mathbb{C}^{\omega})$ is fully determined by~$\omega$.
}

\new{Finally, consider the Mayer-Vietoris sequence associated to~\(M_{L}=X_{L} \cup_{\partial} -P(L)\):
\[
    \begin{tikzcd}[ampersand replacement=\&]
        H_{1}(\partial X_{L};\C^{\omega}) \arrow{d} \& \& H_{0}(X_{L};\C^\omega) \oplus H_{0}(P(L);\C^\omega)\,. \\
        H_{1}(X_{L};\C^{\omega}) \oplus H_{1}(P(L);\C^{\omega}) \arrow{r} \& H_{1}(M_{L};\C^{\omega}) \arrow{r} \& H_{0}(\partial X_{L};\C^{\omega}) \arrow{u}
    \end{tikzcd}
\]
By the first two steps of the proof, the maps and entries in the left and right columns are independent of a particular choice of~\(\varphi\) (as long as it is meridional).
It follows that \(\dim_{\C} H_{1}(M_{L};\C^{\omega})\) does not depend on~\(\varphi\).} 
\end{proof}

\new{
Our proof of the invariance of the extended signature and nullity (Theorem~\ref{thm:extension}) makes use of the~$\rho$-invariant defined by Atiyah-Patodi-Singer in~\cite{atiyahSpectralAsymmetryRiemannian1975a},
whose main properties we now recall.}

\new{Given a closed oriented~\(3\)-manifold~$M$ endowed with a homomorphism \(\alpha \colon H_{1}(M) \to S^{1}\), we can assign, by analyzing the spectrum of the associated twisted odd signature operator, the \(\rho\)-invariant~\(\rho(M,\alpha) \in \R\).
From our point of view, the most important property of this invariant is its relation to the signature.
Namely, if~\(W\) is a compact and oriented \(4\)-manifold such that \(\partial W = M\) and the homomorphism~\(\alpha\) extends to a map \(\beta \colon H_{1}(W) \to S^{1}\), then
\begin{equation}\label{eq:rho-invariant}
    \rho(M,\alpha) = \sigma(W) - \sigma_{\beta}(W),
\end{equation}
see~\cite[Theorem~2.4]{atiyahSpectralAsymmetryRiemannian1975a}.
Another useful property is that the \(\rho\)-invariant is additive under disjoint sums, see~\cite[Theorem~1.2.1]{Neumann}.
If the~$3$-manifold~$M$ is endowed with a homomorphism~\(\varphi \colon H_{1}(M) \to \Z^{\mu}\), we will use the shortened notation
\[\rho_{\omega}(M) := \rho(M, \varphi_{\omega})\]
for any~$\omega\in\mathbb{T}^\mu$, where~$\varphi_\omega\colon H_1(M)\to S^1$ denotes the composition of~$\varphi$ with the homomorphism~$\chi_\omega\colon\mathbb{Z}^\mu\to S^1$ determined by~\(t_{i}\mapsto \omega_{i}\).}

\new{Recall that a plumbing graph is said to be}
{\em balanced\/} if for any pair of vertices~$v,w\in V$, we have~$\sum_{e=(v,w)}\varepsilon(e)=0$, where the sum is over the set of edges~$e\in E$ with~$s(e)=v$ and~$t(e)=w$.

\new{
\begin{lemma}
\label{lem:rho=0}
Let~$G$ be a balanced plumbing graph with vertices given by closed oriented surfaces. Then,
for any meridional homomorphism~\(\varphi\) on~$P(G)$, we have~$\rho_\omega(P(G))=0$
for all~$\omega\in\mathbb{T}^\mu$.
\end{lemma}}
\begin{proof}
    \new{Fix a balanced plumbing graph~$G$ and a meridional homomorphism~\(\varphi\colon H_1(P(G))\to\Z^\mu\).
    Recall from~\cite[Lemma~4.9]{CNT} that there exists a compact and oriented \(4\)-manifold~\(Z\) equipped with a map \(\psi \colon H_{1}(Z) \to \Z^{\mu}\), such that:
    \begin{enumerate}
        \item \(\partial Z = -P(G) \sqcup P\), where \(P = P(G')\) with~\(G'\) an appropriate plumbing graph with no edges; in other words,~\(P\) is a disjoint union of manifolds of the form \(\Sigma \times S^{1}\), with \(\Sigma\) closed, oriented and connected surfaces;
        \item the restriction of~\(\psi\) to~\(P(G)\) is~\(\varphi\), and the restriction of~\(\psi\) to~\(P\) is meridional;
        \item \(\sigma(Z) = \sigma_{\omega}(Z) = 0\) for all \(\omega \in (S^{1} \setminus \{1\})^{\mu}\), where~$\sigma_\omega(Z)$ denotes the signature with twisted coefficients induced by~$\psi$ and~$\omega$.
    \end{enumerate}
    We will prove that the conclusion of the quoted lemma is valid in more generality, i.e. that \(\sigma_{\omega}(Z) = 0\) for all \(\omega \in \mathbb{T}^{\mu}\).
    Once this fact is established, formula~\eqref{eq:rho-invariant} yields
    \[\rho_{\omega}(P) - \rho_{\omega}(P(G)) =\rho_\omega(\partial Z) = \sigma(Z) - \sigma_{\omega}(Z) = 0\]
    for all~$\omega\in\mathbb{T}^\mu$.
    From~\cite[Corollary~4.3]{CNT}, we have \(\rho_{\omega}(P) = 0\), hence \(\rho_{\omega}(P(G)) = 0\) for all~\(\omega \in \mathbb{T}^{\mu}\). (Note that~\cite[Corollary~4.3]{CNT} is stated only for~$\omega\in(S^1\setminus\{1\})^\mu$, but the proof applies to arbitrary~$\omega\in\mathbb{T}^\mu$.)}

    \new{For that matter, let us review the construction of~\(Z\) given in~\cite[Lemma~4.9]{CNT}, using the notation~$I:=[0,1]$.
    We start with \(Z_{0} = P(G) \times I\) and attach to \(Z_{0}\) a toral handle \(TH_{1} = I \times I \times S^1 \times S^1\) along its attaching region \(ATH_1 = \partial I \times I \times S^1 \times S^1\).
    The attaching of~\(TH_{1}\) is determined by a pair of edges~\(e,e'\) of~\(G\) with common initial and terminal vertices and different signs.
    The resulting \(4\)-manifold \(Z_{1} = Z_{0} \cup TH_{1}\) has boundary
    \[\partial Z_{1} = -P(G) \sqcup P(G_{1}),\]
    where \(G_{1}\) is the plumbing graph obtained from~\(G\) be removing the edges~\(e,e'\) and increasing the genus of the surfaces associated to the initial and terminal vertices of~\(e\) and~\(e'\).
    Since \(G\) is balanced, we can repeat this construction finitely many times to obtain
    \[Z_{n} = Z_{0} \cup TH_{1} \cup TH_{2} \cup \ldots \cup TH_{n}\]
    with
    \[\partial Z_{n} = -P(G) \sqcup P(G_{n})\]
    and~\(G_{n}\) a graph with no edge.
    We can then take \(Z = Z_{n}\) and \(G'=G_{n}\).
Additionally, it is shown in the proof of~\cite[Lemma~4.9]{CNT} that for each~\(0 \leq k \leq n\), the manifold~\(Z_{k}\) is equipped with a homomorphism \(\psi_{k} \colon H_{1}(Z_{k}) \to \Z^{\mu}\) whose restriction to~\(P(G)\) is~\(\varphi\) and whose restriction~\(\varphi_{k}\) to~\(P(G_{k})\) is meridional.
Finally, the authors argue inductively that for all \(0 \leq k \leq n\),
    \[\sigma_{\omega}(Z_{k}) = \sigma(Z_{k}) = 0\]
    for~\(\omega \in (S^{1} \setminus \{1\})^{\mu}\), where~$\sigma_\omega(Z_k)$ denotes the signature with twisted coefficients induced by~$\psi_k$ and~$\omega$.
    In order to complete the proof, we need to check that the equality
    \[\sigma_\omega(Z_{k}) = 0\]
    holds for any \(\omega \in \mathbb{T}^{\mu}\) and any \(0 \leq k \leq n\).}
    
    \new{Let us start with the \(k=0\) case.
    Since \(Z_{0} = P(G) \times I\), we can take \(\psi = \varphi \circ (\iota_{\ast})^{-1}\), where
    \[\iota \colon P(G) \times \{0\} \hookrightarrow Z_{0}\]
    is the inclusion of one of the components of the boundary.
    Since~\(\iota\) induces an isomorphism on twisted homology, the twisted intersection form of~\(Z_{0}\) is trivial, and~$\sigma_\omega(Z_{0}) = 0$ for all~\(\omega \in \mathbb{T}^{\mu}\).}
    
    \new{For the inductive step, we apply the twisted version of the Novikov-Wall theorem, as recalled in Section~\ref{sub:NW}, to the union
    \[Z_{k+1} = Z_{k} \cup TH_{k+1}.\]
    The existence of~\(\psi_{k+1}\) is proven in~\cite[Lemma~4.9]{CNT}.
    Using the notation from Section~\ref{sub:NW}, we can write
    \[X_{0} = ATH_{k+1}, \quad X_{+} = \overline{\partial TH_{k+1} \setminus ATH_{k+1}}, \quad X_{-} = -P(G)\sqcup (P(G_{k}) \setminus f(ATH_{k+1})),\]
    where \(f \colon ATH_{k+1} \to P(G_{k})\) is the gluing diffeomorphism, and
    \[\Sigma = X_{+} \cap X_{0} = \partial ATH_{k+1} = \partial I\times\partial I \times S^1 \times S^{1}.\]
    From Theorem~\ref{thm:NW} we obtain
    \[\sigma_\omega(Z_{k+1}) = \sigma_\omega(Z_{k}) + \sigma_\omega(TH_{k+1}) + \mathit{Maslov}(\mathcal{L}_{-}, \mathcal{L}_{0}, \mathcal{L}_{+}),\]
    where~$\sigma_\omega(TH_{k+1})$ stands for the twisted signature induced by the restriction~\(\theta\colon H_1(TH_{k+1})\to\Z^\mu\) of~\(\psi_{k+1}\) and by~$\omega$. 
    By inductive assumption, we have~\(\sigma_\omega(Z_{k}) = 0\) for all~$\omega\in\mathbb{T}^\mu$, and hence
    \[
    \sigma_\omega(Z_{k+1}) = \sigma_\omega(TH_{k+1}) + \mathit{Maslov}(\mathcal{L}_{-}, \mathcal{L}_{0}, \mathcal{L}_{+})
   \]
    for all~$\omega\in\mathbb{T}^\mu$.
    Observe that if~\(\theta_\omega:=\chi_\omega\circ\theta\) is nontrivial, then \(H_{\ast}(TH_{k+1};\C^\omega) = H_{\ast}(\Sigma;\C^\omega) = 0\), implying 
    $\sigma_\omega(Z_{k+1}) = 0$ as claimed.
    Therefore, to complete the proof, we only need to deal with the case of trivial~\(\theta_\omega\). 
    We then have
    \[H_{1}(\Sigma; \C^\omega) = H_{1}(\Sigma;\C) = H_{1}(\partial I\times \partial I \times S^{1} \times S^{1};\C) \cong H_{1}(S^{1} \times S^{1};\C)^{4}.\]
    Furthermore,~\(H_{2}(TH_{k+1};\C^\omega) = H_{2}(TH_{k+1};\C)\) is generated by the image of~\(H_{2}(\partial TH_{k+1};\C)\), yielding~\(\sigma_\omega(TH_{k+1}) = 0\)
    and
    \[
        \sigma_{\omega}(Z_{k+1}) = \mathit{Maslov}(\mathcal{L}_{-}, \mathcal{L}_{0}, \mathcal{L}_{+})\,,
    \]
    so we are left with the proof that this Maslov index vanishes.}

    \new{
    In order to compute it, we will write down explicit generators of \(\mathcal{L}_{-}\), \(\mathcal{L}_{0}\) and \(\mathcal{L}_{+}\) using notation from Section~\ref{sec:plumbed}.
    Let~\(e\) and~\(e'\) denote the edges of~\(G_{k}\) involved in the construction of~$Z_{k+1}$, which satisfy \(s(e) = s(e')\) and \(t(e) = t(e')\). Recall that there are four disks \(D_{e}, D_{e'} \subset F_{s(e)}\) and \(D_{\overline{e}}, D_{\overline{e}'} \subset F_{t(e)}\) which are removed in the construction of~$P(G_k)$. Moreover, the tori \((-\partial D_{e}) \times S^1\) and \((-\partial D_{\overline{e}}) \times S^{1}\) are identified in~\(P(G_{k})\) via~\eqref{eq:plumbing}, the image of these two identified tori being denoted by~\(T_{e}\), and similary for~$e'$.
    The gluing map \(f \colon ATH_{k+1} \to P(G_{k})\) appropriately identifies \(\{0\} \times I \times S^{1} \times S^{1} \subset ATH_{k+1}\) with a tubular neighborhood~\(\nu(T_{e})\) of~\(T_{e}\), and \(\{1\} \times I \times S^{1} \times S^{1} \subset ATH_{k+1}\) with a tubular neighborhood~\(\nu(T_{e'})\) of~\(T_{e'}\).
    Using this notation, we can write
    \[\Sigma = \underbrace{(-\partial D_{e}) \times S^{1}_{e} \;\sqcup\; (-\partial D_{\overline{e}}) \times S^{1}_{\overline{e}}}_{\partial \nu(T_{e})} \;\sqcup\; \underbrace{(-\partial D_{e'}) \times S^{1}_{e'} \;\sqcup\; (-\partial D_{\overline{e}'}) \times S^{1}_{\overline{e}'}}_{\partial \nu(T_{e'})}\,.\]
    Assuming without loss of generality that~$\varepsilon(e)=1$ and~$\varepsilon(e')=-1$, we get
    \begin{align*}
        \mathcal{L}_{+} &= \operatorname{span} \left\{ [-\partial D_{e}] + [-\partial D_{e'}], [S^{1}_{e}] - [S^{1}_{e'}], [-\partial D_{\overline{e}}] + [-\partial D_{\overline{e}'}], [S^{1}_{\overline{e}}] - [S^{1}_{\overline{e}'}] \right\}, \\  
        \mathcal{L}_{0} &= \operatorname{span} \left\{ [-\partial D_{e}] + [S^{1}_{\overline{e}}], [-\partial D_{\overline{e}}] + [S^{1}_{e}], [-\partial D_{e'}] - [S^{1}_{\overline{e}'}], [-\partial D_{\overline{e}'}] - [S^{1}_{e'}] \right\}\,. 
    \end{align*}
    In order to obtain the description of~\(\mathcal{L}_{-}\) and complete the proof, we need to consider three cases. Firstly, if the restriction of~\(\psi_{k,\omega}:=\chi_\omega\circ\psi_k\) to \(H_{1}(F^{\circ}_{s(e)})\) and \(H_{1}(F^{\circ}_{t(e)})\) is trivial, then the Maslov index vanishes as verified in the proof of~\cite[Lemma~4.9]{CNT}. Secondly, if the restriction of~\(\psi_{k,\omega}\) is trivial on \(H_{1}(F_{s(e)}^{\circ})\) and nontrivial on~\(H_{1}(F^{\circ}_{t(e)})\) (and similarly the other way around), then Lemma~\ref{lem:P(G)} implies that
    \[\mathcal{L}_{-} = \operatorname{span} \left\{ [S^{1}_{e}]-[S^{1}_{e'}], [-\partial D_{e}] + [-\partial D_{e'}], [S^{1}_{\overline{e}}], [S^{1}_{\overline{e}'}]\right\}.\]
    Using elementary but tedious calculations, one can check that
    \begin{equation}
        \label{eq:lag}
    (\mathcal{L}_{-} + \mathcal{L}_{0}) \cap \mathcal{L}_{+} = (\mathcal{L}_{-} \cap \mathcal{L}_{+}) + (\mathcal{L}_{0} \cap \mathcal{L}_{+})\,.
    \end{equation}
    By Remark~\ref{rem:NW}.4, this implies that~\(\mathit{Maslov}(\mathcal{L}_{-}, \mathcal{L}_{0}, \mathcal{L}_{+}) = 0\) as claimed.
    Lastly, assume that the restriction of~\(\psi_{k,\omega}\) is nontrivial both on \(H_{1}(F_{s(e)}^{\circ})\) and on~\(H_{1}(F^{\circ}_{t(e)})\).
    Lemma~\ref{lem:P(G)} then implies that
    \[\mathcal{L}_{-} = \operatorname{span} \left\{ [S^{1}_{e}], [S^{1}_{e'}], [S^{1}_{\overline{e}}], [S^{1}_{\overline{e}'}]\right\}.\]
    Once more, elementary but tedious calculations show that the equality~\eqref{eq:lag} holds in this case as well,    
    so the Maslov index vanishes once again. This completes the proof.}
\end{proof}

\new{\begin{corollary}
  \label{cor:surjective-ap}
Let~$G$ be a balanced plumbing graph with vertices given by closed oriented surfaces, and let~$\varphi_P\colon H_1(P(G))\to\Z^\mu$ be a meridional homomorphism such
that~$(P(G),\varphi_P)$ bounds over~$\Z^\mu$.
      Then, it bounds a compact connected oriented~\(\Z^{\mu}\)-manifold~\((Y,f)\) such that~\(\pi_{1}(Y) = \Z^{\mu}\),~\(f\) is an isomorphism and~$\sign_\omega(Y)=0$ for
      all~\(\omega \in \mathbb{T}^{\mu}\).
\end{corollary}}
\begin{proof}
\new{By hypothesis, the~$\Z^\mu$-manifold~$(P(G),\varphi_P)$ bounds a~$\Z^\mu$-manifold~$(Z,\psi_Z)$. By Lemma~\ref{lem:rho=0}
and~\eqref{eq:rho-invariant}, we have
\[
0=\rho_\omega(P(G))=\sigma(Z)-\sigma_\omega(Z)
\]
for all~$\omega\in\mathbb{T}^\mu$. Via connected sums with copies of~$\mathbb{CP}^2$ or~$\overline{\mathbb{CP}^2}$, which leave the first homology group unaffected, it can be assumed that~$\sigma(Z)$ vanishes, hence all twisted signatures as well.}

\new{It remains to transform this 4-manifold in order to have its fundamental group isomorphic to~$\Z^\mu$.} Note that the homomorphism~$\psi_Z\colon H_1(Z)\to\Z^\mu$ is surjective: indeed, the homomorphism~$\varphi\colon\pi(P(G))\to\mathbb{Z}^\mu$ being meridional, it is surjective; since it factors through~$\psi_Z$,
    this latter homomorphism is surjective as well.
    Observe that there exists a finite collection of group elements \(g_1,\ldots,g_l \in \ker \psi_{Z}\) such that the smallest normal subgroup of \(\pi_{1}(Z)\) containing these elements is equal to \(\ker \psi_{Z}\).
    In other words, all conjugates of \(g_1,\ldots,g_l\) in \(\pi_{1}(Z)\) generate \(\ker \psi_{Z}\).
    Indeed, let \(p \colon \widetilde{Z} \to Z\) be the \(\Z^{\mu}\)-covering determined by \(\psi_{Z}\).
    Observe that \((\ker \psi_{Z})^{ab} = H_1(\widetilde{Z})\).
    Since \(Z\) is compact, it follows that~\(H_{1}(\widetilde{Z})\) is a finitely-generated \(\Z[\Z^{\mu}]\)-module.
    Let \(x_1,\ldots, x_{l}\) denote the generators of \(H_{1}(\widetilde{Z})\) as a \(\Z[\Z^{\mu}]\)-module.
    We can choose, \(g_1,\ldots,g_{l}\) to be preimages of \(x_1,\ldots,x_{l}\) under the quotient map
    \[\ker \psi_{Z} \to (\ker \psi_{Z})^{ab} = H_{1}(\widetilde{Z})\,.\]
    The manifold \(Y\) will be constructed by performing surgery on loops representing \(g_{1},\ldots,g_{l}\).
    To be more precise, suppose that the map \(f_{1} \colon S^{1} \to Z\) represents \(g_{1}\).
    Without loss of generality, we can assume that \(f_{1}\) is a smooth embedding.
    Let~\(N_{1}\) denote a closed tubular neighborhood of~\(f_{1}(S^{1})\), together with the identification \(\alpha_{1} \colon N_{1} \xrightarrow{\cong} S^{1} \times D^{3}\), where \(\alpha_{1}\) maps \(f_{1}(S^{1})\) to \(S^{1} \times \{0\}\).
    Consider the manifold
    \[Y_{1} = \overline{Z \setminus N_{1}} \cup_{\partial N_{1}} (D^{2} \times S^{2})\,,\]
    where we use the map \(\alpha_{1}\) to identify the boundary of \(N_{1}\) with the boundary of \(D^{2} \times S^{2}\).
    By the Seifert-van Kampen theorem, \(\pi_{1}(Y_{1})\) is isomorphic to the quotient of \(\pi_{1}(Z)\) by the normal subgroup generated by \(g_{1}\).
  Since \(g_{1}\) is in the kernel of \(\psi_{Z}\), one easily shows that \(Y_{1}\) is \(\Z^{\mu}\)-bordant to \(Z\).
    In particular, Novikov additivity implies that \(\sign(Y_{1})\)
    coincides with \(\sign(Z)\), which vanishes by hypothesis.
    Similarly, for \(\omega \in \mathbb{T}^{\mu} \setminus \{(1,1,\ldots,1)\}\), the fact that \(Y_{1}\) and \(Z\) are \(\Z^{\mu}\)-bordant implies that
    \[0 = \sign_{\omega}(Z \cup_{\partial} \overline{Y_{1}}) = \sign_{\omega}(Z) - \sign_{\omega}(Y_{1}) = -\sign_{\omega}(Y_{1})\,,\]
    where the first equality follows from~\cite[Theorem D.B]{Viro}, the second inequality from Novikov additivity, and the last equality from our assumptions.
  We can iterate the above procedure to obtain manifolds \(Y_{1}, Y_{2}, \ldots, Y_{l} = Y\) with the desired properties.
\end{proof}

\new{
\begin{corollary}
    \label{cor:rho-ind}
    For any~$\mu$-colored link~$L$ and any~$\omega\in\mathbb{T}^\mu$, the integer~$\rho_\omega(M_{L})=\rho(M_L,\varphi_\omega)$ does not depend on the choice of the meridional homomorphism~$\varphi \colon H_1(M_L)\to\Z^\mu$.
\end{corollary}}
\begin{proof}
    \new{Let \(\varphi,\varphi' \colon H_{1}(M_{L}) \to \Z^{\mu}\) be two meridional homomorphisms. By definition (recall Remark~\ref{rems:M_L}.2), their restrictions to \(X_{L} \subset M_{L}\) coincide and are equal to~\(\varphi_{X}\colon H_1(X_L)\to\Z^\mu\).
    Let us denote by~\(\varphi_{P}\) and~\(\varphi_{P}'\) the restriction of~\(\varphi\) and~\(\varphi'\) to \(P(L) \subset M_{L}\), respectively. For any~$\omega\in\mathbb{T}^\mu$, we write~$\varphi_\omega$ for the composition of~$\varphi$ with the map~$\Z^\mu\to S^1$ given by~$t_i\mapsto \omega_i$, and similarly for the other meridional homomorphisms.}
    
    \new{To show that \(\rho(M_{L},\varphi_{\omega})\) and \(\rho(M_{L},\varphi_{\omega}')\) coincide, we will use~\cite[Theorem~3.9]{toffoli} and its notation. Consider the oriented~$3$-manifolds
    \[X_{0} = X_{L}, \quad X_{1} = -P(L)\quad\text{and}\quad X_{2} = P(L)\,,
    \]
    which have common boundary~$\Sigma:=\partial X_1=-\partial X_0=-\partial X_2$. Since the restrictions of~$\varphi_{P}$ and~$\varphi'_{P}$ coincide on~$H_1(\partial P(L))$, they induce a map
    \[\varphi_{P} \cup \varphi_{P}' \colon H_{1}(D(L)) \to \Z^\mu,\]
    where \(D(L)\) is the oriented closed~$3$-manifold~\(-P(L) \cup_{\partial} P(L)\). Obviously, the same holds true for~$\varphi_{P,\omega}$ and~$\varphi'_{P,\omega}$, which induce
    \begin{equation}
    \label{eq:varphi-omega}
        \varphi_{P,\omega} \cup \varphi_{P,\omega}'=(\varphi_{P} \cup \varphi_{P}')_\omega\colon H_{1}(D(L)) \to S^1\,.
    \end{equation}
    Note also that by assumption, this map extends to~$H_1(X_L)$ via~$\varphi_{X,\omega}$, thus producing the maps $\varphi_\omega=\varphi_{P,\omega}\cup\varphi_{X,\omega}$ and~$\varphi'_\omega=\varphi'_{P,\omega}\cup\varphi_{X,\omega}$ on~$H_1(M_L)$. Therefore, we are in the setting of~\cite[Theorem~3.9]{toffoli}. It gives the equality
    \[\rho(D(L), \varphi_{P,\omega} \cup \varphi_{P,\omega}') = \rho(M_{L},\varphi_{\omega}) + \rho(-M_{L},\varphi_{\omega}') + C= \rho(M_{L},\varphi_{\omega}) - \rho(M_{L},\varphi_{\omega}') +C\,,\]
    where~$C$ is the difference of the associated Maslov indices on~$H_1(\Sigma;\C)$ and on~$H_1(\Sigma;\C^\omega)$. 
    Consequently, we need to prove that~$C$ and~\(\rho(D(L), \varphi_{P,\omega} \cup \varphi_{P,\omega}')\) both vanish.}

    \new{The kernel of the inclusion induced maps~$H_1(\Sigma;\C)\to H_1(P(L);\C)$ and~$H_1(\Sigma;\C)\to H_1(-P(L);\C)$ obviously coincide, so the Maslov index on~$H_1(\Sigma;\C)$ vanishes. The same holds true with twisted coefficients, as we know that the kernel of the inclusion induced map~$H_1(\Sigma;\C^\omega)\to H_1(P(L);\C^\omega)$
    does not depend on the choice of the meridional homomorphism~$\varphi_P$ (recall the second step in the proof of Lemma~\ref{lem:kernel-dim}). Therefore, the difference~$C$ of these Maslov indices vanishes.}
    

    \new{For the last step, observe that by~\eqref{eq:P-of-the-mirror}, we have
    \[
D(L)=-P(L)\cup_\partial P(L)\simeq P(\overline{L})\cup_\partial P(L)=P(G)\,,
    \]
    where~$G$ is the plumbing graph defined as follows: the vertex set of~$G$ is given by the colors $\{1,\dots,\mu\}$, the vertex~$i$ being decorated with the disjoint union of~2-spheres indexed by~$K\subset L_i$; given two components~$K,K'$ of different colors,
    the corresponding spheres are linked by~$\vert\lk(K,K')\vert$ positive edges and~$\vert\lk(K,K')\vert$ negative edges.
    Since this plumbing graph is balanced and the homomorphism \(\varphi_{P} \cup \varphi_{P}'\) meridional, we can apply Lemma~\ref{lem:rho=0} to the pair \((D(L),\varphi_{P} \cup \varphi_{P}')\): together with~\eqref{eq:varphi-omega}, it gives 
    \[\rho(D(L), \varphi_{P,\omega} \cup \varphi_{P,\omega}') = \rho(D(L),(\varphi_{P} \cup \varphi_{P}')_{\omega}) = \rho_{\omega}(D(L)) = 0
    \]
    for all~\(\omega \in \mathbb{T}^{\mu}\). This concludes the proof.}
\end{proof}

\new{The final lemma of this appendix makes use of the notations of Section~\ref{sub:extension}.
\begin{lemma}
    \label{lem:Maslov-ind}
    Let~$\widetilde{L}$ be a~$\mu$-colored link obtained from a~$\mu$-colored link~$L$ by adding a component which has zero linking number with all the other components. Then, the associated 4-dimensional~$\Z^\mu$-manifolds~$W_{\widetilde{F}},W_F$ defined as in~\eqref{eq:the-W-F-mfd} satisfy~$\sigma(W_{\widetilde{F}})=\sigma(W_F)$.
\end{lemma}}
\begin{proof}
\new{Recall from the proof of Theorem~\ref{thm:extension} that
    the Novikov-Wall theorem applied to the decomposition~$W_F=V_F\cup_{P(F)} Y_F$ implies~$\sigma(W_{F}) = \mathit{Maslov}(\mathcal{L}_{-},\mathcal{L}_{0},\mathcal{L}_{+})$
    for appropriate Lagrangian subspaces~$\mathcal{L}_{-},\mathcal{L}_{0},\mathcal{L}_{+}$ of~$H_{1}(\partial X_{L};\C)$ which only depend
    on the colors and linking numbers of the components of~$L$.
    More precisely, if we denote by~$m_K$ (resp.~$\ell_K$) a meridian (resp. Seifert longitude) of the component~$K\subset L$ (recall Section~\ref{sec:plumbed}), then we have
    \begin{equation}
        \label{eq:XL}
    H_{1}(\partial X_{L};\C) = \bigoplus_{K\subset L}\left(\C m_{K}\oplus\C\ell_{K}\right)\,.
        \end{equation}
    Moreover, by definition of the Seifert longitude~\eqref{eq:Seifert}, we get
    \begin{equation}
        \label{eq:L+}
    \mathcal{L}_{+} \coloneqq \ker \left( H_{1}(\partial X_{L};\C) \to H_{1}(X_{L};\C) \right)=\operatorname{span} \Big\{ \ell_{K} - \sum_{K'\subset L} \lk(K,K')\,m_{K'} \;\Big\vert\; K\subset L\Big\}
     \end{equation}
    with~$\lk(K,K)\coloneqq-\sum_{K'\subset L_i\setminus K}\lk(K,K')$ if~$K\subset L_i$. Denoting this inclusion by~$c(K)=i$, Lemma~\ref{lem:P(G)} yields 
    \begin{align}
     \begin{split}
        \label{eq:L-}
    \mathcal{L}_{-} &\coloneqq \ker \left( H_{1}(\partial X_{L};\C) \to H_{1}(P(L);\C) \right)\\
    &=\operatorname{span} \Big\{ \ell_{K} - \sum_{K'\subset L,\,c(K')\neq c(K)} \lk(K,K')\,m_{K'} \;\Big\vert\; K\subset L\Big\}\,.
     \end{split}
    \end{align}
    Finally, one more application of Lemma~\ref{lem:P(G)} gives
     \begin{align}
     \begin{split}
     \label{eq:L0}
    \mathcal{L}_{0} &\coloneqq \ker \left( H_{1}(\partial X_{L};\C) \to H_{1}(P(F);\C) \right)\\
    &=\operatorname{span} \Big\{ \sum_{K\subset L_i}\ell_{K} - \sum_{K'\subset L\setminus L_i} \lk(K,K')\,m_{K'},\;\{m_K-m_{K_i^0}\,\vert\, K\subset L_i\setminus K_i^0\} \;\Big\vert\; 1\le i\le\mu\Big\}\,,
    \end{split}
    \end{align}
    where~$K^0_i$ is an arbitrary fixed component of~$L_i$.}
    
    \new{Let us now consider a link~$\widetilde{L}=L\cup\widetilde{K}$ with~$\lk(K,\widetilde{K})=0$ for all~$K\subset L$. The goal is to compare the Lagrangians~$\widetilde{\mathcal{L}}_-,\widetilde{\mathcal{L}}_0,\widetilde{\mathcal{L}}_+$ for~$\widetilde{L}$ with their counterparts for~$L$.
    By~\eqref{eq:XL},\eqref{eq:L+} and~\eqref{eq:L-}, we have
    \begin{align*}
H_1(\partial X_{\widetilde{L}};\C)&= H_1(\partial X_L;\C)\oplus\C  m_{\widetilde{K}}\oplus\C\ell_{\widetilde{K}}\,,\\
\widetilde{\mathcal{L}}_{+}&=\mathcal{L}_{+}\oplus \operatorname{span}\{\ell_{\widetilde{K}}\}\,,\\
\widetilde{\mathcal{L}}_{-}&=\mathcal{L}_{-}\oplus \operatorname{span}\{\ell_{\widetilde{K}}\}\,.
\end{align*}
Also, assuming without loss of generality that~$c(\widetilde{K})=1$, Equation~\eqref{eq:L0} yields
\[
\widetilde{\mathcal{L}}_{0}=\mathcal{L}'_{0}\oplus \operatorname{span}\{m_{\widetilde{K}}-m_{K_1^0}\}\,,
\]
where~$K_1^0$ is some fixed component of~$L_1$, and~$\mathcal{L}'_{0}$
denotes the Lagrangian subspace~$\mathcal{L}_{0}$ given by~\eqref{eq:L0}
with the basis vector~$x_1\coloneqq\sum_{K\subset L_1}\ell_{K} - \sum_{K'\subset L\setminus L_1} \lk(K,K') m_{K'}$ replaced by~$x_1'\coloneqq x_1+\ell_{\widetilde{K}}$.
Using these three equalities, a straightforward computation yields
\begin{equation}
\label{eq:Lag}
    \big(\widetilde{\mathcal{L}}_{-}+\widetilde{\mathcal{L}}_{0}\big)\cap\widetilde{\mathcal{L}}_{+}=\left(\left({\mathcal{L}}_{-}+{\mathcal{L}}_{0}\right)\cap{\mathcal{L}}_{+}\right)\oplus\operatorname{span}\{\ell_{\widetilde{K}}\}\,.
\end{equation}
To compare the corresponding Maslov indices, we now relate the form~$\widetilde{f}$ on this later space to the form~$f$ on~$\big({\mathcal{L}}_{-}+{\mathcal{L}}_{0}\big)\cap{\mathcal{L}}_{+}$,
writing~$(\ell,m,m_0)\coloneqq(\ell_{\widetilde{K}}, m_{\widetilde{K}}, m_{K_1^0})$ for simplicity.
By~\eqref{eq:Lag}, any~$\widetilde{a}\in\big(\widetilde{\mathcal{L}}_{-}+\widetilde{\mathcal{L}}_{0}\big)\cap\widetilde{\mathcal{L}}_{+}$
can be written~$\widetilde{a}=a+\lambda\ell$ with~$a\in\left({\mathcal{L}}_{-}+{\mathcal{L}}_{0}\right)\cap{\mathcal{L}}_{+}$ and~$\lambda\in\C$. It can also be written~$\widetilde{a}=\widetilde{a}_-+\widetilde{a}_0$ with~$\widetilde{a}_-=a_-+\lambda_-\ell\in\widetilde{\mathcal{L}}_{-}=\mathcal{L}_{-}\oplus \operatorname{span}\{\ell\}$,~$(a_-\in\mathcal{L}_-,~\lambda_-\in\C)$ and~$\widetilde{a}_0=a_0'+\lambda_0(m-m_0)\in\widetilde{\mathcal{L}}_{0}=\mathcal{L}'_{0}\oplus \operatorname{span}\{m-m_0\}$,~$(a_0'=a_0+\lambda_1\ell\in\mathcal{L}_0',~a_0\in\mathcal{L}_0,~\lambda_1\in\C)$. Gathering all these equalities yields
\[
a+\lambda\ell=\widetilde{a}=(a_-+a_0)+(\lambda_-+\lambda_1)\ell+\lambda_0(m-m_0)\in \left(\left({\mathcal{L}}_{-}+{\mathcal{L}}_{0}\right)\cap{\mathcal{L}}_{+}\right)\oplus\operatorname{span}\{\ell\}\,,
\]
which implies~$a=a_-+a_0$ and~$\lambda_0=0$. In particular, we have~$\widetilde{a}_0=a_0'=a_0+\lambda_1\ell$.
Hence, the form~$\widetilde{f}$ maps~$\widetilde{a}$ as above and~$\widetilde{b}=b+\kappa\ell$ with~$b\in\left({\mathcal{L}}_{-}+{\mathcal{L}}_{0}\right)\cap{\mathcal{L}}_{+}$,~$\kappa\in\C$ to
\[
\widetilde{f}(\widetilde{a},\widetilde{b})=\widetilde{a}_0\cdot\widetilde{b}=(a_0+\lambda_1\ell)\cdot(b+\kappa\ell)=a_0\cdot b=f(a,b)\,.
\]
In conclusion, we have~$\widetilde{f}=f\oplus(0)$, implying the desired equality
    \[
    \sigma(W_{\widetilde{F}})=\mathit{Maslov}(\widetilde{\mathcal{L}}_{-},\widetilde{\mathcal{L}}_{0},\widetilde{\mathcal{L}}_{+})=\sigma(\widetilde{f})=\sigma(f)=\mathit{Maslov}(\mathcal{L}_{-},\mathcal{L}_{0},\mathcal{L}_{+})=\sigma(W_{F})\,.
    \]
This concludes the proof.}
\end{proof}

\section{Representing intersection forms by matrices}
\label{sec:repr-inters-forms}

The purpose of this appendix is to prove Lemma~\ref{lemma:representing-int-forms}, whose statement we now repeat for the reader's convenience.

Set~\(\Lambda_{\mu}=\C[\Z^{\mu}]=\C[t_{1}^{\pm1},\ldots,t_{\mu}^{\pm1}]\) and let~\(Q(\Lambda_{\mu})\) be the quotient field of~\(\Lambda_{\mu}\).

  \begin{lemma}\label{lemma:representing-int-forms-2}
  Suppose that~\((W,\psi)\) is a compact connected oriented~\(4\)-manifold over~\(\Z^{\mu}\) with connected boundary, such that the composition
  \[H_1(\partial W) \to H_1(W) \xrightarrow{\psi} \Z^{\mu}\]
  is surjective and~\(H_{1}(W;\Lambda_{\mu}) = 0\).
  Then, for any~\(j=1,\ldots,\mu\),
  there exists a Hermitian matrix~\(H_{j}\) over~\(Q(\Lambda_{\mu})\) such that for any~\(\omega \in U_{j}\coloneqq\{\omega \in \mathbb{T}^{\mu} \colon \omega_{j} \neq 1\}\), the intersection form
  \[Q_{\omega} \colon H_{2}(W;\C^{\omega}) \times H_{2}(W;\C^{\omega}) \to \C\]
  is represented by~\(H_{j}(\omega)\).
  Furthermore, if~$\mu=1$, then~$Q_\omega$ is represented by a Hermitian matrix~$H(\omega)$ for all~$\omega\in S^1$.
\end{lemma}

The proof of this lemma being rather technical,
we divide it into several steps.
In Section~\ref{sec:algebr-preliminaries}, we are concerned with naturality of twisted intersection forms, see Lemma~\ref{lemma:naturality-intersection-forms}, which is a key point of the proof.
Furthermore, we review all the ingredients needed to
prove this naturality
statement, namely twisted (co)homology, evaluation maps, the construction of twisted intersection forms following~\cite{conwayTwistedSignaturesFibered2021,FriedlKim,KirkLivington},
as well as
the Universal Coefficient Spectral Sequence
~\cite{Levine,mcclearyUserGuideSpectral2000}.
In Section~\ref{sec:homol-comp}, we compute the twisted homology module~\(H_{\ast}(W;\Lambda_{\mu,j})\),
where~\(\Lambda_{\mu,j} = \Lambda_{\mu} \left[ (t_{j}-1)^{-1}
\right]\) for \(j=1,\ldots,\mu\);
in particular, we prove that~$H_2(W;\Lambda_{\mu,j})$ is a free~\(\Lambda_{\mu,j}\)-module.
In Section~\ref{sec:proof-lemma-homology-computations} we combine results from Sections~\ref{sec:algebr-preliminaries} and~\ref{sec:homol-comp} to give a proof of Lemma~\ref{lemma:representing-int-forms-2}.
Roughly speaking, the desired Hermitian matrices~\(H_{j}\) can be taken to be matrices representing twisted intersection forms on~\(H_{2}(W;\Lambda_{\mu,j})\).

\subsection{Naturality of intersection forms}
\label{sec:algebr-preliminaries}

In this section, we recall the definition of twisted homology and cohomology, the statement of the Universal Coefficient Spectral Sequence, the definition of the twisted intersection form, and prove its naturality.

\subsubsection*{Twisted homology and cohomology}

Recall that the ring~\(\Lambda_{\mu}\) admits an involution
\[\overline{(-)} \colon \Lambda_{\mu} \to \Lambda_{\mu},\]
which acts by the complex conjugation on scalars and maps each indeterminate~\(t_{j}\) to its inverse.
If~\(N\) is a (left)~\(\Lambda_{\mu}\)-module, then we define the {\em transpose\/} of~\(N\), denoted by~\(N^{\tr}\), to be the
(right)~\(\Lambda_{\mu}\)-module with the same underlying~\(\C\)-vector space as~$N$,
but with the action of~\(\Lambda_{\mu}\)
given by
\[N^{\tr}\times\Lambda_{\mu} \owns (n,\lambda) \mapsto n\cdot\lambda=\overline{\lambda}\cdot n \in N^{\tr}\,.\]

Let \(X\) be a finite connected pointed CW-complex with \(\pi_{1}(X) \cong \Z^{\mu}\), and let \(p \colon \widetilde{X} \to X\) denote the universal covering of \(X\).
If \(Y \subset X\) is a subcomplex containing the basepoint, then the action of~\(\pi_{1}(X)\) equips the 
chain complex \(C_{\ast}(\widetilde{X}, p^{-1}(Y);\C)\) with the structure of a (left)~\(\Lambda_{\mu}\)-module.
Given any (right)~\(\Lambda_{\mu}\)-module \(M\),
let us define the chain and cochain complexes of \(\Lambda_{\mu}\)-modules
\begin{align*}
  C_{\ast}(X,Y;M) &= M \otimes_{\Lambda_{\mu}} C_{\ast}(\widetilde{X},p^{-1}(Y)), \\
  C^{\ast}(X,Y;M) &= \hom_{\Lambda_{\mu}}(C_{\ast}(\widetilde{X},p^{-1}(Y))^{\tr},M).
\end{align*}
The homology~\(H_{\ast}(X,Y;M)\) (resp. cohomology~\(H^{\ast}(X,Y;M)\))
of the above (co)chain complex is called the {\em twisted
(co)homology\/} of~$X$.
Note that both \(H_{\ast}(X,Y;M)\) and \(H^{\ast}(X,Y;M)\) are modules
over~\(\Lambda_{\mu}\).
Furthermore, if~\(M\) is an~\((R,\Lambda_{\mu})\)-bimodule for some ring~\(R\), then~\(H_{\ast}(X,Y;M)\) and~\(H^{\ast}(X,Y;M)\) inherit the structure of left~\(R\)-modules.

\begin{remark}
\begin{enumerate}
\item   The ring~$\Lambda_\mu=\C[\Z^\mu]$ being commutative, there is no problem with distinguishing left and right modules, hence the
parenthesis above around these words.
  In the general setting however,~$M$ is required to be an~\((R,\C[\pi_1(X)])\)-bimodule for some ring~\(R\), hence we tensor by~\(M\) from the left
  in the definition of~\(C_{\ast}(X,Y;M)\)
  to be consistent with sources~\cite{conwayTwistedSignaturesFibered2021,CFT,CNT}.
  \item
  It is for the same consistency reasons that we transpose
  the cellular chain complex in the definition of~\(C^{\ast}(X,Y;M)\), rather than the module~$M$.
  Note however that since~\(\Lambda_{\mu}\) is commutative, for any two~\(\Lambda_{\mu}\)-modules~\(M\) and~\(N\), we have
\[\hom_{\Lambda_{\mu}}(N^{\tr},M) = \hom_{\Lambda_{\mu}}(N,M^{\tr}) = \hom_{\Lambda_{\mu}}(N,M)^{\tr}\,,\]
which consists of the additive maps~$f \colon N \to M$ such that~$f(\lambda \cdot n) = \overline{\lambda} \cdot f(n)$ for all~$\lambda \in \Lambda_{\mu}$ and~$n\in N$.
  This leads to
  \[H^{\ast} \left( \hom_{\Lambda_{\mu}}(C_{\ast}(\widetilde{X},p^{-1}(Y))^{\tr},M)\right) = H^{\ast} \left( \hom_{\Lambda_{\mu}}(C_{\ast}(\widetilde{X},p^{-1}(Y)),M)\right)^{\tr}\,,\]
  so the transposed module in the definition of~\(C^{\ast}(X,Y;M)\) simply changes the resulting cohomology groups by a transposition.
\end{enumerate}
\end{remark}

Computations of twisted (co)homology modules are usually performed with the aid of the Universal Coefficient Spectral Sequence (UCSS) whose statement we now recall, referring the reader to~\cite[Theorem 2.3]{Levine} and~\cite[Theorem 2.20]{mcclearyUserGuideSpectral2000} for a proof.

\begin{theorem}[Universal Coefficient Spectral Sequence]\label{thm:UCSS}
Let \(R\) and \(S\) be associative rings with unit.
  Let \(C_{\ast}\) be a chain complex of finitely generated free left \(R\)-modules.
  If \(M\) is any \((S,R)\)-bimodule, then there are
  natural spectral sequences of left \(S\)-modules
  \begin{align*}
    E^{2}_{p,q} &= \Tor_{p}^{R}(M,H_{q}(C_{\ast})) \Rightarrow 
    H_{p+q}(M\otimes_R C_*), \\
    E_{2}^{p,q} &= \Ext^{q}_{R}(H_{p}(C_{\ast})^{\tr},M) \Rightarrow
    H^{p+q}(\hom_{\mathrm{right}-R}(C_*^{\tr},M))
  \end{align*}
  with differentials of degree \((-r,r-1)\) and \((1-r,r)\), respectively.
\end{theorem}

\subsubsection*{Twisted intersection forms}
We now turn to the definition of the twisted intersection form,
which requires two ingredients: twisted Poincar\'e-Lefschetz duality,
and the evaluation map.

\smallskip

Let \(X\) be a connected compact oriented smooth~\(4\)-manifold with \(\pi_{1}(X) \cong \Z^{\mu}\).
 For any \(\Lambda_{\mu}\)-module \(M\), one can define the {\em twisted Poincar\'e-Lefschetz duality}
\[\mathit{PD}_{M} \colon H_{k}(X,\partial X;M) \xrightarrow{\cong} H^{4-k}(X;M)\]
as the inverse of the isomorphism
\[(-) \cap [X,\partial X] \colon H^{4-k}(X;M) \xrightarrow{\cong} H_{k}(X,\partial X;M)\]
induced by the cap product with the fundamental class \([X,\partial X] \in H_{4}(X;\C)\), see~\cite[Section 2.4]{conwayTwistedSignaturesFibered2021}.

\smallskip

We now come to the definition of the evaluation map.
Let~\(A\) be a commutative~\(\C\)-algebra with unit, and involution 
denoted by~$a\mapsto\overline{a}$.
Let \(\psi \colon \Lambda_{\mu} \to A\) be a homomorphism of algebras with involutions which preserves units.
Observe that \(A\) becomes an \((A,\Lambda_{\mu})\)-bimodule via~\(\psi\).
Let \(M\) be an \((A,\Lambda_{\mu})\)-bimodule an let \(N\) be an \((A,A)\)-bimodule with involution.
Suppose that 
we are given a~\(\Z^{\mu}\)-equivariant, sesquilinear, nonsingular pairing
\[\theta \colon M \times M \to N\,.\]
In other words, this pairing~\(\theta\) satisfies the following conditions:
\begin{enumerate}
    \item for any \(g \in \Z^{\mu}\) and any \(m_{1},m_{2} \in M\), we have~\(\theta(m_{1} \cdot g,m_{2} \cdot g) = \theta(m_{1},m_{2})\);
    \item \(\theta\) is \(A\)-linear in the first variable
    and satisfies~\(\theta(m_{1},m_{2}) = \overline{\theta(m_{2},m_{1})}\) for all \(m_{1},m_{2} \in M\);
    \item the adjoint map
  \[\theta^{D} \colon M \to \hom_{\text{left-}A}(M,N)^{\tr}\]
  defined by~$\theta^D(m_1)(m_2)=\theta(m_1,m_2)$ is an isomorphism of left \(A\)-modules.
\end{enumerate}
Given this piece of data, we can construct an associated evaluation map
as follows. Firstly, consider the chain map
\begin{align*}
  \kappa \colon \hom_{\Lambda_{\mu}}(C_{\ast}(\widetilde{X})^{\tr},M) &\to \hom_{A}(M \otimes_{\Lambda_{\mu}} C_{\ast}(\widetilde{X}),N)^{\tr} \\
  f &\mapsto (m \otimes \sigma \mapsto \theta(m,f(\sigma)))\,,
\end{align*}
where \(m \in M\) and \(\sigma \in C_{\ast}(\widetilde{X})\).
By nonsingularity of \(\theta\), this is an isomorphism of cochain complexes of~$A$-modules.
Secondly, the edge homomorphism in the UCSS yields a map
\[E \colon H^{k}(\hom_{A}(M \otimes_{\Lambda_{\mu}} C_{\ast}(\widetilde{X}),N)^{\tr}) \to \hom_{A}(H_{k}(X;M),N)^{\tr}.\]
The {\em evaluation map\/} is defined as the composition
\[\ev(\theta) \colon H^{k}(X;M) \xrightarrow{\kappa_*} H^{k}(\hom_{A}(M \otimes_{\Lambda_{\mu}} C_{\ast}(\widetilde{X}),N)^{\tr}) \xrightarrow{E} \hom_{A}(H_{k}(X;M),N)^{\tr}\,,\]
where~$\kappa_*$ is the isomorphism of~$A$-modules induced by the chain map~$\kappa$.

\smallskip

We are finally ready to define the twisted intersection form.
By composing the evaluation map with Poincar\'e-Lefschetz duality and the map induced by the inclusion of~\((X,\emptyset)\) in~\((X,\partial X)\),
we obtain a homomorphism of~$A$-modules
\[Q(\theta)^{D} \colon H_{2}(X;M) \xrightarrow{} H_{2}(X,\partial X;M) \xrightarrow{\mathit{PD}_{M}} H^{2}(X;M) \xrightarrow{\ev(\theta)} \hom_{A}(H_{2}(X;M),N)^{\tr}\,.\]
The associated Hermitian form
\[Q(\theta) \colon H_{2}(X;M) \times H_{2}(X;M) \to N, \quad Q(\theta)(x,y) = Q(\theta)^{D}(x)(y)\]
is the {\em twisted intersection form\/} of~$X$.

In our setting, the most relevant examples of~\(\Z^{\mu}\)-equivariant sesquilinear pairings are the ones given below.

\begin{example}\label{example:Hermitian-pairings}
  \begin{enumerate}
    \item For any~\(j=1,2,\ldots,\mu\), set \(\Lambda_{\mu,j} = \Lambda_{\mu} \left[(t_{j}-1)^{-1} \right]\),
    i.e., \(\Lambda_{\mu,j}\) is constructed by adjoining the inverse of~\(t_{j}-1\) to \(\Lambda_{\mu}\).
    Note that the involution on~\(\Lambda_{\mu}\) extends naturally to an
    involution on~\(\Lambda_{\mu,j}\).
    Set~\(A=\Lambda_{\mu,j}\), let~\(\psi\colon\Lambda_\mu\to A\) be the localization map, and set~\(M=N=\Lambda_{\mu,j}\). Then, the pairing
    \[\theta_j \colon \Lambda_{\mu,j} \times \Lambda_{\mu,j} \to \Lambda_{\mu,j}, \quad \theta_j(\lambda_{1},\lambda_{2}) =\lambda_{1} \overline{\lambda_{2}}\]
    is clearly nonsingular, \(\Z^{\mu}\)-equivariant,
    and sesquilinear over~$\Lambda_{\mu,j}$.
    We denote the associated twisted intersection form by
    \[Q_{j}(X) \colon H_{2}(X;\Lambda_{\mu,j}) \times H_{2}(X;\Lambda_{\mu,j}) \to \Lambda_{\mu,j}\,.\]
      \item For any~\(\omega \in \mathbb{T}^{\mu}\),
    set~\(A=\C\), let
    \(\psi \colon \Lambda_{\mu} \to \C\) be given by~$t_i\mapsto\omega_i$, and set~$M=N=\C^{\omega}$.
    Then, the pairing
    \[\theta_{\omega} \colon \C^{\omega} \times \C^{\omega} \to \C^{\omega}, \quad \theta_{\omega}(z_{1},z_{2}) = z_{1}\overline{z_{2}}\]
    is nonsingular, \(\Z^{\mu}\)-equivariant and sesquilinear.
    We denote
    the associated Hermitian intersection form by
    \[Q_{\omega}(X) \colon H_{2}(X;\C^{\omega}) \times H_{2}(X;\C^{\omega}) \to \C^{\omega}\,.\]
  \end{enumerate}
\end{example}

These two intersection forms~$Q_j(X)$ and~$Q_{\omega}(X)$ are related in the following
way, a fact of crucial importance for the proof of Lemma~\ref{lemma:representing-int-forms-2}.

\begin{lemma}[Naturality of intersection forms]\label{lemma:naturality-intersection-forms}
  Fix \(j=1,\ldots,\mu\) and suppose that \(\omega \in \mathbb{T}^{\mu}\) satisfies~$\omega_j\neq 1$,
  yielding a homomorphism \(\phi_{\omega} \colon \Lambda_{\mu,j} \to \C^{\omega}\)
  via~$t_i\mapsto\omega_i$.
  Then, for any \(x,y \in H_{2}(X;\Lambda_{\mu,j})\), the following equality is satisfied
  \[Q_{\omega}(X)(\phi_{\omega,\ast}(x),\phi_{\omega,\ast}(y)) = (\phi_{\omega} \circ Q_{j}(X))(x,y),\]
  where \(\phi_{\omega,\ast} \colon H_{2}(X;\Lambda_{\mu,j}) \to H_{2}(X;\C^{\omega})\) is the map induced by \(\phi_{\omega}\).
\end{lemma}
\begin{proof}
  Observe that the
  statement is equivalent to the commutativity of the diagram
  \begin{center}
    \begin{tikzcd}
      H_{2}(X;\Lambda_{\mu,j}) \arrow[r] \arrow[dd,"\phi_{\omega,\ast}"] & H_{2}(X,\partial X;\Lambda_{\mu,j}) \arrow[r,"\mathit{PD}_{\Lambda_{\mu,j}}"] \arrow[dd,"\phi_{\omega,\ast}"] &
      H^{2}(X;\Lambda_{\mu,j}) \arrow[r,"\ev(\theta_j)"] \arrow[dd,"\phi_{\omega,\ast}"] & \hom_{\Lambda_{\mu,j}}\left( H_{2}(X;\Lambda_{\mu,j}),
        \Lambda_{\mu,j}\right)^{\tr} \arrow[d,"\phi_{\omega,\ast}"] \\
      & & & \hom_{\Lambda_{\mu,j}} \left( H_{2}(X;\Lambda_{\mu,j}),\C^{\omega} \right)^{\tr} \\
      H_{2}(X;\C^{\omega}) \arrow[r] & H_{2}(X,\partial X;\C^{\omega}) \arrow[r,"\mathit{PD}_{\C^{\omega}}"] & H^{2}(X;\C^{\omega}) \arrow[r,"\ev(\theta_\omega)"] & \hom_{\C}\left( H_{2}(X;\C^{\omega}), \C\right)^{\tr}\,, \arrow[u,"\phi_{\omega}^{\ast}"]
    \end{tikzcd}
  \end{center}
  with~$\theta_j$ and~$\theta_\omega$ as in Example~\ref{example:Hermitian-pairings}.
  The commutativity of the left square follows from naturality of the inclusion-induced map
  of the pair \((X,\partial X)\).
  The commutativity of the middle square follows from the naturality of the twisted Poincar\'e-Lefschetz duality isomorphism, see e.g.~\cite[Lemma 2.10]{conwayTwistedSignaturesFibered2021}.
  Therefore, we are left with the proof of the commutativity of the right pentagonal diagram.

  For that purpose, consider the following diagram of cochain complexes
  \begin{center}
    \begin{tikzcd}
      \hom_{\Lambda_{\mu,j}}\left( C_{\ast}(\widetilde{X}), \Lambda_{\mu,j} \right) \arrow[r,"\kappa"] \arrow[dd,"\phi_{\omega,\ast}"] & \hom_{\Lambda_{\mu,j}} \left( \Lambda_{\mu,j} \otimes_{\Lambda_{\mu}}
        C_{\ast}(\widetilde{X}), \Lambda_{\mu,j}\right) \arrow[d,"\phi_{\omega,\ast}"] \\
      & \hom_{\Lambda_{\mu,j}}\left( \Lambda_{\mu,j} \otimes_{\Lambda_{\mu}} C_{\ast}(\widetilde{X}), \C^{\omega} \right) \\
      \hom_{\Lambda_{\mu,j}} \left(C_{\ast}(\widetilde{X}), \C^{\omega} \right) \arrow[r,"\kappa_{\omega}"] & \hom_{\C} \left( \C^{\omega} \otimes C_{\ast}(\widetilde{X}), \C \right)\,. \arrow[u,"\phi_{\omega}^{\ast}"]
    \end{tikzcd}
  \end{center}
  For any \(f \in \hom_{\Lambda_{\mu,j}}\left( C_{\ast}(\widetilde{X}),\Lambda_{\mu,j}\right)\), the right-down composition yields
  \[(\phi_{\omega,\ast} \circ \kappa)(f)(\lambda \otimes \sigma) = \phi_{\omega} \left( \theta_j(\lambda, f(\sigma))\right)\]
  for all \(\lambda \in \Lambda_{\mu,j}\) and \(\sigma \in C_{\ast}(\widetilde{X})\).
  On the other hand, the down-right-up composition gives
  \[(\phi_{\omega}^{\ast} \circ \kappa_{\omega} \circ \phi_{\omega,\ast})(f)(\lambda \otimes \sigma) = (\kappa_{\omega} \circ \phi_{\omega}^{\ast})(\phi_{\omega} \circ f) (\lambda \otimes \sigma) =
    \theta_{\omega}(\phi_{\omega}(\lambda),\phi_{\omega}(f(\sigma)))\,.\]
  Since for any~\(\lambda,\lambda' \in \Lambda_{\mu,j}\) we have
  \[\phi_{\omega}(\theta_j(\lambda,\lambda')) = \theta_{\omega}(\phi_{\omega}(\lambda),\phi_{\omega}(\lambda')),\]
  it follows that
  the diagram is commutative. To conclude the proof, consider the
  following diagram
  \begin{center}
    \begin{tikzcd}
      H^{2}(X;\Lambda_{\mu,j}) \arrow[r,"\kappa_*"] \arrow[dd,"\phi_{\omega,\ast}"] & H^{2} \left( \hom_{\Lambda_{\mu,j}}(\Lambda_{\mu,j} \otimes_{\Lambda_{\mu}} C_{\ast}(\widetilde{X}), \Lambda_{\mu,j})  \right)
      \arrow[r,"E_{1}"] \arrow[d,"\phi_{\omega,\ast}"] & \hom_{\Lambda_{\mu,j}}
      (H_{2}(X;\Lambda_{\mu,j}),\Lambda_{\mu,j}) \arrow[d,"\phi_{\omega,\ast}"] \\
      & H^{2} \left( \hom_{\Lambda_{\mu,j}}(\Lambda_{\mu,j} \otimes_{\Lambda_{\mu}} C_{\ast}(\widetilde{X}), \C^{\omega})  \right) \arrow[r,"E_{2}"] & \hom_{\Lambda_{\mu,j}} (H_{2}(X;\Lambda_{\mu,j}),\C^{\omega}) \\
      H^{2}(X;\C^{\omega}) \arrow[r,"\kappa_{\omega,*}"] & H^{2} \left( \hom_{\C^{\omega}}(\C^{\omega} \otimes_{\Lambda_{\mu}} C_{\ast}(\widetilde{X}), \C^{\omega}) \right) \arrow[r,"E_{3}"] \arrow[u,"\phi_{\omega}^{\ast}"] & \hom_{\C}
      (H_{2}(X;\C^{\omega}),\C)\,. \arrow[u,"\phi_{\omega}^{\ast}"]
    \end{tikzcd}
  \end{center}
  By our previous considerations, the left pentagonal
  diagram is commutative.
  Furthermore, by naturality of the UCSS, the upper and lower right squares are commutative.
  Since the horizontal compositions are equal to the respective evaluation maps, the lemma follows.
\end{proof}

\begin{remark}
    \label{rem:naturality}
    Similarly to Example~\ref{example:Hermitian-pairings}, one can set~$A=M=N=\Lambda_\mu$ (with~$\psi=\mathit{id}_{\Lambda_\mu}$) and consider
    the non-singular sesquilinear pairing~$\theta\colon\Lambda_\mu\times\Lambda_\mu\to\Lambda_\mu$ given by~$\theta(\lambda_1,\lambda_2)=\lambda_1\overline{\lambda_2}$. The associated twisted intersection form
    \[
        Q(X) \colon H_{2}(X;\Lambda_{\mu}) \times H_{2}(X;\Lambda_{\mu}) \to \Lambda_{\mu}
    \]
    obviously enjoys the same naturality property as~$Q_j(X)$ with respect
    to~$Q_\omega(X)$, but without any restriction on~$\omega\in\mathbb{T}^\mu$.
\end{remark}

\subsection{Twisted homology of \textit{{W}}}
\label{sec:homol-comp}

Throughout this section, we fixed an index~\(j=1,2,\ldots,\mu\) and set \(\Lambda_{\mu,j} = \Lambda_{\mu}\left[ (t_{j}-1)^{-1} \right]\) as in Example~\ref{example:Hermitian-pairings}.
The purpose of this section is to prove the following lemma.

\begin{lemma}\label{lemma:freeness-second-homology}
  If~$W$ is a~\(4\)-manifold as in Lemma~\ref{lemma:representing-int-forms-2}, then \(H_{2}(W;\Lambda_{\mu,j})\) is a free~\(\Lambda_{\mu,j}\)-module.
\end{lemma}

Its proof requires one more preliminary statement.

\begin{lemma}\label{lemma:computations-relative-homology}
  If~$W$ is a~\(4\)-manifold as in Lemma~\ref{lemma:representing-int-forms-2}, then \(H_{i}(W,\partial W;\Lambda_{\mu,j}) \cong H^{4-i}(W;\Lambda_{\mu,j}) = 0\) unless \(i=2\),
  and \(H_{i}(W;\Lambda_{\mu,j}) \cong H^{4-i}(W;\partial W;\Lambda_{\mu,j}) = 0\) unless \(i=2\).
\end{lemma}
\begin{proof}[Proof of Lemma~\ref{lemma:computations-relative-homology}]
  First note that our assumptions imply~\(H_{i}(W,\partial W;\Lambda_{\mu,j}) = 0\) for \(i=0,1\).
  Indeed, we have
  \begin{equation}
    \label{eq:vanishing-H0}
    H_{0}(W; \Lambda_{\mu,j}) \cong \Lambda_{\mu,j} \otimes_{\Lambda_{\mu}} H_{0}(W; \Lambda_{\mu}) \cong \Lambda_{\mu,j} \otimes_{\Lambda_{\mu}} \C = 0\,,
  \end{equation}
which implies~\(H_{0}(W,\partial W;\Lambda_{\mu,j}) = 0\).
Similarly, we have~$H_{0}(\partial W; \Lambda_{\mu,j})=0$.
  Furthermore, since we assume~\(H_{1}(W;\Lambda_{\mu}) = 0\) and since~\(\Lambda_{\mu,j}\) is a flat \(\Lambda_{\mu}\)-module, it follows that
  \begin{equation}
    \label{eq:vanishing-H1}
    H_{1}(W;\Lambda_{\mu,j}) \cong \Lambda_{\mu,j} \otimes_{\Lambda_{\mu}} H_{1}(W;\Lambda_{\mu}) = 0.
  \end{equation}
  The vanishing of~\(H_{1}(W;\Lambda_{\mu,j})\) and of~\(H_{0}(\partial W; \Lambda_{\mu,j})\) implies that \(H_{1}(W,\partial W; \Lambda_{\mu,j}) = 0\), as desired.

  Now, consider the second part of Theorem~\ref{thm:UCSS}
  applied to~$M=\Lambda_{\mu,j}$ and~$C_*=C_*(\widetilde{W},p^{-1}(\partial W))$: it yields the spectral sequence
  \[E_{2}^{p,q} = \Ext^{q}_{\Lambda_{\mu,j}}(H_{p}(W,\partial W;\Lambda_{\mu,j})^{\tr},\Lambda_{\mu,j}) \Rightarrow H^{p+q}(W, \partial W;\Lambda_{\mu,j})\,.\]
  Since \(H_{0}(W,\partial W;\Lambda_{\mu,j}) = H_{1}(W,\partial W;\Lambda_{\mu,j})= 0\), it follows that~\(E_{2}^{p,q} = 0\) for \(p=0,1\), which implies \(H^{i}(W,\partial W;\Lambda_{\mu,j}) = 0\) for \(i=0,1\).
  
  Similarly, we have a spectral sequence
  \[E_{2}^{p,q} = \Ext^{q}_{\Lambda_{\mu,j}}(H_{p}(W;\Lambda_{\mu,j})^{\tr},\Lambda_{\mu,j}) \Rightarrow H^{p+q}(W;\Lambda_{\mu,j})\]
  which implies
  \[H^{0}(W;\Lambda_{\mu,j}) \cong \hom_{\Lambda_{\mu,j}}(H_{0}(W;\Lambda_{\mu,j})^{\tr},\Lambda_{\mu,j}) = 0\,.\]
  From this spectral sequence, we obtain an exact sequence
  \[0 \to \Ext^{1}_{\Lambda_{\mu,j}}(H_{0}(W;\Lambda_{\mu,j})^{\tr},\Lambda_{\mu,j}) \to H^{1}(W;\Lambda_{\mu,j}) \to \hom_{\Lambda_{\mu,j}}\left(H_{1}(W;\Lambda_{\mu,j})^{\tr}, \Lambda_{\mu,j} \right).\]
  Using~\eqref{eq:vanishing-H0} and~\eqref{eq:vanishing-H1}, we deduce that \(H^{1}(W;\Lambda_{\mu,j}) = 0.\)
  The lemma now follows from Poincar\'e-Lefschetz duality.

\end{proof}

\begin{corollary}\label{cor:duality-H2}
  For~$W$ a~\(4\)-manifold as in Lemma~\ref{lemma:representing-int-forms-2}, there is an isomorphism of \(\Lambda_{\mu,j}\)-modules
  \begin{align*}
      H_{2}(W;\Lambda_{\mu,j}) &\cong \hom_{\Lambda_{\mu,j}}(H_{2}(W, \partial W;\Lambda_{\mu,j})^{\tr},\Lambda_{\mu,}), 
  \end{align*}
\end{corollary}
\begin{proof}
  Consider the UCSS
  \[E_{2}^{p,q} = \Ext^{q}_{\Lambda_{\mu,j}}\left( H_{p}(W, \partial W; \Lambda_{\mu,j})^{\tr}, \Lambda_{\mu,j} \right) \Rightarrow H^{p+q}(W, \partial W; \Lambda_{\mu,j})\,.\]
  From Lemma~\ref{lemma:computations-relative-homology}, we obtain that \(E_{2}^{p,q} = 0\) unless \(p=2\).
  In particular, we get
  \[H_{2}(W; \Lambda_{\mu,j}) \cong H^{2}(W, \partial W; \Lambda_{\mu,j}) \cong \hom_{\Lambda_{\mu,j}} \left( H_{2}(W, \partial W; \Lambda_{\mu,j})^{\tr}, \Lambda_{\mu,j} \right),\]
  where the first isomorphism comes from Poincar\'e-Lefschetz duality.
\end{proof}

\begin{proof}[Proof of Lemma~\ref{lemma:freeness-second-homology}]
  Consider the relative cellular chain complex \(C_{\ast} := C_{\ast}(W, \partial W; \Lambda_{\mu,j})\) of the pair \((W, \partial W)\),
  and let~\(Z_{2}\) denote the submodule of~\(2\)-cycles in~\(C_{2}\).
  Observe that by Lemma~\ref{lemma:computations-relative-homology}, we have \(H_{i}(C_{\ast}) = 0\) for \(i=0,1\), leading to the exact sequence
  \begin{equation}
    \label{eq:Z2-module}
    0 \to Z_{2} \xrightarrow{j} C_{2} \xrightarrow{\partial_{2}} C_{1} \xrightarrow{\partial_{1}} C_{0} \to 0\,.
  \end{equation}
  Since~$C_0$ is a free~\(\Lambda_{\mu,j}\)-module, this leads to an isomorphism~$C_1\simeq\ker\partial_1\oplus C_0$. Since~$C_1$ is free,
  it follows that~$\ker\partial_1$ is finitely generated and projective, hence free by Roitman's theorem, see Theorem~1.11 and Corollary~1.12 of~\cite[Chapter~V]{lamSerreProblemProjective2006}.
  Applying the same argument
  to the short exact sequence~$0 \to Z_{2} \to C_{2} \to \ker\partial_1 \to 0$, the fact that~$C_2$ and~$\ker\partial_1$ are free implies that~$Z_2$ is free as well.

  By Lemma~\ref{lemma:computations-relative-homology}, we
  also have~\(H_{i}(C_{\ast}) = 0\) for \(i=3,4\), yielding another exact sequence
  \[0 \to C_{4} \xrightarrow{\partial_{4}} C_{3} \xrightarrow{\partial_{3}} Z_{2} \xrightarrow{p} H_{2}(W,\partial W;\Lambda_{\mu,j}) \to 0.\]
  Consider the commutative diagram
  \begin{center}
    \begin{tikzcd}
      0 \arrow[r] & H_{2}(W;\partial W; \Lambda_{\mu,j})^{\bullet} \arrow[r,"p^{\bullet}"] & Z_{2}^{\bullet} \arrow[r,"\partial_{3}^{\bullet}"] & C_{3}^{\bullet} \arrow[r,"\partial_{4}^{\bullet}"] & C_{4}^{\bullet} \arrow[r] & 0 \\
      0 \arrow[r] & Z^2 \arrow[r] \arrow[u,"k"] & C_{2}^{\bullet} \arrow[r,"\partial_{3}^{\bullet}"] \arrow[u,"j^{\bullet}"] & C_{3}^{\bullet} \arrow[r,"\partial_{4}^{\bullet}"] \arrow[u,"="] & C_{4}^{\bullet} \arrow[r] \arrow[u,"="] & 0\,, 
    \end{tikzcd}
  \end{center}
  where~\(M^{\bullet}\) stands for~\(\hom_{\Lambda_{\mu,j}}(M^{\tr},\Lambda_{\mu,j})\) and~\(Z^{2} = \ker(\partial_{3}^{\bullet})\).
  By Lemma~\ref{lemma:computations-relative-homology}, the bottom row is exact, which implies exactness of the top row at~\(C_{4}^{\bullet}\)\and~\(C_{3}^{\bullet}\). Furthermore,
  left-exactness of the hom functor implies that~\(p^{\bullet}\)
  is injective. Also, since~\eqref{eq:Z2-module} is a split exact sequence, it follows that~\(j^{\bullet}\) is surjective.
  Now, a bit of diagram chasing shows exactness of the
  top row at~$Z_{2}^{\bullet}$, so the top row is exact.
  Since~\(C_{4}\),~\(C_{3}\) and~\(Z_{2}\) are free, so are~\(C_{4}^{\bullet}\),~\(C_{3}^{\bullet}\) and~\(Z_{2}^{\bullet}\).
  As above, Roitman's theorem now implies that~\(H_{2}(W, \partial W; \Lambda_{\mu,j})^{\bullet}\) is free as well.
  By Corollary~\ref{cor:duality-H2}, it is isomorphic to~$H_{2}(W;\Lambda_{\mu,j})$, which concludes the proof.
\end{proof}

\subsection{Proof of Lemma~\ref{lemma:representing-int-forms-2}}
\label{sec:proof-lemma-homology-computations}

Let us first assume~$\mu>1$ and fix~\(j=1,2,\ldots,\mu\).
By Lemma~\ref{lemma:freeness-second-homology}, we know that~\(H_{2}(W;\Lambda_{\mu,j})\) is a free \(\Lambda_{\mu,j}\)-module.
Hence, the twisted intersection form
\[Q_{j}(W) \colon H_{2}(W;\Lambda_{\mu,j}) \times H_{2}(W;\Lambda_{\mu,j}) \to \Lambda_{\mu,j}\]
can be represented by some matrix~\(H_{j}\).
For any \(\omega \in U_{j}:= \{\omega \in \mathbb{T}^{\mu} \colon \omega_{j} \neq 1\}\), observe that the action of~\(\Lambda_{\mu}\) on~\(\C^{\omega}\) extends to an action of \(\Lambda_{\mu,j}\). In other words, we have a natural homomorphism~$\phi_\omega\colon\Lambda_{\mu,j}\to\C^\omega$
given by~$t_i\mapsto\omega_i$.
Consider the first part of Theorem~\ref{thm:UCSS}
  applied to~$M=\C^\omega$ seen as a module over~$R=\Lambda_{\mu,j}$ via~$\phi_\omega$, and to~$C_*=C_*(\widetilde{W};\Lambda_{\mu,j})$: it yields the spectral sequence
\[E^{2}_{p,q} = \Tor_{p}^{\Lambda_{\mu,j}}(\C^{\omega}, H_{q}(W;\Lambda_{\mu,j})) \Rightarrow H_{p+q}(W;\C^{\omega}).\]
By Lemma~\ref{lemma:computations-relative-homology}, we have \(H_{i}(W;\Lambda_{\mu,j}) = 0\) unless \(i=2\) and by Lemma~\ref{lemma:freeness-second-homology}, \(H_{2}(W;\Lambda_{\mu,j})\) is free.
Consequently, \(E_{2}^{p,q} = 0\) unless \(p=0\) and \(q=2\).
Therefore, \(H_{i}(W;\C^{\omega}) = 0\) unless \(i=2\) and
\[H_{2}(W;\C^{\omega}) \cong \C^{\omega} \otimes_{\Lambda_{\mu,j}} H_{2}(W;\Lambda_{\mu,j})\,.\]
In particular, we have
\[
\rank_{\C} H_{2}(W;\C^{\omega}) = \rank_{\Lambda_{\mu,j}} H_{2}(W;\Lambda_{\mu,j})\,.
\]
Lemma~\ref{lemma:naturality-intersection-forms} now implies that
for any~$\omega\in U_j$, the intersection form~$Q_\omega(W)$ can be
represented by the matrix~$H_j(\omega)$
obtained by evaluating~$H_j$ at~$(t_1,\dots,t_\mu)=\omega$.


We now turn to the case~$\mu=1$. First note that our assumptions together with the exact sequence of the pair~$(W,\partial W)$ imply
that~$H_1(W,\partial W;\Lambda_1)$ vanishes.
Since~\(\Lambda_{1}\) is a PID, the Universal Coefficient Theorem then yields
\[
H^{2}(W,\partial W;\Lambda_{1}) \cong \hom_{\Lambda_{1}}(H_{2}(W,\partial W;\Lambda_{1}),\Lambda_{1})\,.
\]
This shows that~\(H^{2}(W,\partial W;\Lambda_{1})\) is torsion free, hence free. By Poincar\'e-Lefschetz duality, the same holds for~$H_{2}(W;\Lambda_{1})$.
Since~$H_1(W;\Lambda_1)$ vanishes by assumption, one more application of the
Universal Coefficient Theorem yields
\[
    H_{2}(W;\C^{\omega}) \cong \C^{\omega} \otimes_{\Lambda_{1}} H_{2}(W;\Lambda_1)\,.
\]
The statement now follows from the naturality of the twisted intersection form as stated in Remark~\ref{rem:naturality}. \qed
  
\bibliographystyle{plain}
\bibliography{bibliography}

\end{document}